\documentclass[twoside]{amsart}

\usepackage{amsmath}
\usepackage[english]{babel}
\usepackage{amsfonts}
\usepackage{amssymb}
\usepackage{bbm}

\usepackage{hyperref}
\usepackage{xcolor}
\hypersetup{
    colorlinks,
    linkcolor={red!50!black},
    citecolor={blue!50!black},
    urlcolor={blue!80!black}
}

\usepackage{graphicx}
\usepackage[usenames,dvipsnames]{pstricks}
\usepackage{epsfig}
\usepackage{pst-grad} % For gradients
\usepackage{pst-plot} % For axes
\usepackage{pst-tree}
\usepackage{pst-math}

\usepackage{mathtools} %scommentare per far numerare solo le equationi che sono chiamate nel testo
\mathtoolsset{showonlyrefs,showmanualtags} %scommentare per far numerare solo le equationi che sono chiamate nel testo

\usepackage{color}

\newcommand{\distr}{\mathcal{D}}				% distribuzione
\newcommand{\metr}{\langle\cdot|\cdot\rangle}	% metrica
\newcommand{\la}{\lambda}						% lambda per i pigri
\newcommand{\al}{\alpha}						% alpha per i pigri
\newcommand{\g}{\gamma}							% gamma per i pigri
\newcommand{\I}{\mathbbm{1}}					% identita'
\newcommand{\R}{\mathbb{R}}						% reals
\newcommand{\N}{\mathbb{N}}						% naturals
\newcommand{\so}{\mathfrak{so}}					% ortogonal algebra
\newcommand{\Z}{\mathbb{Z}}						% integers
\newcommand{\U}{\mathrm{U}}						% unitary group
\newcommand{\Sp}{\mathrm{Sp}}					% symplectic group
\renewcommand{\epsilon}{\varepsilon}
\newcommand{\be}{\begin{equation}}
\newcommand{\ee}{\end{equation}}

\newtheorem{theorem}{Theorem}
\newtheorem{lemma}[theorem]{Lemma}

\newtheorem{proposition}[theorem]{Proposition}

\theoremstyle{remark} 
\newtheorem{remark}{Remark}
\newtheorem{example}{Example}

\theoremstyle{definition} %% COSI DEF E REMARK NON SONO SCRITTI IN ITALIC
\newtheorem{definition}[theorem]{Definition}

\DeclareMathOperator{\diag}{\mathrm{diag}}

\DeclareMathOperator{\spn}{\mathrm{span}}
\DeclareMathOperator{\rank}{\mathrm{rank}}
\DeclareMathOperator{\ISO}{\mathrm{ISO}}
\DeclareMathOperator*{\cart}{\times}

\title{How many geodesics join two points on a contact sub-Riemannian manifold?}
\author{A. Lerario}
\address{Scuola Internazionale Superiore di Studi Avanzati, via Bonomea 265, 34136 Trieste, Italy} 
\email{lerario@sissa.it}

\author{L. Rizzi}
\address{Univ. Grenoble Alpes, CNRS, Institut Fourier, F-38000 Grenoble, France \newline \emph{Former institution:} CNRS, CMAP \'Ecole Polytechnique and \'Equipe INRIA GECO Saclay \^Ile-de-France, Paris, France}
\email{luca.rizzi@univ-grenoble-alpes.fr}

\begin{document}

%!TEX root = enumerative-v5.tex

\begin{abstract}
We investigate the structure and the topology of the set of geodesics (critical points for the \emph{energy functional}) between two points on a contact Carnot group $G$ (or, more generally, corank-one Carnot groups). Denoting by $(x,z)\in \R^{2n}\times \R$ exponential coordinates on $G$, we find constants $C_1, C_2>0$ and $R_1, R_2$ such that the number $\hat{\nu}(p)$ of geodesics joining the origin with a generic point $p=(x,z)$ satisfies:
\begin{equation}\label{eq:boundabstract}
C_1\frac{|z|\phantom{^2}}{\|x\|^2}+R_1\leq \hat{\nu}(p)\leq C_2\frac{|z|\phantom{^2}}{\|x\|^2}+R_2.
\end{equation}
We give conditions for $p$ to be joined by a unique geodesic and we specialize our computations to standard Heisenberg groups, where $C_1=C_2=\frac{8}{\pi}$.

The set of geodesics joining the origin with $p\neq p_0$, parametrized with their initial covector, is a topological space $\Gamma(p)$, that naturally splits as the disjoint union
\begin{equation}
\Gamma(p) = \Gamma_0(p) \cup \Gamma_\infty(p),
\end{equation}
where $\Gamma_0(p)$ is a finite set of isolated geodesics, while $\Gamma_\infty(p)$ contains continuous families of non-isolated geodesics (critical \emph{manifolds} for the energy). We prove an estimate similar to~\eqref{eq:boundabstract} for the ``topology'' (i.e. the total Betti number) of $\Gamma(p)$, with no restriction on $p$.

When $G$ is the Heisenberg group, families appear if and only if $p$ is a \emph{vertical} nonzero point and each family is generated by the action of isometries on a given geodesic. Surprisingly, in more general cases, families of \emph{non-isometrically equivalent} geodesics do appear.

If the Carnot group $G$ is the \emph{nilpotent approximation} of a contact sub-Riemannian manifold $M$ at a point $p_0$, we prove that the number $\nu(p)$ of geodesics in $M$ joining $p_0$ with $p$ can be estimated from below with $\hat{\nu}(p)$. The number $\nu(p)$ estimates indeed geodesics whose image is contained in a coordinate chart around $p_0$ (we call these ``local'' geodesics).

As a corollary we prove the existence of a sequence $\{p_n\}_{n\in \mathbb{N}}$ in $M$ such that:
\begin{equation}
\lim_{n\to \infty}p_n=p_0\qquad \text{and}\qquad \lim_{n\to \infty}\nu(p_n)=\infty,
\end{equation}
i.e. the number of ``local'' geodesics between two points can be arbitrarily large, in sharp contrast with the Riemannian case.
\end{abstract}

\maketitle

%!TEX root = enumerative-v5.tex

\section{Introduction}\label{s:intro}

If the topology of a Riemannian manifold $M$ is ``complicated enough'' (for example if $M$ is closed) a well known theorem of J-P. Serre \cite{serre} states that there are infinitely many geodesics\footnote{In the spirit of Morse theory, we define Riemannian geodesics as locally length minimizing curves parametrized by constant speed or, equivalently, critical points for the energy functional.} between any two points in $M$. These geodesics have the property of being ``global'', in the sense that their existence is guaranteed by the global topology of the manifold.

At the opposite extreme, if the manifold $M$ is a convex neighbourhood of a point in a Riemannian manifold, the structure of geodesics resembles the Euclidean one, and between any two points there is only one geodesic.

In the contact sub-Riemannian case the \emph{global} picture is the same as of the Riemannian case. The study of geodesics that ``loop'' in the topology of the manifold was recently done by the first author and F. Boarotto in \cite{boalarry}: every two points on a compact sub-Riemannian contact manifold are joined by infinitely many geodesics (the result uses a weak homotopy equivalence between the space of all curves and the space of horizontal ones). %Thus in a sense the \emph{global} picture in the sub-Riemannian case is the same as the Riemannian.
On the opposite, our main interest will be in the set of ``local'' geodesics, i.e. geodesics between two points whose image is contained in a coordinate chart: here the sub-Riemannian picture is dramatically different.
To mention one example, the only geodesically convex neighborhood of the origin in the Heisenberg group (see below) is the entire group, \cite{Monti}.

In this framework we consider a constant-rank distribution $\distr \subset TM$ with the property that iterated brackets of vector fields on $\distr$ generate the tangent space (H\"ormander condition). This condition guarantees that any two points in $M$ can be joined by a Lipschitz continuous curve whose velocity is a.e. in $\distr$ (Chow-Rashevskii theorem).

If a smooth scalar product is defined on $\distr$, it makes sense to consider, for any horizontal curve $\gamma$, the norm of its velocity and the \emph{energy} of this curve is defined by:
\begin{equation} 
J(\gamma)=\frac{1}{2}\int_{I}\|\dot{\gamma}(t)\|^2dt.
\end{equation}
Sub-Riemannian \emph{geodesics} between $p_0$ and $p$ are critical points of $J$ \emph{constrained} to have endpoints $p_0$ and $p$. From now on the word geodesic will always mean sub-Riemannian geodesic.

\begin{example}[Heisenberg]
The \emph{Heisenberg group} $\mathbb{H}_3$ is the smooth manifold $\R^3$ with coordinates $(x_1,x_2,z)$ and the distribution:
\begin{equation}\label{eq:deltah}
\distr=\spn\left\{\frac{\partial}{\partial x_1}-\frac{x_2}{2}\frac{\partial}{\partial z}, \frac{\partial}{\partial x_2}+\frac{x_1}{2}\frac{\partial}{\partial z}\right\}.
\end{equation}
The sub-Riemannian structure is given by declaring the above vector fields an orthonormal basis.

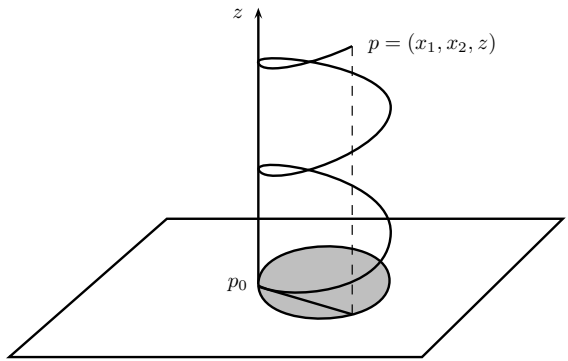
\begin{figure}
\centering
\scalebox{0.8} % Change this value to rescale the drawing.
{
\begin{pspicture}(1,-3)(10.2,3.1091993)
\pspolygon[linewidth=0.04](1.0,-2.8)(3.62,-0.5)(10.2,-0.5)(7.86,-2.8)
\psline[linewidth=0.04cm,arrowsize=0.05291667cm 2.0,arrowlength=1.4,arrowinset=0.4]{->}(5.14,-1.5891992)(5.14,3.0108008)
\psbezier[fillstyle=solid,fillcolor=lightgray,linewidth=0.04](5.14,-1.5891992)(5.14,-2.3891993)(7.32,-2.3291993)(7.32,-1.5291992)(7.32,-0.72919923)(5.14,-0.78919923)(5.14,-1.5891992)
\psbezier[linewidth=0.04](5.14,-1.6091992)(5.58,-1.8691993)(7.301063,-1.7484409)(7.34,-0.7491992)(7.378937,0.25004244)(4.92,0.59080076)(5.16,0.29080078)(5.4,-0.009199219)(7.3390436,0.63254535)(7.34,1.3508008)(7.3409567,2.0690563)(5.02,2.3108008)(5.14,2.0908008)(5.22,1.8708007)(6.12,2.0908008)(6.7,2.3708007)
\psline[linewidth=0.02cm,linestyle=dashed,dash=0.16cm 0.16cm](6.7,2.3708007)(6.7,-2.0891993)
\psline[linewidth=0.04cm](5.14,-1.6291993)(6.7,-2.0891993)
\rput(4.8,2.9158008){$z$}
\usefont{T1}{ptm}{m}{n}
\rput(8.031455,2.4158008){$p=(x_1, x_2, z)$}
\rput[c](4.8,-1.5891992){$p_0$}
\end{pspicture} 
}
\caption{Geodesics in the Heisenberg group.}\label{fig:heis1}
\end{figure}
Let $p_0=(0,0,0)$ be the origin. Geodesics are curves whose projection on the $(x_1,x_2)$-plane is an arc of a circle (possibly with infinite radius, i.e. a segment on a straight line); the signed area swept out on the circle equals the $z$-coordinate of the final point $p$.

If $p$ belongs to the $(x_1,x_2)$-plane there is only one geodesic joining it with the origin (this is precisely the segment trough $p_0$ and $p$); if $p$ has both nonzero components in the $(x_1,x_2)$ plane and the $z$ axis, the number of geodesics is finite; finally, if $p$ belongs to the $z$-axis there are infinitely many geodesics.
In the latter case, given one geodesic, we obtain infinitely many others (a \emph{continuous family}) by composing it with a rotation around the $z$-axis (see Fig. \ref{fig:heis1}).
\end{example}

In the general sub-Riemannian case a Sard's like argument for the sub-Riemannian exponential map guarantees that for the generic choice of the two points geodesics are isolated, but finiteness is  more delicate. 
The following fact is proved in \cite[Prop. 7]{AGL}, but to the authors' knowledge the general question seems to be an open problem.

\begin{proposition}
Let $M$ be a step-two Carnot group such that $\emph{\textrm{rk}}(\distr)>\frac{1}{2}\dim(M)$. Then for the generic choice of $p_0$ and $p$ the number of geodesics between them is finite.
\end{proposition}

The goal of this paper is to make the above picture quantitative, at least in the case of contact\footnote{We stress here that all our results remain true with almost no modification for more general \emph{corank-one} sub-Riemannian structure. For simplicity we restrict our exposition to the contact case.} sub-Riemannian manifolds, addressing the following question:

\medskip
\begin{center}
\emph{``How many geodesics join two points on a contact sub-Riemannian manifold?''}
\end{center}
\medskip

%\begin{equation*}
%\emph{ \text{``How many geodesics join two points on a contact sub-Riemannian manifold?''}}
%\end{equation*}
A contact sub-Riemannian manifold is the simplest example of nonholonomic geometry. From the point of view of differential geometry it consists of a $(2n+1)$-dimensional, connected manifold $M$ together with a distribution $\distr\subset TM$ of hyperplanes locally defined as the kernel of a one-form $\alpha$ (the \emph{contact form}) such that the restriction $d\alpha|_{\distr}$ is non-degenerate. The sub-Riemannian structure is given by assigning a smooth metric on the hyperplane distribution.
The non-degeneracy condition implies H\"ormander's condition.

\begin{example}[Heisenberg, continuation] The Heisenberg group is a contact manifold with contact form $\alpha=-dz+\frac{1}{2}\left(x_1dx_2-x_2dx_1\right)$. As we will show later:
\begin{equation}\label{heisbound}
\#\{\textrm{geodesics between the origin and $p=(x_1, x_2, z)$}\}=\frac{8}{\pi}\frac{|z|\phantom{^2}}{\|x\|^2}+O(1). 
\end{equation}
In particular when $p$ is ``vertical'', $p=(0,0,z)$ the number of geodesics is infinite; otherwise it is finite and equals the r.h.s. (the $O(1)$ notation means ``up to a bounded error'').
\end{example}

For any point $p_0 \in M$ one can consider the so-called \emph{nilpotent approximation} of the sub-Riemannian structure at $p_0$. The result of this construction (that depends only on the germ of the structure at $p_0$) is a sub-Riemannian manifold $G_{p_0}$, and is an example of a \emph{Carnot group}. 

Thm.~\ref{t:limitintro} states that the geodesic count on the Carnot group $G_{p_0}$ controls the geodesic count on the original manifold $M$. For this reason, we start our analysis with the study of \emph{contact Carnot groups}, namely Carnot groups arising as the nilpotent approximation of contact manifolds.

\subsection{Contact Carnot groups}

A contact Carnot group is a connected, simply connected Lie group $G$, with $\dim G = 2n+1$, such that its Lie algebra $\mathfrak{g}$ of left-invariant vector fields admits a nilpotent stratification of step $2$, namely:
\begin{equation}
\mathfrak{g} = \mathfrak{g}_1 \oplus \mathfrak{g}_2, \qquad \mathfrak{g}_1,\mathfrak{g}_2 \neq \{0\},
\end{equation}
where $\dim\mathfrak{g}_2 = 1$ and
\begin{equation}
[\mathfrak{g}_1,\mathfrak{g}_1] = \mathfrak{g}_{2} \qquad \text{and}\qquad [\mathfrak{g}_1,\mathfrak{g}_2]=[\mathfrak{g}_2,\mathfrak{g}_2]=\{0\}.
\end{equation}
A scalar product is defined on $\mathfrak{g}_1$, by declaring a set $f_1,\ldots,f_{2n} \in \mathfrak{g}_1$ to be a global orthonormal frame. The group exponential map:
\begin{equation}
\mathrm{exp}_{G} : \mathfrak{g} \to G,
\end{equation}
associates with $v \in \mathfrak{g}$ the element $\gamma(1)$, where $\gamma: [0,1] \to G$ is the unique integral line of the vector field defined by $v$ such that $\gamma(0) = 0$. Since $G$ is simply connected and $\mathfrak{g}$ is nilpotent, $\mathrm{exp}_G$ is a smooth diffeomorphism.
The choice of an orthonormal frame $f_1,\ldots,f_{2n} \in \mathfrak{g}_1$ and $f_0 \in \mathfrak{g}_2$ defines \emph{exponential coordinates} $(x,z) \in \R^{2n}\times \R$ on $G$ such that $p =(x,z)$ if and only if
\begin{equation}
p = \mathrm{exp}_{G} \left(\sum_{i=1}^{2n} x_i f_i + z f_0\right).
\end{equation}
For any such a choice there exists a skew-symmetric matrix $A \in \so(2n)$ such that
\begin{equation}
[f_i,f_j] = A_{ij} f_0.
\end{equation}
For contact Carnot groups $A$ is non-degenerate. We denote by:
\begin{equation}
\al_1< \cdots <\al_k\in \mathbb{R}_+
\end{equation}
the distinct singular values of $A$ and $n_j$ their multiplicities. Let $x_j \in \R^{2n_j}$ be the projections of $x$ on the invariant subspaces associated with $\al_j$. Accordingly we write $p=(x_1,\ldots,x_k,z)$. 
\begin{example}\label{ex:heisenberg}
A classical example is the $(2n+1)$-dimensional Heisenberg group $\mathbb{H}_{2n+1}$. This is the case with $k=1$, i.e. a unique singular value $\alpha_1 = 1$ with multiplicity $n$. In this case, for $i=1,\ldots,n$
\begin{equation}
f_{i} := \frac{\partial}{\partial x_i} - \frac{1}{2}x_{i+n} \frac{\partial}{\partial z}, \qquad f_{n+i} := \frac{\partial}{\partial x_{n+i}} + \frac{1}{2}x_{i} \frac{\partial}{\partial z}, \qquad
f_0 := \frac{\partial}{\partial z},
\end{equation}
and $A$ is the standard symplectic matrix $J = \left(\begin{smallmatrix} 0 & \I_n \\ -\I_n & 0 \end{smallmatrix}\right)$.
\end{example}

The geodesic count for $G$ can be made quite explicit in term of the exponential coordinates of $p$ and the singular values of the matrix $A$. Define for this purpose the ``counting'' function:
\begin{equation}
\hat{\nu}(p)=\#\{\textrm{geodesics in a Carnot group between the origin and $p$}\},
\end{equation}
where, by convention, the ``hat'' stresses the fact that we refer to a Carnot group. We have the following estimates for $\hat{\nu}(p)$ (see Thms.~\ref{thm:upper}--\ref{thm:lower}). None of these bounds is trivial: the upper bound because the exponential map is not proper; the lower bound is in fact even more surprising, as the typical finiteness techniques from semialgebraic (semianalytic) geometry only produce upper bounds (we use indeed a kind of ``ergodicity'' property argument).
\begin{theorem}[The ``infinitesimal'' bound]\label{thm:orderintro}
Given a contact Carnot group $G$, there exist constants $C_1,C_2 >0$ and $R_1, R_2$ such that if $p=(x,z)\in G$ is a point with all components $x_j$ different from zero, then:
\begin{equation}\label{eq:orderintro} 
C_1\frac{|z|\phantom{^2}}{\|x\|^2}+R_1\leq \hat{\nu}(p)\leq C_2\frac{|z|\phantom{^2}}{\|x\|^2}+R_2.
\end{equation}
\end{theorem}
In fact $C_1,C_2$ (resp. $R_1,R_2$) are homogeneous of degree $-1$ (resp. $0$) in the singular values $\alpha_1<\cdots<\alpha_k$ of $A$ and are given by:
\begin{equation}\label{eq:numerical}
C_1=\frac{8}{\pi}\frac{\alpha_1}{\alpha_k^2} \sin\left(\frac{\delta \pi}{2}\right)^2\qquad\text{with}\qquad \delta = \left(\sum_{j=1}^k \frac{\alpha_1}{\alpha_j} \left\lfloor \frac{\alpha_j}{\alpha_1}\right\rfloor\right)^{-1} \qquad \text{and}\qquad C_2=\frac{8k}{\pi} \frac{\al_k}{\al_1^2}.
\end{equation}

\begin{remark}
For any other choice of $f_1',\ldots,f_{2n}' \in \mathfrak{g}_1$ (orthonormal) and a complement $f_0' \in \mathfrak{g}_2$ there exists a matrix $M \in \mathrm{O}(2n)$ and a constant $c$ such that:
\begin{equation}
f_i = \sum_{j=1}^{2n} M_{ij} f_j', \qquad f_0 = c f_0'.
\end{equation}
Indeed this new choice defines a new skew-symmetric matrix $A'$ and also new exponential coordinates $(x',z')$. One can easily check that: 
\begin{equation}
A' = cM^*AM, \qquad x' = M^*x, \qquad z' = c z.
\end{equation}
Since $C_1, C_2$ are homogeneous functions of degree $-1$ in the singular values of $A$, the upper and lower bounds~\eqref{eq:orderintro} are invariant w.r.t. different choices of exponential coordinates.
\end{remark}

\begin{example}[Heisenberg, continuation] In the Heisenberg group $\mathbb{H}_{2n+1}$ there is only one singular value $\al =1$, with multiplicity $n$. By using \eqref{eq:orderintro} and \eqref{eq:numerical} one obtains:
\begin{equation}
C_1=C_2=\frac{8}{\pi},
\end{equation}
recovering~\eqref{heisbound} (that holds true for any Heisenberg group, not just the three-dimensional one).
\end{example}

%\subsection{An exact count}
%The estimates of Thm.~\ref{thm:orderintro} will play a key role later, when we relate the geodesic count on a contact manifold $M$ with the geodesic count on its nilpotent approximation. Nevertheless, for contact Carnot groups one can provide an exact expression for the function $\hat{\nu}(p)$. Consider the function:
%\begin{equation}\label{eq:gintro}
%g(\la)=\frac{1}{8}\frac{\la-\sin \la}{\left(\sin \frac{\la}{2}\right)^2};
%\end{equation}
%if all the $x_j$ components of the $x$ coordinate of $p$ are non-zero, then:
%\begin{equation}\label{explintro}
%\hat{\nu}(p)=\#\left\{\la\,\bigg|\,z=\sum_{j=1}^k\al_j \|x_j\|^{2}g(\la \al_j)\right\}.
%\end{equation} 

An interesting related question is to determine the set of points $p$ such that $\hat{\nu}(p)=1$ (as it happens for example if $p = (x,0)$, i.e. $p$ is horizontal). In the Heisenberg group:
\begin{equation}
\hat{\nu}(p)=1\iff \frac{|z|\phantom{^2}}{\|x\|^2}\leq \frac{\la_1}{4}\approx 1.12335,
\end{equation}
where $\la_1$ is the first positive solution of $\tan \la=\la$; in the general case we have the following.

\begin{proposition}\label{p:onei}
Let $G$ be a contact Carnot group and $p=(x,z)$ such that:
\begin{equation}\label{eq:onei}
|z|<\frac{\pi}{8}\left(\frac{2\al_1^2}{\al_k}-\al_k\right)\|x\|^2.
\end{equation}
Then there is only one geodesic from $p_0$ to $p$.
\end{proposition}

\subsection{Critical manifolds}
It is interesting to discuss the structure of \emph{all} geodesics ending at $p$, including the case when the point $p$ belongs to a hyperplane coordinate space (i.e. $x_j=0$ for some $j$), which was excluded from Thm. \ref{thm:orderintro}. We still exclude the case $p=p_0$, as for the case of Carnot groups there is only one geodesic: the trivial one $\gamma(t)\equiv p_0$.

Sub-Riemannian geodesics starting from $p_0$ are parametrized by their initial covector $\eta \in T^*_{p_0}M$. The subset $\Gamma(p)$ of geodesics ending at $p$ has the subset topology from $T_{p_0}^*M$. We have the following characterization (see Thm.~\ref{thm:disjoint}).

\begin{theorem}[Topology of critical manifolds]\label{t:topologyintro}Let $G$ be a contact Carnot group.
The set $\Gamma(p)$ of geodesics ending at $p\neq p_0$ can be decomposed into the disjoint union of two closed submanifolds:
\begin{equation}
\Gamma(p)=\Gamma_0(p)\cup\Gamma_\infty(p).
\end{equation}
The set $\Gamma_0(p)$ is finite and the set $\Gamma_\infty(p)$ is homeomorphic to a union of spheres.
Moreover the energy function $J$ is constant on each component of $\Gamma(p)$.
\end{theorem}

\begin{remark}The structure of the sets of geodesics whose final point is vertical, in the general step-two Carnot group of type $(k,n)$ is studied in \cite{AGL}. 
Geodesics to $p$ are critical points for the energy functional $J:\Omega_p\to \R$ (here $\Omega_p$ is the space of all admissible curves to $p$ and $J$ is defined as above); for the generic vertical $p$ these geodesics appear in families, which are tori of finite dimension depending on the ``multiplicity'' of the Lagrange multiplier (in particular they are never isolated and $J$ is a Morse-Bott function).
A Morse theoretical study proves that:
\begin{equation}
\#\{\text{critical manifolds of $J$ with energy less then $c$}\} \leq O(c^{n-k}).
\end{equation}
On the other hand the ``order of growth'' of the topology of $\Omega_p^c=\{\gamma\in \Omega_p\,|\, J(\gamma)\leq c\}$ (the sublevel set of the energy) is given by (here $b(X)$ denotes the \emph{total Betti number} of $X$):
\begin{equation}
b(\Omega_p^c)\leq O(c^{n-k-1}),
\end{equation}
an inequality which is stronger than the classical Morse-Bott prediction $b(\Omega_p^c)\leq O(c^{n-k})$.
\end{remark}

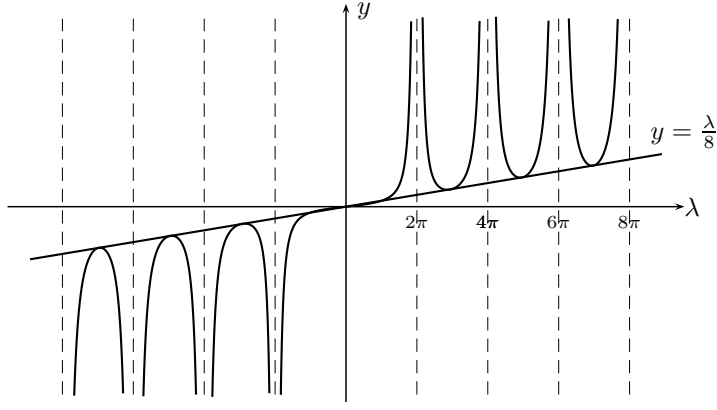
\begin{figure}[t]
\psset{algebraic,plotpoints=10000,plotstyle=line,xunit=1.5mm,yunit=1mm}
{\begin{pspicture}(-30,-26)(30,27)
\psline[linewidth=0.02cm]{->}(-30,0)(30,0)
\psline[linewidth=0.02cm]{->}(0,-26)(0,27)
\psline[linewidth=0.01cm,linestyle=dashed](6.28319,-25)(6.28319,25)
\psline[linewidth=0.01cm,linestyle=dashed](12.5664,-25)(12.5664,25)
\psline[linewidth=0.01cm,linestyle=dashed](18.8496,-25)(18.8496,25)
\psline[linewidth=0.01cm,linestyle=dashed](25.1327,-25)(25.1327,25)
\psline[linewidth=0.01cm,linestyle=dashed](-6.28319,-25)(-6.28319,25)
\psline[linewidth=0.01cm,linestyle=dashed](-12.5664,-25)(-12.5664,25)
\psline[linewidth=0.01cm,linestyle=dashed](-18.8496,-25)(-18.8496,25)
\psline[linewidth=0.01cm,linestyle=dashed](-25.1327,-25)(-25.1327,25)
\psplot[yMaxValue=25,yMinValue=-25]{-25}{25}{2/8*(x-sin(x))/(sin(x/2))^2}
\psplot[yMaxValue=25,yMinValue=-25]{-28}{28}{2/8*x}
\rput[l](27,10){$y=\frac{\lambda}{8}$}
\rput[l](30,0){$\lambda$}
\rput[l](1,26){$y$}
\rput[c](6.28,-2){\scriptsize$2\pi$}
\rput[c](12.57,-2){\scriptsize$4\pi$}
\rput[c](12.57,-2){\scriptsize$4\pi$}
\rput[c](18.85,-2){\scriptsize$6\pi$}
\rput[c](25.13,-2){\scriptsize$8\pi$}
\end{pspicture}}
\caption{The graph of $g$.}
\label{fig:g}
\end{figure}

Since geodesics in $\Gamma_0(p)$ are always finite, the preimage of a regular value of $\hat{E}$ is finite. Geodesics in $\Gamma_\infty(p)$ appear in \emph{families}. Since geodesics are critical points for the energy functional, we call each connected component of $\Gamma_\infty(p)$ a \emph{critical family} (or \emph{critical manifold}). The set $\Gamma_\infty(p)$ has the following description. Given $\alpha_1, \ldots, \alpha_k$ (the singular values of $A$) define:
\begin{equation}\label{eq:gintro}
g(\lambda)=\frac{1}{8}\frac{\lambda-\sin\lambda}{\left(\sin\frac{\lambda}{2}\right)^2},
\end{equation}
and the sets:
\begin{equation}
\Lambda_j=\frac{2\pi}{\al_j}\mathbb{Z}\setminus\{0\}, \qquad \Lambda=\bigcup_{j=1}^k\Lambda_j\qquad  \text{and}\qquad L(\la)=\{j\,|\, \la \in \Lambda_j\}.
\end{equation}
Thus $\Lambda_j$ consists of the poles of $\lambda \mapsto g(\la \al_j)$ and the set of indices $L(\la)$ tells how many of these poles occur at $\la$ (see Fig.~\ref{fig:g}). With these conventions we have:
\begin{equation}
\Gamma_\infty(p) \simeq \bigcup_{\lambda \in \Lambda_p} S^{2N(\lambda)-1}, \qquad N(\lambda)=\sum_{j \in L(\lambda)} n_j,
\end{equation} 
where $n_j$ is the multiplicity of the singular value $\alpha_j$ and
\begin{equation}
\Lambda_p=\left\{\la \in \Lambda \,\bigg|\,\left(z-\sum_{x_j \neq 0}\al_jg(\la \al_j)\|x_j\|^2\right) \la >0\right\}.
\end{equation}

For the generic $A$ all singular values are distinct ($k=n$) and non-com\-mensurable, thus for every $\la \in \Lambda_p$ we have $\#L(\la)=1$, $N(\la)=1$ and all critical manifolds are homeomorphic to circles. If some of the singular values have multiplicities greater than one, but still are all pairwise non-com\-mensurable, $\#L(\lambda) =1$ but we can have critical manifolds of various dimensions.

As we will see, $\Gamma_\infty(p)$ is not empty only if some of the coordinates $x_j$ vanish. If $\Gamma_\infty(p)$ is not empty, each critical manifold is homeomorphic to a sphere; here the estimate~\eqref{eq:orderintro} can be extended to \emph{all} points $p\neq p_0$ if one adopts a ``topological'' viewpoint. Denoting by:
\begin{equation}
\hat{\beta}(p)=\{\textrm{sum of the Betti numbers of the set of geodesics from the origin to $p$}\},
\end{equation}
we have the following generalization of Thm.~\ref{thm:orderintro}  which bounds the number of spheres in $\Gamma_\infty(p)$ (see Thms.~\ref{thm:upper}--\ref{thm:lowertop}).
\begin{theorem}[The ``infinitesimal'' bound for the topology]\label{thm:ordertopintro}Let $G$ be a contact Carnot group.
There exist constants $C_1', C_2'>0$  and $R_1', R_2'$ such that for every $p=(x,z)\in G$, \text{with} $p\neq (0,0)$:
\begin{equation}\label{eq:topintro} 
C_1'\frac{|z|\phantom{^2}}{\|x\|^2}+R_1'\leq \hat{\beta}(p)\leq C_2'\frac{|z|\phantom{^2}}{\|x\|^2}+R_2'.
\end{equation}
\end{theorem}

As above, $C_1',C_2'$ (resp. $R_1',R_2'$) are homogeneous of degree $-1$ (resp. $0$) in the singular values $\alpha_1<\cdots<\alpha_k$ of $A$ and are given by:
\begin{equation}
C_1'=\frac{8}{\pi} \frac{\alpha_1}{\alpha_k^2}\sin\left(\frac{\delta' \pi}{2}\right)^2\qquad\text{with}\qquad \delta' = \left(\sum_{x_j\neq 0} \frac{\alpha_1}{\alpha_j} \left\lfloor \frac{\alpha_j}{\alpha_1}\right\rfloor\right)^{-1} \qquad \text{and}\qquad C_2'=\frac{8k}{\pi} \frac{\al_k}{\al_1^2};
\end{equation}
and in particular again these upper bounds are invariant w.r.t. change of exponential coordinates.

Fig. \ref{comparison} compares the contribution to $\hat{\nu}$ and $\hat{\beta}$ coming respectively from $\Gamma_0$ and $\Gamma_\infty$. In some sense, $\hat{\beta}(p)$ counts the geodesics ``up to families''. Thus if $x \neq 0$ then geodesics might appear in families, but still the topology of these families is controlled, in particular the number of disjoint families is bounded. 

\begin{figure}[t]
\centering
\def\arraystretch{2.2}
\begin{tabular}{|c|c|c|c|}
\cline{2-4}
\multicolumn{1}{c|}{}&$\#\Gamma_0$ & $\#\Gamma_\infty$ & $\hat\nu$ \\ \hline
all $x_j\neq0$ & $\frac{|z|}{\,\,\|x\|^2}$ & $0$ & $\frac{|z|\phantom{^2}}{\|x\|^2}$\\ \hline
some $x_j=0$  & $\frac{|z|\phantom{^2}}{\|x\|^2}$ & $\infty$ & $\infty$ \\ \hline
$x=0$ & $0$ & $\infty$ & $\infty$\\ \hline
\end{tabular}\qquad\qquad
\begin{tabular}{|c|c|c|c|}
\cline{2-4}
\multicolumn{1}{c|}{}&$b(\Gamma_0)$ & $b(\Gamma_\infty)$ & $\hat\beta$ \\ \hline
all $x_j\neq0$ & $\frac{|z|\phantom{^2}}{\|x\|^2}$ & $0$ & $\frac{|z|\phantom{^2}}{\|x\|^2}$\\ \hline
some $x_j=0$  & $\frac{|z|\phantom{^2}}{\|x\|^2}$ & $\frac{|z|\phantom{^2}}{\|x\|^2}$ & $\frac{|z|\phantom{^2}}{\|x\|^2}$ \\ \hline
$x=0$ & $0$ & $\infty$ & $\infty$\\ \hline
\end{tabular}
\caption{The order of the contributions to $\hat{\nu}$ and $\hat{\beta}$ coming respectively from $\Gamma_0$ and $\Gamma_\infty$ (it is assumed $p=(x,z)\neq (0,0)$). The ``topology'' counting function $\hat{\beta}$ is more stable: it behaves as a rational function, whereas $\hat{\nu}$ has a ``delta function'' when some $x_j$ is zero. Notice that isolated geodesics are always finite.}\label{comparison}
\end{figure}

\begin{remark}On a contact Carnot group there is a well defined family of ``non-homogeneous dilations'' $\delta_\epsilon(x,z)=(\epsilon x, \epsilon^2 z)$, where $\epsilon>0$ (see \cite{nostrolibro, Bellaiche}). These dilations have the property that if $\gamma$ is a geodesic between the origin and $p$, then $\delta_\epsilon\gamma$ is a geodesic between the origin and $\delta_\epsilon(p)$ (the energies are though different, see Prop. \ref{p:blowup} below). In particular both the counting function and the topology function are constants along the trajectories of $\delta_\epsilon$: 
\begin{equation}
\hat{\nu}(\delta_\epsilon(p))=\hat{\nu}(p)\qquad \text{and}\qquad \hat{\beta}(\delta_\epsilon(p))=\hat{\beta}(p)\qquad \text{for all $\epsilon>0$}.
\end{equation}
\end{remark}

\subsection{Families of geodesics}

A simple way to produce families of geodesics (critical manifolds) is to act on a geodesic $\gamma$ with sub-Riemannian isometries fixing the endpoints of $\gamma$. 

\begin{example}[Heisenberg, continuation]
Let us consider the Heisenberg group $\mathbb{H}_{2n+1}$. Thus $k=1$ and $\al=1$ ($A$ is the canonical symplectic matrix). Let $p=(0,z)$ be a vertical point and $\gamma$ a geodesic from the origin to $p$. The group of isometries fixing the origin is isomorphic to:
\begin{equation}
\textrm{ISO}(\mathbb{H}_{2n+1})\simeq \U(n)\rtimes \Z_2.
\end{equation}
Each isometry $g$ in the connected component $\U(n)$ of the identity fixes $p=(0,z)$, thus $g\gamma$ is still a geodesic from the origin to $p$; such an isometry stabilizes the whole $\gamma$ if it fixes the initial covector. Then, the stabiliser subgroup of the geodesic $\gamma$ is $\textrm{ISO}_\gamma(\mathbb{H}_{2n+1})\simeq \U(n-1).$ In this way we produce a family:
\begin{equation}
X_\gamma=\U(n)/\U(n-1)\simeq S^{2n-1},
\end{equation}
consisting of distinct geodesics isometrically equivalent to $\gamma$. In other words all geodesics in $X_\gamma$ are obtained from $\gamma$ by composition with an isometry (and they all have the same energy). In this case, it turns out that $X_\gamma$ is a connected component of $\Gamma_\infty(p)$, i.e. a critical manifold.
\end{example}

Surprisingly this is not the case for more general Carnot groups. In fact, given a critical manifold $X\subset \Gamma_{\infty}(p)$ (one of the above spheres), this need not be obtained by acting with the stabilizer of $p$ on a fixed geodesic. In other words, geodesics forming $X$, although all having the same energy and endpoints, might be isometrically non-equivalent. They are ``deformations'' of each other, but not via isometries.

We say that two geodesics with the same endpoints are \emph{isometrically equivalent} if they are obtained one from the other by composition of an isometry of $G$. We denote by $\bar\Gamma_{\infty}(p)$ the set of equivalence classes of isometrically equivalent geodesics ending at $p$. For example, a family of isometrically equivalent geodesics corresponds to just a point in the quotient $\bar{\Gamma}_\infty(p)$.

The topology of this set (a quotient of $\Gamma_{\infty}(p)$) is related with the commensurability of the singular values of $A$ (see Thm.~\ref{t:noninvnoneqfamilies}).

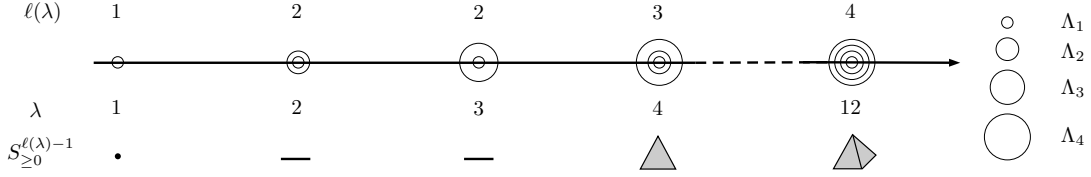
\begin{figure}
\psscalebox{0.8 0.8} % Change this value to rescale the drawing.
{
\begin{pspicture}(0,-1.3966794)(17.8571,1.3966794)
\definecolor{colour0}{rgb}{0.8,0.8,0.8}
\rput[bl](0.25722224,1.0366795){$\ell(\lambda)$}
\rput(1.7905554,1.2366793){$1$}
\rput(4.790555,1.2366793){$2$}
\rput(7.812778,1.2255683){$2$}
\rput(10.790555,1.2255683){$3$}
\rput(0.535,-1.107765){$S^{\ell(\lambda)-1}_{\geq 0}$}
\pstriangle[linecolor=black, linewidth=0.02, fillstyle=solid,fillcolor=colour0, dimen=outer](10.790555,-1.3966539)(0.62222224,0.5777778)
\psdots[linecolor=black, dotsize=0.1](1.812778,-1.1744317)
\psline[linecolor=black, linewidth=0.04](4.523889,-1.218876)(5.012778,-1.218876)
\psline[linecolor=black, linewidth=0.04](7.567639,-1.218876)(8.056527,-1.218876)
\rput[bl](17.4871,-0.19436){$\Lambda_3$}
\rput[bl](17.4871,0.40564){$\Lambda_2$}
\pscircle[linecolor=black, linewidth=0.01, dimen=outer](16.596779,-0.02603713){0.29032257}
\pscircle[linecolor=black, linewidth=0.01, dimen=outer](16.596779,0.60680187){0.19354838}
\rput[bl](17.4871,0.90564){$\Lambda_1$}
\pscircle[linecolor=black, linewidth=0.01, dimen=outer](16.59563,1.0500127){0.1}
\pscircle[linecolor=black, linewidth=0.01, dimen=outer](16.596779,-0.8524245){0.38709676}
\rput[bl](17.4871,-0.99436){$\Lambda_4$}
\rput(13.990556,1.2366793){$4$}
\pspolygon[linecolor=black, linewidth=0.02, fillstyle=solid,fillcolor=colour0](13.706326,-1.3821942)(14.183998,-1.3821942)(14.412778,-1.1545345)(14.023098,-0.8144165)
\psline[linecolor=black, linewidth=0.02](14.026404,-0.81433064)(14.175427,-1.3823105)
\pscircle[linecolor=black, linewidth=0.01, dimen=outer](14.012777,0.38757548){0.38709676}
\pscircle[linecolor=black, linewidth=0.01, dimen=outer](14.012777,0.38757548){0.29032257}
\pscircle[linecolor=black, linewidth=0.01, dimen=outer](14.012777,0.38757548){0.19354838}
\pscircle[linecolor=black, linewidth=0.01, fillstyle=solid, dimen=outer](14.012777,0.38757548){0.1}
\psline[linecolor=black, linewidth=0.04, linestyle=dashed, dash=0.17638889cm 0.10583334cm](11.412778,0.38112387)(13.6,0.39001274)
\psline[linecolor=black, linewidth=0.04, arrowsize=0.05291667cm 2.0,arrowlength=1.4,arrowinset=0.0]{->}(13.2,0.39001274)(15.812778,0.38112387)
\pscircle[linecolor=black, linewidth=0.01, dimen=outer](10.812778,0.38757548){0.38709676}
\pscircle[linecolor=black, linewidth=0.01, dimen=outer](7.812778,0.38757548){0.32258064}
\pscircle[linecolor=black, linewidth=0.01, dimen=outer](10.812778,0.38757548){0.19354838}
\pscircle[linecolor=black, linewidth=0.01, dimen=outer](4.812778,0.38757548){0.19354838}
\pscircle[linecolor=black, linewidth=0.01, fillstyle=solid, dimen=outer](10.812778,0.38757548){0.1}
\pscircle[linecolor=black, linewidth=0.01, fillstyle=solid, dimen=outer](7.812778,0.38757548){0.1}
\pscircle[linecolor=black, linewidth=0.01, fillstyle=solid, dimen=outer](4.812778,0.38757548){0.1}
\pscircle[linecolor=black, linewidth=0.01, fillstyle=solid, dimen=outer](1.812778,0.38757548){0.1}
\psline[linecolor=black, linewidth=0.04](1.4127778,0.38112387)(11.412778,0.38112387)
\rput[bl](0.34611112,-0.51887614){$\lambda$}
\rput(1.7905554,-0.36332065){$1$}
\rput(4.790555,-0.36332065){$2$}
\rput(7.812778,-0.37443164){$3$}
\rput(10.790555,-0.37443164){$4$}
\rput(13.990556,-0.36332065){$12$}
\end{pspicture}
}
\caption{Equivalence classes of isometrically non-equivalent families of geodesics for $k=4$ commensurable singular values $\alpha_i = 2\pi/ i$ for $i=1,2,3,4$. Thus $\Lambda_i = i\Z \setminus \{0\}$.}\label{fig:familiesintro}
\end{figure}

\begin{theorem}[Isometrically equivalent geodesics]\label{t:noninvnoneqfamiliesintro}Let $G$ be a contact Carnot group.
The set $\bar\Gamma_{\infty}(p)$ of equivalence classes of isometrically equivalent geodesics ending at $p\neq p_0$ is homeomorphic to:
\begin{equation}
\bar\Gamma_{\infty}(p) \simeq \bigcup_{\lambda \in {\Lambda}_p} S_{\geq 0}^{\ell(\lambda)-1},\qquad \ell(\lambda):=\# L(\lambda), 
\end{equation}
where $S^m_{\geq 0} = S^{m} \cap \R^{m+1}_{\geq 0}$ is the intersection of the $m$-sphere with the positive quadrant in $\R^{m+1}$. 
\end{theorem}

See Fig.~\ref{fig:familiesintro}. When $A$ is generic, for every $\la \in \Lambda_p \subseteq \Lambda$ we have $\ell(\la)=1$ and $\bar\Gamma_{\infty}$ is a discrete set of points, one for each $\la \in \Lambda_p$ (all the geodesics in a critical manifold $X\simeq S^{1}$ are isometrically equivalent to a given one).  Nevertheless, non-trivial manifolds of isometrically non-equivalent geodesics appear when there are \emph{resonances}.

\subsection{A limiting procedure}
We discuss here the main ingredient of our study for contact sub-Riemannian manifolds: the \emph{nilpotent approximation} of the structure at a point $p_0$. Because of the local nature of the problem, we can assume that $M=\R^{2n+1}$ and the point $p_0$ is the origin. Moreover, the distribution $\distr \subset T\R^{2n+1}$ is given by:
\begin{equation}
\distr=\spn\{f_1, \ldots, f_{2n}\},
\end{equation}
where $f_1, \ldots, f_{2n}$ are bounded vector fields on $\R^{2n+1}$. The sub-Riemannian structure on $\distr$ is obtained by declaring these vector fields to be orthonormal at each point.

We assume that the coordinates $(x, z)\in \R^{2n}\times \R$ are \emph{adapted} to the distribution at the origin namely, $\distr_{p_0} = \spn\{\partial_{x_1},\ldots,\partial_{x_{2n}}\}$ (for example we take canonical Darboux's coordinates). In the language of sub-Riemannian geometry these coordinates, at least in the contact (or step $2$) case, are also called \emph{privileged}.  Using these coordinates we define ``dilations'' $\delta_\epsilon: M \to M$ by:
\begin{equation}
\delta_{\epsilon}(x,z):=(\epsilon x, \epsilon^2 z),\qquad \epsilon>0,
\end{equation}
and the \emph{nilpotent approximation at $p_0$}, another sub-Riemannian structure on the same base manifold $M$, given by declaring the following fields:
\begin{equation}
\hat{f}_i:=\lim_{\epsilon\to 0}\epsilon \delta_{\frac{1}{\epsilon}*}f_i, \qquad \forall i=1,\ldots,2n,
\end{equation}
a new orthonormal frame. Thus, the nilpotent approximation at a point $p_0$ is the ``principal part'' of the original sub-Riemannian structure in a neighbourhood of $p_0$ w.r.t. the non-homogeneous dilations $\delta_\epsilon$. Moreover, it turns out that the nilpotent approximation at any point $p_0$ of a contact sub-Riemannian manifold is a contact Carnot group.

%The idea of the study of the asymptotic behaviour of geodesics in contact manifolds originates in \cite{Agraexponential}. In that case the author adopts the point of view that geodesics are rays through the exponential map and produces a kind of ``blowup'' of the initial covectors; as a result conjugate times for geodesics, at least asymptotically, are equal to the conjugate times for the nilpotent approximation \cite[Thm. 3.1]{Agraexponential}. 
%Let's see how such a Carnot group can be characterized in terms of the contact form $\alpha$. Let $f_0$ any vector field transversal to $\distr$ (e.g. the Reeb vector field of the contact structure), and let $\hat{f}_0$ the ``nilpotentization'' $f_0$, namely
%\begin{equation}
%\hat{f}_0=\lim_{\epsilon\to 0}\epsilon^2 \delta_{\frac{1}{\epsilon}*}f_0,
%\end{equation}
%(notice the factor $\epsilon^2$ in the r.h.s.). The nilpotent approximation at $p_0$, as a manifold, is diffeomorphic to $\R^{2n+1}$, and its stratified Lie algebra $\mathfrak{g} = \mathfrak{g}_1 \oplus \mathfrak{g}_2$ is given by
%\begin{equation}
%\mathfrak{g}_1 = \mathrm{span}\{\hat{f}_1,\ldots,\hat{f}_{2n}\}, \qquad \mathfrak{g}_2= \mathrm{span}\{\hat{f}_0\}.
%\end{equation}
%In particular it is not hard to prove that, for $i,j=1,\ldots,2n$, we have:
%\begin{equation}
%[\hat{f}_i, \hat{f}_j] = A_{ij} \hat{f}_0, \qquad A_{ij} := -\left.\frac{d \alpha(f_i,f_j)}{\alpha(f_0)}\right\rvert_{p_0}.
%\end{equation}

We introduce the following notation:
\begin{equation}
\nu(p)=\#\{\textrm{geodesics joining $p_0$ and $p$}\}.
\end{equation}
Thus $\nu(p)$ will denote the number of \emph{local} geodesics between $p_0$ and $p$, i.e. geodesics in $M$ that are contained in a coordinate chart of $p_0$. Similarly $\hat{\nu}$ denoted the number of geodesics between the origin and $p$ for the nilpotent approximation. The next theorem relates the geodesic count on the original structure and on the nilpotent Carnot group structure (see Thm.~\ref{t:limit}). 

\begin{theorem}[Counting in the limit]\label{t:limitintro}
Let $M$ be a contact sub-Rieman\-nian manifold. For the generic $p \in M$ sufficiently close to $p_0$:
\begin{equation}
\hat{\nu}(p)\leq \liminf_{\epsilon \to 0}\nu(\delta_\epsilon(p)).
\end{equation}
where $\delta_\epsilon$ is the non-homogeneous dilation defined in some set of adapted coordinates in a neighbourhood of $p_0$.
\end{theorem}

Combining Thm. \ref{t:limitintro} and Thm. \ref{thm:orderintro} we obtain an estimate for the order of growth of the number of ``local'' geodesics between two close points on a contact manifold (see Thm.~\ref{t:darboux}). 

\begin{theorem}[The local bound]\label{t:darbouxintro}
Let $M$ be a contact manifold and $q \in M$. Denote by $(x,z)$ Darboux's coordinates on a neighbourhood $U$ of $q$. There exist constants $C(q), R(q)$ such that, for the generic $p=(x,z)\in U$:
\begin{equation}
\liminf_{\epsilon \to 0}\nu(\delta_\epsilon(p)) \geq C(q)\frac{|z|\phantom{^2}}{\|x\|^2} + R(q).
\end{equation}
\end{theorem}
%\red{
%\begin{remark}
%The constants $C(q)$ and $R(q)$ appearing in Thm.~\ref{t:darbouxintro} can be made more explicit by exploiting the (linear) relation between Darboux's coordinates and exponential coordinates on the nilpotent approximation at $q$ (see Sec.~\ref{s:contact}).
%\end{remark}}

A completely new phenomenon in the sub-Riemannian case is the existence of a sequence of points $q_n\to q$ with arbitrary large number of local geodesics between the two (see Thm.~\ref{t:local}). Notice that, in general, we cannot predict the existence of a point $p$ with infinitely many local geodesics between $q$ and $p$.

\begin{theorem}[Abundance of ``local'' geodesics]\label{t:localintro}
Let $M$ be a contact sub-Riemannian manifold and $q\in M$. Then there exists a sequence $\{q_n\}_{n\in \N}$ in $M$ such that:
\begin{equation}
\lim_{n\to \infty}q_n= q\qquad \text{and}\qquad \lim_{n\to \infty}\nu(q_n)=\infty.
\end{equation}
\end{theorem}

%\subsection{Structure of the paper}
%In Sec. \ref{s:prel} we discuss some preliminary material, such as the definition of the main objects. Sec. \ref{s:main} is devoted to the study of the exponential map for contact Carnot groups and terminates with the proof of the Theorem on the topology of the critical manifolds (Thm. \ref{thm:disjoint}). Thm. \ref{thm:orderintro} above is split into two parts: the upper bound is discussed in Sec. \ref{s:upperbound} (Thm. \ref{thm:upper}) and the lower bound in Sec. \ref{s:lower} (Thm. \ref{thm:lower} and Thm. \ref{thm:lowertop}); Prop. \ref{p:oneb} is proved in section \ref{s:upperbound} (Prop. \ref{p:oneb}). In Secs. \ref{s:isoheis}~\ref{s:isocarnot} we study the isometry group of contact Carnot Groups. The theorem on equivalence classes of isometric geodesics is proved in Sec. \ref{s:families} (Thm. \ref{t:noninvnoneqfamilies}). Sec. \ref{s:contact} is devoted to general contact manifolds: we first review their nilpotent approximation at a point and we prove some related properties; next we prove Thm. \ref{t:limitintro} (Thm. \ref{t:limit}), Thm. \ref{t:darbouxintro} (Thm. \ref{t:darboux}) and Thm. \ref{t:localintro} (Thm. \ref{t:local}).

\subsection*{Acknowledgments}
The authors thank A. A. Agrachev, D. Barilari, A. Gentile for stimulating comments. We thank also the anonymous referee and P. Silveira for carefully reading the manuscript. The first author was supported by the European Community's Seventh Framework Programme ([FP7/2007-2013] [FP7/2007-2011]) under grant agreement No. [258204].
		% introduction
%!TEX root = enumerative-v5.tex

\section{Preliminaries}\label{s:prel}

We recall some basic facts in sub-Riemannian geometry. We refer to \cite{nostrolibro,librorifford,Jea-2100,montbook} for further details. Let $M$ be a smooth, connected manifold of dimension $n \geq 3$. A sub-Riemannian structure on $M$ is a pair $(\distr,\metr)$ where $\distr$ is a smooth vector distribution of constant rank $k\leq n$ satisfying the \emph{H\"ormander condition} (i.e. $\mathrm{Lie}_x\distr = T_x M$, $\forall x \in M$) and $\metr$ is a smooth Riemannian metric on $\distr$. A Lipschitz continuous curve $\g:[0,1]\to M$ is \emph{admissible} (or \emph{horizontal}) if $\dot\g(t)\in\distr_{\g(t)}$ for a.e. $t \in [0,1]$.
Given a horizontal curve $\g:[0,1]\to M$, the \emph{energy} of $\g$ is
\begin{equation}%\label{e-lunghezza}
J(\g)=\int_I \|\dot{\g}(t)\|^2dt,
\end{equation}
where $\|\cdot\|$ denotes the norm induced by $\metr$.  The pair $(\distr,\metr)$ can be given, at least locally, by assigning a set of $k$ smooth vector fields that span $\distr$, orthonormal for $\metr$. In this case, the set $\{f_1,\ldots,f_k\}$ is called a \emph{local orthonormal frame} for the sub-Riemannian structure. 

%In this paper we are mainly concerned with \emph{contact} sub-Riemannian manifolds; in this framework admissible curves are also called \emph{Legendrian}.

\begin{definition}\label{ex:contact}
A sub-Riemannian manifold is \emph{contact} if locally there exists a one form $\alpha$ such that $\distr= \ker \alpha$, and $d\alpha|_{\distr} $ is non degenerate (the rank of $\distr$ must be even). Admissible curves are called \emph{Legendrian}. %The \emph{Reeb vector field} $f_0 \in \VecM$ is the unique vector field such that $d\alpha(f_0,\cdot) = 0$ and $\alpha(f_0) = 1$. 
%Due to the non-degeneracy assumption, the rank of $\distr$ is even.
\end{definition}

\begin{definition}
Let $M$ be a contact manifold. A sub-Riemannian \emph{geodesic} is a non-constant Legendrian curve $\g:[0,1]\to M$ that is locally energy minimizer. More precisely, for any $t \in [0,1]$ there exists a sufficiently small interval $I \subseteq [0,1]$, containing $t$, such that the restriction $\g|_{I}$ minimizes the energy between its endpoints. 
\end{definition}

Any geodesic starting at $p_0$ can be lifted to a Lipschitz curve $\eta :[0,1] \to T^*M$ called \emph{sub-Riemannian extremal}, as we discuss now. In general, sub-Riemannian extremals can be \emph{normal} or \emph{abnormal}, but abnormal extremals do not appear in contact or Riemannian structures. For this reason we only discuss normal extremals.

%many structures, and their regularity is one of the main open problems in sub-Riemannian geometry (see \cite{LeLeMoVi12,LeMo08} and references therein for recent progresses on this topic). Abnormal extremals do not appear in contact (or Riemannian) structures. We now give an explicit characterization of \emph{normal extremals}.

\begin{definition}
The \emph{Hamiltonian function} $H \in C^\infty(T^*M)$ is
\begin{equation}
H(\eta) = \frac{1}{2}\sum_{i=1}^k\langle \eta, f_i\rangle^2, \qquad \forall \eta \in T^*M,
\end{equation}
where $f_1,\ldots,f_k$ is a local orthonormal frame and $\langle\eta,\cdot\rangle$ denotes the action of the covector $\eta$ on vectors. 
\end{definition}

Let $\sigma$ be the canonical symplectic form on $T^*M$. With the symbol $\vec{a}$ we denote the Hamiltonian vector field on $T^*M$ associated with a function $a \in C^\infty(T^*M)$. Indeed $\vec{a}$ is defined by the formula $da = \sigma(\cdot,\vec{a})$. Consider the \emph{Hamiltonian vector field} $\vec{H} \in \mathrm{Vec}(T^*M)$.

\begin{definition}
Non-constant trajectories of the Hamiltonian system $\dot\eta = \vec{H}(\eta)$ are normal \emph{sub-Riemannian extremals}.
\end{definition}

In any structure where abnormal extremals do not exist (such as contact or Riemannian structures), the next theorem completely characterizes \emph{all} geodesics.
%Since normal sub-Riemannian extremals are non-constant trajectories of the Hamiltonian field $\vec{H}$, we obtain the following well known result.
\begin{theorem}\label{t:pmpw}
Normal sub-Riemannian geodesics are exactly projections on $M$ of normal sub-Riemannian extremals. In particular, all normal geodesics are smooth.
\end{theorem}
%We stress that in any structure where abnormal extremals do not exist (such as contact or Riemannian structures), Thm.~\ref{t:pmpw} completely characterizes \emph{all} geodesics. 
Moreover, any normal sub-Riemannian geodesic can be specified by its \emph{initial covector}. 
\begin{definition}
The \emph{sub-Riemannian exponential map} (with origin $p_0$) $E: T_{p_0}^*M \to M$ is
\begin{equation}\label{eq:expmap}
E(\eta_{0}):= \pi(e^{\vec{H}}(\eta_{0})), \qquad \forall \eta_0 \in T_{p_0}^*M.
\end{equation}
where $e^{t\vec{H}}(\eta_0)$ denotes the integral curve of $\vec{H}$ starting from $\eta_0$.
\end{definition}
Thus all geodesics from $p_0$ are the image through $E$ of the ray $t \mapsto t\eta$. We denote by $\Gamma(p)=E^{-1}(p)\subset T_{p_0}^*M$ the set of geodesics from $p_0$ to $p\neq p_0$, with the subset topology.

\subsection{Fibers of the exponential map and geodesics}
Notice that the correspondence:
\begin{equation}
\eta\mapsto \gamma_\eta, \qquad \gamma_{\eta}(t)=\pi(e^{t\vec{H}}(\eta))
\end{equation}
defines a continuous map from $T_{p_0}^*M$ to the set of admissible curves. If we endow this set with the $W^{1, \infty}$-topology and we assume $p\neq p_0$, this map restricts to a homeomorphism between $\Gamma(p)$ and the set of geodesics to $p$: the topologies on $\Gamma(p)$ as a subset of $T_{p_0}^*M$ or as a subset of the space of admissible curves coincide and the point of view we adopted is not restrictive.

On the other hand, recall that extremals (resp. geodesics) are non-constant and for these reasons we will always make the assumption $p\neq p_0$.
Most of our results are true also for $p=p_0$, but then one should regard $\Gamma(p)$ simply as the fiber of $E$ and not as the set of geodesics to $p$.

\subsection{Contact Carnot groups}

%We start with the preliminary study of \emph{contact Carnot groups}. As we will see, their analysis provides the main building block for the general theory on contact manifolds. Since we are mainly interested in contact structures, we restrict ourselves to step $2$ Carnot groups. We refer to \cite{Jea-2100,noteledonne,montbook} for more details.

A corank 1 Carnot group $G$ is a simply connected Lie group whose Lie algebra of left-invariant vector fields $\mathfrak{g}$ admits a nilpotent stratification of step $2$:
\begin{equation}
\mathfrak{g} = \mathfrak{g}_1 \oplus \mathfrak{g}_2, \qquad \mathfrak{g}_1,\mathfrak{g}_2 \neq \{0\},
\end{equation}
with $\dim \mathfrak{g}_2 = 1$ and
\begin{equation}
[\mathfrak{g}_1,\mathfrak{g}_1] = \mathfrak{g}_{2}, \qquad \text{and} \qquad [\mathfrak{g}_1,\mathfrak{g}_2] = [\mathfrak{g}_2,\mathfrak{g}_2]= \{0\}.
\end{equation}
%For any $x \in G$, we denote $L_x:G \to G$ the left translation by the element $x$, namely
%\begin{equation}
%L_x y := x y, \qquad \forall y \in G
%\end{equation}
We define a scalar product on $\mathfrak{g}_1$ by declaring a set $f_1,\ldots,f_{k} \in \mathfrak{g}_1$ to be a global orthonormal frame. In particular, $\distr|_x = \mathfrak{g}_1|_x$, for all $x \in G$. The group exponential map,
\begin{equation}
\mathrm{exp}_{G} : \mathfrak{g} \to G,
\end{equation}
associates with $v \in \mathfrak{g}$ the element $\gamma(1)$, where $\gamma: [0,1] \to G$ is the unique integral line of the vector field $v$ such that $\gamma(0) = 0$. Since $G$ is simply connected and $\mathfrak{g}$ is nilpotent, $\mathrm{exp}_G$ is a smooth diffeomorphism. Thus we can identify $G \simeq \R^m$ with a polynomial product law.

%\begin{remark}
%In the literature, these structures are also referred to as \emph{Carnot groups} of type $(k,m)$, where $k = \dim \distr = \rank \mathfrak{g}_1$ is the \emph{rank} of the distribution and $m$ is the dimension of $G$. 
%\end{remark}

\begin{definition}
A contact Carnot group is a corank $1$ Carnot group that admits a contact structure with $\distr = \mathfrak{g}_1$.
\end{definition}
The only non-trivial request is the non-degeneracy of the contact form. In fact, let $G$ be a contact Carnot group, $f_1,\ldots,f_{2n} \in \mathfrak{g}_1$ be a global orthonormal frame of left-invariant vector fields, and $f_0 \in \mathfrak{g}_2$ a generator for the second layer. Indeed:
\begin{equation}
[f_i,f_j] = A_{ij} f_0, \qquad \forall i,j=1,\ldots,2n,
\end{equation}
for some constant matrix $A \in \so(2n)$. Observe that there exists a unique never-vanishing left invariant one-form $\alpha$ (up to constant scaling) such that $\distr = \ker \alpha$. 
Using the identity $d\alpha(X,Y)=X\alpha(Y)-Y\alpha(X)-\alpha([X,Y])$ we obtain:
\begin{equation}\label{eq:cartan}
d\alpha(f_i,f_j) = -\alpha([f_i,f_j]) = A_{ji} \alpha(f_0).
\end{equation}
Since $\alpha(f_0) \neq 0$ the matrix $A$ is non-degenerate.

\subsection{Normal form of contact Carnot groups}

By acting on $\mathfrak{g}_1$ with an orthogonal transformation, it is always possible to put $A$ in its canonical form. Such a transformation can be trivially extended to an automorphism of $\mathfrak{g}$, and thus lifts to a group automorphism of $G$ that preserves the scalar product. Therefore, up to isometries, contact Carnot groups are parametrised by the possible singular values of non-degenerate matrices $A \in \so(2n)$. In the following we describe the possible normal forms of contact Carnot groups. Consider the triple $(k, \vec{n},\vec{\alpha})$, where:
\begin{itemize}	
\item[(i)] $k \in \N$, with $1 \leq k \leq n$,
\item[(ii)] $\vec{n} =(n_1,\ldots,n_k)$ is a partition of $n$, namely $n_j \in \N$ and $\sum_{j=1}^k n_j = n$,
\item[(iii)] $\vec{\alpha} =(\alpha_1,\ldots,\alpha_k)$ with $0 < \alpha_1 < \ldots < \alpha_k$.
\end{itemize}
For a fixed choice of $(k,\vec{n},\vec{\alpha})$, let:
\begin{equation}\label{eq:matrixA}
A:= \diag(\alpha_1J_{n_1},\ldots, \alpha_k J_{n_k}) \in \so(2n), \qquad \text{with} \qquad J_{m} = \begin{pmatrix}
0 & \I_m \\ -\I_m & 0
\end{pmatrix}.
\end{equation}
In other words, $A$ has $k$ distinct singular values $0<\alpha_1<\cdots < \alpha_k$, with multiplicities $n_1,\ldots,n_k$ (half the dimension of the corresponding invariant subspaces). This gives the normal form of the $(2n,2n+1)$ graded Lie algebra with parameters $(k,\vec{n},\vec{\alpha})$. As an abstract algebra is given by:
\begin{equation}
\mathfrak{g} = \mathfrak{g}_1 \oplus \mathfrak{g}_2, \qquad \mathfrak{g}_1 = \spn\{f_1,\ldots,f_{2n}\},\quad \mathfrak{g}_2 = \spn\{ f_0\},
\end{equation}
with:
\begin{equation}
[f_i,f_j] = A_{ij} f_0, \qquad i,j=1,\ldots,2n.
\end{equation}
Let $G$ be the unique connected, simply connected Lie group such that $\mathfrak{g}$ is its Lie algebra. Define a scalar product on $\mathfrak{g}_1$ such that $f_1,\ldots,f_{2n}$ is an orthonormal frame. Any contact Carnot group is isomorphic to one of these structures, for a choice of $(k,\vec{\alpha},\vec{n})$. Notice that the normal form is determined only up to global rescaling of the eigenvalues $\vec{\alpha}$ (see \cite[Remark 1]{AGL}).

\subsection{Exponential coordinates}
The orthonormal basis $f_1,\ldots,f_{2n}$ and $f_0$ realize the splitting
\begin{equation}
\mathfrak{g} = \mathfrak{g}_1^{\alpha_1} \oplus \cdots \oplus \mathfrak{g}_1^{\alpha_k} \oplus \mathfrak{g}_2,
\end{equation}
with respect to the generalized eigenspaces of $A$. Accordingly, we identify:
\begin{equation}
G \simeq \R^{2n_1} \oplus \cdots \oplus \R^{2n_k} \oplus \R,
\end{equation}
through the group exponential map $\mathrm{exp}_{G}: \mathfrak{g} \to G$, in such a way that $p \in G$ has exponential coordinates $(x_1,\ldots,x_k,z)$ with $x_i \in \R^{2n_i}$ for $i=1,\ldots,k$ and $z \in \R$.

\subsection{An explicit representation}

An explicit representation of the contact Carnot group with parameters $(k,\vec{\alpha},\vec{n})$ is given by the sub-Riemannian structure induced by the following vector fields on $\mathbb{R}^{2n+1}$, with coordinates $(x,z) \in \R^{2n}\times \R$:
\begin{equation}\label{eq:explicitrep}
f_{i} := \frac{\partial}{\partial x_i} - \frac{1}{2}\sum_{j=1}^{2n}A_{ij}x_j \frac{\partial}{\partial z}, \qquad f_0 := \frac{\partial}{\partial z}, \qquad i =1,\ldots,2n,
\end{equation}
where $A$ is the matrix of Eq.~\eqref{eq:matrixA} with $k$ singular values $\vec{\alpha}$ and multiplicities $\vec{n}$. For the Heisenberg groups $\mathbb{H}_{2n+1}$ (see Example~\ref{ex:heisenberg}) $A$ is the standard symplectic matrix.

%The next lemma clarifies the relation between ``tautological'' coordinates $(x,z)$ and exponential coordinates. In fact, there is no difference.
\begin{lemma}\label{l:expcoords}
The coordinates $(x,z)$ are the exponential coordinates induced by $f_1,\ldots,f_{2n},f_0$.
\end{lemma}
\begin{proof}
Assume that $p=(x,z)$ has exponential coordinates $(\theta, \rho)$. This means that $(x,z)=\gamma(1)$, where $\gamma(t) = (x(t),z(t))$ is the solution of the Cauchy problem
\begin{equation}
\dot{x}_i(t) = \theta_i, \qquad \dot{z}(t) = \rho + \frac{1}{2}\sum_{i,j=1}^{2n} x_i A_{ij} \theta_j, \qquad \gamma(0)=(0,0),
\end{equation}
By the skew-symmetry of $A$, the solution is $x(t) = \theta t$ and $z(t) = \rho t$. Then $(x,z) =(\theta,\rho)$.
\end{proof}
%By the Campbell-Baker-Hausdorff formula, the product on $G$, in exponential coordinates, is:
%\begin{equation}
%(x,z) \cdot (x',z') = \left(x+x',z+ z' + \frac{1}{2} x^* A x\right).
%\end{equation}
%The classical example is the $(2n+1)$-dimensional Heisenberg group $\mathbb{H}_{2n+1}$ (see Example~\ref{ex:heisenberg}).
%\begin{example}\label{ex:heisenberg}
%A classical example is the $(2n+1)$-dimensional contact Carnot group $\mathbb{H}_{2n+1}$. This is the case with $k=1$, i.e. a unique singular value $\alpha_1 = 1$ with multiplicity $n$. In this case
%\begin{equation}
%f_{i} := \frac{\partial}{\partial x_i} - \frac{1}{2}x_{i+n} \frac{\partial}{\partial z}, \qquad f_{n+i} := \frac{\partial}{\partial x_{n+i}} + \frac{1}{2}x_{i} \frac{\partial}{\partial z}, \qquad
%f_0 := \frac{\partial}{\partial z}, \qquad i =1,\ldots,n,
%\end{equation}
%and $A$ is the standard symplectic matrix $J = \left(\begin{smallmatrix} 0 & \I_n \\ -\I_n & 0 \end{smallmatrix}\right)$.
%\end{example}

%\begin{remark}
%We introduced contact Carnot groups as special contact sub-Riemannian structures. We stress that the following construction can be generalized, with minor modifications, to any corank-one Carnot group, namely structures of type $(k,k+1)$ where $A$ may also have odd dimension and some vanishing singular values.
%\end{remark}
		% subriemannian preliminaries
%!TEX root = enumerative-v5.tex

\section{The fibers of the exponential map for contact Carnot groups}\label{s:main}
Let $\hat{E}:T_0^*G\to G$ be the exponential map for the contact Carnot group whose (nonzero) structure constants for its Lie algebra are given by equation~\eqref{eq:matrixA}. 
%\begin{equation}
%A=\textrm{diag}(\alpha_1J_{n_1}, \ldots, \alpha_{k}J_{n_k}), \qquad \text{where} \qquad J_m = \begin{pmatrix} 0 & \I_m \\ -\I_m & 0 \end{pmatrix}.
%\end{equation}
In the following, we write $p \in G$ in exponential coordinates as $p = (x_1,\ldots,x_k,z)$, with $x_j \in \R^{2n_j}$ and analogously, for $\eta \in T_0^*G$, we write $\eta = (u_1,\ldots,u_k,\lambda)$, with $u_j \in \R^{2n_j}$. Thus:
\begin{equation}
\hat{E}(u_1,\ldots,u_k, \lambda)=(x_1, \ldots, x_{k}, z)\qquad\text{with}\qquad x_j, u_j \in \R^{2n_j}, \quad j=1, \ldots, k.
\end{equation}
When convenient, we write $p = (x,z)$ and $\eta =(u,\lambda)$, with $x,u \in \R^{2n}$ and $n = \sum_{j=1}^k n_j$.
 
\begin{proposition}\label{propo:exp}
With the above notation we have for every $j=1, \ldots, k$:
\begin{equation}\label{exp}
x_j=\left(\frac{\sin (\la \al_j)}{\la \al_j}\I+\frac{\cos(\la \al_j)-1}{\la \al_j}J\right)u_j\qquad \text{and}\qquad z=\sum_{j=1}^k\left(\frac{\la \al_j -\sin(\la \al_j)}{2\la^2\al_j}\right)\|u_j\|^2.
\end{equation}
If $\lambda = 0$, then $x_j = u_j$ for $j=1,\ldots,k$ and $z = 0$, i.e. $\hat{E}(u, 0)=(u,0)$.
\end{proposition}

\begin{proof}
We recall that the sub-Riemannian exponential map is given explicitly by \cite{ABBHausdorff}:
\begin{equation}
(u, \la)\mapsto\left(\int_{0}^1e^{-\la A t}udt, -\frac{1}{2}\int_{0}^1\left\langle e^{-\la A t}u, A \int_{0}^te^{-\la A s}u ds\right\rangle dt\right).
\end{equation}
We start by considering the horizontal components (we omit the subscript for $J=J_{n_j}$):
\begin{equation}
x_j=\int_{0}^1e^{-\la \al_j J t}u_jdt.
\end{equation}
If $\la=0$, then $e^{-\la \al_j J t}=\I$ and $x_j=u_j$; otherwise the expression for $x_j$ follows immediately from writing the integrand matrix as:
\begin{equation}\label{exp1}
e^{-\la \al_j J t}=\cos(\la \al_j t)\I-\sin(\la \al_j t)J.
\end{equation}
In fact using  \eqref{exp1} we can also evaluate the matrix integral:
\begin{equation}\label{inv2}
\int_0^te^{-\la \al_j J t}dt=\frac{\sin (\la \al_j t)}{\la \al_j}\I+\frac{\cos(\la \al_j t)-1}{\la \al_j}J=a(t)\I+b(t)J.
\end{equation}
For the $z$ component, we notice that it can be rewritten as $z=u^*Su$, where $S$ is the matrix:
\begin{equation}
S=-\frac{1}{2}\int_{0}^1 \int_{0}^te^{\la A t}Ae^{-\la A s} ds dt,
\end{equation}
and since $A$ is assumed to be block-diagonal, we obtain:
\begin{equation}
z=\sum_{j=1}^{k} u_j^*S_ju_j\qquad \text{with}\qquad S_j=-\frac{1}{2}\int_{0}^1 e^{\la \al_j Jt}\al_jJ\int_0^te^{-\la \al_j J s} ds dt,
\end{equation}
%where once again we removed the subscript from the matrix $J=J_{n_j}$ in the expression for $S_j$. 
Notice that if $\la=0$ then $S=-\frac{1}{4}A$ and, being skew-symmetric, $z=u^*Su=0$. If $\la\neq 0$ the integrand matrix in $S_j$ equals, using \eqref{exp1}:
%\begin{equation}\label{eq:cd}
%\begin{aligned}
%\underbrace{e^{\la \al_j Jt}\al_jJ}_{\al_j\cos(\la\al_j t)J-\al_j\sin(\la \al_jt)\I}\cdot\underbrace{\int_0^te^{-\la \al_j J s} ds}_{a(t)\I+b(t)J}&=(c(t)\I+d(t)J)(a(t)\I+b(t)J)\\
%&=(ac-bd)(t)\I+(ad+bc)(t)J.
%\end{aligned}
%\end{equation}
%\begin{equation}
\begin{multline}\label{eq:cd}
e^{\la \al_j Jt}\al_jJ \int_0^te^{-\la \al_j J s} ds  = \left(\alpha_j \cos(\lambda \alpha_j t) J -\alpha_j \sin(\lambda \alpha_j t) \I\right)\left(a(t) \I + b(t)J\right) =\\
 = \left(c(t)\I + d(t)J\right)\left(a(t) \I + b(t)J\right) =  (ac-bd)(t)\I+(ad+bc)(t)J,
\end{multline}
where $c(t) = \alpha_j\cos(\lambda\alpha_j t)$ and $d(t) = -\alpha_j \sin(\lambda \alpha_j t)$. Since $\int (ad+bc)J$ is skew-symmetric:
\begin{equation}
u_j^*S_ju_j=u_j^*\I\left(-\frac{1}{2}\int_{0}^1(ac-bd)(t) dt\right)u =-\|u_j\|^2\frac{1}{2}\int_{0}^1(ac-bd)(t)dt.
\end{equation}
Using the explicit expression of  $a, b, c, d$ (given by \eqref{inv2} and \eqref{eq:cd}), we obtain $(ac-bd)(t)=\frac{\cos(\la \al_j t)-1}{\la}$, whose integral equals:
\begin{equation}
\int_{0}^1\frac{\cos(\la \al_j t)-1}{\la}dt=\frac{\sin(\la \al_j)-\la \al_j}{\la^2\al_j}.
\end{equation}
Substituting this into the above formula for $u_j^*Su_j$ concludes the proof.
\end{proof}

For all $j=1,\ldots,k$, we define the $2n_j \times 2n_j$ matrix:
\begin{equation}
I(\la \al_j)=\frac{\sin (\la \al_j)}{\la \al_j}\I +\frac{\cos(\la \al_j)-1}{\la \al_j}J,
\end{equation}
where $I(0) = \I$. I this way, equation \eqref{exp} reads $x_j=I(\al_j \la)u_j$.
\begin{proposition}\label{propo2}
Assume $\la \al_j \notin 2\pi\mathbb{Z}\setminus\{0\}$. Then $I(\la \al_j)$ is invertible with inverse:
\begin{equation}
I(\la \al_j)^{-1}=\frac{\la \al_j}{2}\cot \left(\frac{\la \al_j}{2}\right)\I+\frac{\la \al_j}{2}J,
\end{equation}
(if $\la \al_j=0$ we have $I(0)^{-1} = \I$). In particular if $x_j=I(\la \al_j)u_j,$ then:
\begin{equation}
\frac{\la \al_j -\sin(\la \al_j)}{2\la^2\al_j}\|u_j\|^2=\frac{\al_j}{8}\frac{\la \al_j-\sin(\la \al_j)}{\sin \left(\frac{\la \al_j}{2}\right)}\|x_j\|^2.
\end{equation}
Moreover if $\la \al_j\in 2\pi \mathbb{Z}\setminus \{0\}$, then $x_j=0$.
\end{proposition}
\begin{proof}
The determinant if $I(\la \al_j)$ is: 
\begin{equation}
\det I(\la \al_j)=2\left(\frac{1-\cos (\la \al_j)}{\la^2\al_j^2}\right),
\end{equation}
and is nonzero if and only if  $\frac{\la \al_j}{2\pi} \notin \mathbb{Z}\setminus{0}$; in this case the matrix $I(\la \al_j)^{-1}$ is well defined.

For the second part of the statement we write $I(\la \al_j)^{-1}=c_1\I+ c_2J$, where $c_1=\frac{\la \al_j}{2}\cot \left(\frac{\la \al_j}{2}\right)$ and $c_2=\frac{\la \al_j}{2}$. Then, $u_j=c_1x_j+c_2Jx_j$ and since $x_j$ and $Jx_j$ are orthogonal we obtain:
\begin{equation}
\|u_j\|^2=c_1^2\|x_j\|^2+c_2^2\|Jx_j\|^2=(c_1^2+c_2^2)\|x_j\|^2.
\end{equation}
Computing $c_1^2+c_2^2=(\frac{\la \al_j}{2}\frac{1}{\sin (\la \al_j /2)})^2$, and setting $y=\la \al_j$ we finally obtain:
\begin{equation}
\frac{y -\sin y}{2 y^2/\al_j}\|u_j\|^2 =\frac{y -\sin y}{2 y^2/\al_j} \left(\frac{y}{2}\frac{1}{\sin (y/2)}\right)^2\|x_j\|^2=\frac{\al_j}{8}\frac{y-\sin y}{ (\sin\frac{y}{2})^2}.
\end{equation}
The last statement follows immediately by Eq.~\eqref{exp}.
\end{proof}

\subsection{A relevant function}

We introduce the function $g:\R\to \R\cup\{\infty\}$ defined by:
\begin{equation}
g(\la)=\frac{1}{8}\frac{\la-\sin \la}{\left(\sin \frac{\la}{2}\right)^2}.
\end{equation}
Each pole of $g$ is of order two and lies on $2\pi \mathbb{Z}\setminus \{0\}$ (see Fig.~\ref{fig:g} in Sec.~\ref{s:intro} and Fig.~\ref{fig:detailg}). The proof of the following proposition is left to the reader.

\begin{figure}
\centering
{\psset{algebraic,plotpoints=10000,plotstyle=line,xunit=0.8cm,yunit=0.8cm}
\begin{pspicture}(4,-1.3)(14,4)
\psline[linewidth=0.02cm]{->}(4,-1)(14,-1) 		% asse x
\psline[linewidth=0.02cm]{->}(4,-1)(4,4)		% asse y
\psline[linewidth=0.01cm,linestyle=dashed](6.28319,-1)(6.28319,4)	% asintoto a 2Pi
\psline[linewidth=0.01cm,linestyle=dashed](12.5664,-1)(12.5664,4)	% asintoto a 4Pi
\psplot[yMaxValue=4]{4}{14}{1/8*(x-sin(x))/(sin(x/2))^2}			% funzione g (vera)
\psplot[yMaxValue=4]{4}{14}{2.5/8*(x-8.9868)+1.12335}					% retta (con il trucco, ma il minimo e' vero)
\psline[linewidth=0.01cm,linestyle=dashed](8.9868,-1)(8.9868,1.12335)	% tratteggio minimo
\psline[linewidth=0.01cm,linestyle=dashed](10.05,-1)(10.05,1.48)	% tratteggio dispari
\rput[c](6.28,-1.3){\scriptsize$2k\pi$}
\rput[c](12.7,-1.3){\scriptsize$2(k+1)\pi$}
\rput[c](8.9868,-1.3){\scriptsize$\mu_k$}
\rput[c](10.05,-1.3){\scriptsize$\lambda_k$}
\rput[l](14,-1){$\lambda$}
\rput[l](13.2,2.2){$y=\frac{\lambda}{8}$}
\rput[c](4.2,3.8){$y$}
\end{pspicture}}
\caption{Detail of the function $g(\lambda)$ in the interval $I_k = (2k\pi,2k\pi+2\pi)$.
}\label{fig:detailg}
\end{figure}
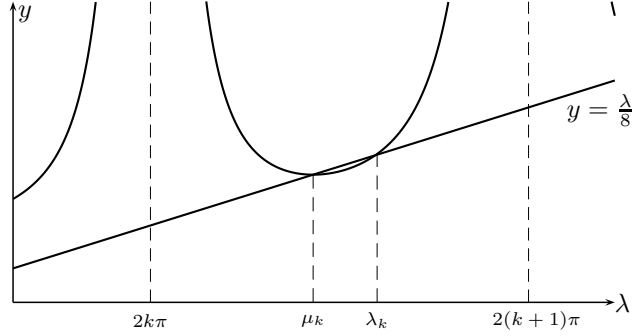

\begin{proposition}\label{g}
Let $k\in \mathbb{Z}$ and $I_k=(2k\pi,2k\pi+2\pi)$. Then:
\begin{itemize}
\item[1.] $g(\la)=-g(-\la)$ and $g(\la)>0$ if $\la>0$;
\item[2.] $|g|$ is strictly convex on each interval $I_k;$ 
\item[3.]  if $\mu_k$ is the point of minimum of $|g|$ on $I_k$, we have $g(\mu_k)=\frac{\mu_k}{8}<\frac{(2k+1)\pi}{8}$;
\item[4.]  $g(|\la|)>\frac{|\la|}{8}-\frac{\pi}{8}$ for every $\lambda$.

\end{itemize}\end{proposition}

\subsection{Decomposition of the fiber}
%In this section we give a description of the fibers of the exponential map for the contact Carnot group $G$. 
We introduce the notation:
\begin{equation}
\Gamma(p)=\hat{E}^{-1}(x,z), \qquad p=(x,z).
\end{equation}
Since $p\neq p_0$, then $\Gamma(p)$ consists of all geodesics ending at $p$.
Given $\alpha_1, \ldots, \alpha_k$ we define:
\begin{equation}
\Lambda_j=\{\text{poles of $\lambda \mapsto g(\la \al_j)$}\},\qquad \Lambda=\bigcup_{j=1}^k\Lambda_j\qquad\text{and}\qquad I_0=\{j\,|\, x_j=0\}.
\end{equation}
Prop. \ref{propo2} implies that, if $(u,\lambda) \in \Gamma(p)$, then: 
\begin{equation}\label{I0} 
L(\la):=\{j \mid \la \in \Lambda_j\}\subseteq I_0.
\end{equation}

\begin{proposition}[Characterization of the fiber]\label{t:fiber}
Let $p=(x,z)\in G$, $p\neq (0,0)$. The set $\Gamma(p)$ consists of the points $(u, \la)$ such that $x_j=I(\la \al_j)u_j$ for every $j=1, \ldots, k$ and:
\begin{equation}\label{descript} 
z=\sum_{j\notin I_0}\al_jg(\la \al_j)\|x_j\|^2+\frac{1}{2\la }\sum_{j\in I_0}\|u_j\|^2.
\end{equation}
\end{proposition}
\begin{proof}
The condition on the $x_j$'s is given by Prop. \ref{propo:exp} and it remains to understand the equation for $z$ in \eqref{exp}. Now we can decompose the summation in the terms defining $z$ as:
\begin{equation}
\label{parti}z=\sum_{j\notin I_0}\left(\frac{\la \al_j -\sin(\la \al_j)}{2\la^2\al_j}\right)\|u_j\|^2+\sum_{j\in  I_0}\left(\frac{\la \al_j -\sin(\la \al_j)}{2\la^2\al_j}\right)\|u_j\|^2.
\end{equation}
If $j\notin I_0$ then $j\notin L(\la)$ by \eqref{I0} and Prop. \ref{propo2} allows to write:
\begin{equation}
\left(\frac{\la \al_j -\sin(\la \al_j)}{2\la^2\al_j}\right)\|u_j\|^2=\frac{\al_j}{8}\frac{\la \al_j-\sin(\la \al_j)}{\sin \left(\frac{\la \al_j}{2}\right)}\|x_j\|^2.
\end{equation}
On the other hand the sum $\sum_{j\in  I_0}\left(\frac{\la \al_j -\sin(\la \al_j)}{2\la^2\al_j}\right)\|u_j\|^2$ can be split as:
\begin{equation}
\sum_{j\in  I_0\cap L(\la)}\left(\frac{\la \al_j -\sin(\la \al_j)}{2\la^2\al_j}\right)\|u_j\|^2+\sum_{j\in  I_0\cap L(\la)^{c}}\left(\frac{\la \al_j -\sin(\la \al_j)}{2\la^2\al_j}\right)\|u_j\|^2.
\end{equation}
The second summation is zero, because for a $j\notin L(\la)$ the matrix $I(\la \al_j)$ is invertible and $u_j=I(\la \al_j)x_j=0$. By~$\eqref{I0}$, the index set for the first summation equals $L(\la)$ itself. Moreover, for each term $j \in L(\lambda)$ we have $\la \al_j\in 2\pi\mathbb{Z}\setminus\{0\}$ and, for some $k_j \in \Z\setminus\{0\}$:
\begin{equation}
\frac{\la \al_j -\sin(\la \al_j)}{2\la^2\al_j}= \frac{2\pi k_j -\sin(2\pi k_j)}{2\la (2\pi k_j)}=\frac{1}{2\la}.
\end{equation}
Substituting what we got into \eqref{parti} we finally obtain:
\begin{equation}
z=\sum_{j\notin I_0}\al_jg(\la \al_j)\|x_j\|^2+\frac{1}{2\la }\sum_{j\in I_0}\|u_j\|^2.\qedhere
\end{equation}
\end{proof}

We decompose $\Gamma(p)$ into two closed disjoint subsets, reflecting its ``discrete'' and ``continuous'' part.
We set indeed $\Gamma(p)=\Gamma_0(p)\cup \Gamma_{\infty}(p)$ where: 
\begin{equation}
\Gamma_0(p)=\left\{(u,\la)\in \Gamma(p)\,\bigg|\, \sum_{j\in I_0}\|u_j\|^2= 0\right\}\qquad \text{and}\qquad \Gamma_\infty(p)=\Gamma_0(p)^c.
\end{equation}
The next theorem clarifies the subscripts and the terminology ``discrete'' and ``continuous'' part.
\begin{theorem}\label{thm:disjoint}
If $p\neq p_0$, the set $\Gamma_0(p)$ is finite and $\Gamma_\infty(p)$ is a closed set homeomorphic to:
\begin{equation}
\Gamma_\infty(p) \simeq \bigcup_{\lambda \in \Lambda_p} S^{2N(\lambda)-1}, \qquad N(\lambda)=\sum_{j \in L(\lambda)} n_j,
\end{equation}
where:
\begin{equation}\label{lap}
\Lambda_p=\left\{\la\in \Lambda \, \bigg|\,\left(z - \sum_{j\notin I_0}\al_jg(\la \al_j)\|x_j\|^2\right)\lambda > 0\right\}.
\end{equation}
Moreover the energy function $J$ is constant on each component of $\Gamma(p)$. 
\end{theorem}
\begin{remark}\label{r:emptyvanish}
By definition, $\Gamma_\infty(p) \neq \emptyset$ implies $I_0\neq \emptyset$. Thus, a necessary condition for occurrence of families of geodesics ending at $p=(x,z)$ is that some of the components $x_j$ must vanish.
\end{remark}
\begin{proof}
We start noticing that if $(u,\la)\in \Gamma_0(p)$ then all the $u_j$'s are determined. In fact if $j\notin I_0$ then, by~\eqref{I0}, $j \notin L(\lambda)$, $I(\al_j \la)$ is invertible and $u_j=I(\al_j \la)^{-1}x_j$; if $j\in I_0$, then the condition $\sum_{j\in I_0}\|u_j\|^2= 0$ implies $u_j=0$.

Consider now the projection $q$ onto the $\la$-axis:
\begin{equation} 
q:T^*_0G\to \R,\qquad (u,\la)\mapsto \la.
\end{equation}
By the above discussion $q|_{\Gamma_0(p)}$ is one-to-one onto its image $q(\Gamma_0(p))$ and it is enough to show that this last set is discrete. To this end we notice that by Prop. \ref{t:fiber} if $(u, \la)\in\Gamma_0(p)$ then:
\begin{equation}\label{add} 
z=\sum_{j\notin I_0}\al_j g(\al_j\la)\|x_j\|^2.
\end{equation}
The set of solutions in $\la$ of this equation coincides with $q(\Gamma_0(p))$ and is discrete: $(x,z)$ is fixed, the function $g$ is strictly convex (by Prop. \ref{g}) and a linear combination of strictly convex functions is still strictly convex (on the domains of definition). Since the set of solutions of \eqref{add} has no accumulation points, $q(\Gamma_0(p))$ is closed and $\Gamma_0(p)=q^{-1}(q(\Gamma_0(p)))$ is closed as well.

We prove that $\Gamma_0(p)$ is finite. If $x\neq 0$ the cardinality of $\Gamma_0(p)$ is bounded by Thm. \ref{thm:upper} below; if $x=0$ then equation \eqref{descript} reduces to $z=\frac{1}{2\la}\sum_{j\in I_0}\|u_j\|^2$ and since $\Gamma_0$ is defined by $\sum_{j\in I_0}\|u_j\|^2=0$, it implies $z=0$ as well, contradicting the assumption $p\neq p_0$.

Now we turn to $\Gamma_\infty(p)$. For each fixed $\la \in q(\Gamma_\infty(p))$ consider the fiber of the projection (the set of pairs $(u, \la)\in \Gamma_{\infty}(p)$). We show that $\la \in \Lambda_p$ and that the fiber is a sphere. By Prop.~\ref{t:fiber}, this is the set of $u \in \R^{2n}$ such that $x_j = I(\lambda\alpha_j) u_j$ for every $j=1,\ldots,k$ and:
\begin{equation}\label{eq:above}
\frac{1}{2\lambda} \sum_{j \in I_0}\|u_j\|^2 = z - \sum_{j \notin I_0} \alpha_j g(\lambda\alpha_j)\|x_j\|^2.
\end{equation}
Now, if $j\notin L(\la)$, then $u_j$ is fixed by the value of $x_j$ (since $I(\al_j \la)$ is invertible). For the remaining ones the only constraint comes from Eq.~\eqref{eq:above}. Consider the summation in the l.h.s. Notice that $L(\lambda) \subseteq I_0$, but if $j \in I_0\cap L(\lambda)^c$ then $u_j = 0$. Therefore:
\begin{equation}\label{eq:linkindex}
\sum_{j\in I_0}\|u_j\|^2 =\sum_{j\in L(\lambda)}\|u_j\|^2.
\end{equation}
In particular, since $(u,\lambda) \in \Gamma_\infty(p)$ this implies that $L(\lambda)$ must be non-empty, namely $\lambda \in \Lambda$. Moreover Eq.~\eqref{eq:above} reduces to:
\begin{equation}\label{eq:below}
\frac{1}{2\lambda} \sum_{j \in L(\lambda)}\|u_j\|^2 = z - \sum_{j \notin I_0} \alpha_j g(\lambda\alpha_j)\|x_j\|^2.
\end{equation}
The r.h.s. of the above equation has the same sign of $\lambda$. Thus $\lambda \in \Lambda_p$ and $q^{-1}(\lambda)$ is a sphere of dimension $2N(\lambda)-1$. 

Finally $q$ is surjective over $\Lambda_p$. In fact, for any $\lambda \in \Lambda_p$, we choose for $j \in L(\lambda)$, $u_j$ that satisfies~\eqref{eq:below}, and for $j\notin L(\lambda)$ we set $u_j= I(\alpha_j\la)^{-1}x_j$. The point $(u,\lambda) \in \Gamma_\infty(p)$ by construction.

The image $q(\Gamma_\infty(p))$ is discrete, as it is contained into $\Lambda$ (and has no accumulation points, since $\Lambda$ itself has no accumulation points). Thus $q(\Gamma_\infty(p))$ is closed and $\Gamma_\infty(p)$ is closed as well.

Since the energy of a geodesic $(u,\la)$ is given by $\|u\|^2/2$, it is constant on each component.
\end{proof}

%\begin{remark}
%Rephrasing the above theorem, we have essentially proven that $\Gamma(p)$ is a smooth submanifold of $T_0^*M$; this manifold can be decomposed into the union its zero-dimensional components, $\Gamma_0(p)$, and the union of the components of dimension at least one. Since the energy function $J$ is constant on each component, these is really a union of ``critical manifolds''.
%%The next section will be devoted to control the topological complexity of the two sets $\Gamma_0$ and $\Gamma_\infty$: we will show that if $x\neq 0$ then $\Gamma_0$ consists indeed of finitely many points and $\Gamma_\infty$ by finitely many spheres, the number of points (resp. spheres) being bounded by a linear function in $\frac{|z|}{\|x\|^2}.$
%%
%%Notice that $\Gamma_\infty(p)$ contains a $(2N(\lambda)-1)$-dimensional sphere for any $\lambda \in \Lambda_p$.
%\end{remark}
				% main contents, structure of the fiber
%!TEX root = enumerative-v5.tex

\section{Upper bounds}\label{s:upperbound}
Let us introduce the following ``counting'' functions $\hat{\nu}, \hat{\beta}:G\to \R\cup\{\infty\}$:
\begin{equation}
\hat{\nu}(p)=\#\Gamma(p) \qquad \textrm{and}\qquad \hat{\beta}(p)=b\left(\Gamma(p)\right),
\end{equation}
where $b(X)$ denotes the sum of the Betti numbers of $X$ (which might as well be infinite a priori).
\begin{remark}The Betti numbers $b_i(X)$ of a topological space $X$ are the ranks of $H_i(X, \mathbb{Z})$ (the homology groups of $X$) and they measure the number of ``holes'' of $X$, see \cite{Hatcher}. For example for a point or a line all $b_i$ are zeroes except $b_0=1$; for a sphere $S^{k} $ they are all zero except $b_0, b_k=1$ (here $k>1$). The sum of the Betti numbers $b(X)$ is sometimes called the \emph{homological complexity} and measure how complicated $X$ is from the topological viewpoint; for example $b(S^k)=2$.
\end{remark}
If $\hat{E}^{-1}(p)$ is finite, then $\hat{\nu}(p)=\hat{\beta}(p)$; on the other hand if a point $p$ has infinitely many geodesics arriving on it $\hat{\nu}(p)=\infty$ and it could either be that they are ``genuinely" infinite, i.e. also $\hat{\beta}(p)=\infty$, or they arrange in finitely many families with controlled topology, i.e. $\hat{\beta}(p)<\infty$.

\begin{theorem}\label{thm:upper}
Let $G$ be a contact Carnot group. Then there exists a constant $R_2$ such that, for every point $p=(x,z)$, with $p\neq p_0$:
\begin{equation}
\hat{\beta}(p)\leq \left( \frac{8k}{\pi} \frac{\al_k}{\al_1^2}\right)\frac{|z|\phantom{^2}}{\|x\|^2}+R_2.
\end{equation}
$R_2$ is homogeneous of degree $0$ in the singular values $\alpha_1<\cdots<\alpha_k$ of $A$. In particular, if $x=(x_1,\ldots,x_k)$ has all components different from zero, then $\Gamma(p)=\Gamma_0(p)$ and:
\begin{equation}
\hat{\nu}(p)\leq \left(  \frac{8k}{\pi} \frac{\al_k}{\al_1^2}\right)\frac{|z|\phantom{^2}}{\|x\|^2}+R_2.
\end{equation}
\begin{remark}\label{r:convention}
Thus, whenever at least one $x_j$ is not zero, the topology of $\Gamma(p)$ is finite; if $z\neq 0$ and $x=0$, then the above formulas are meaningful in the sense that $\frac{|z|}{0}=\infty$.
\end{remark}
\begin{proof}
The decomposition of Thm. \ref{thm:disjoint} implies:
\begin{equation}
b \left(\Gamma(p)\right)=b \left(\Gamma_0(p)\right)+b \left(\Gamma_\infty(p)\right).
\end{equation}
Let us start with $b(\Gamma_0(p))$. Since $\Gamma_0(p)$ consists of points, then $b(\Gamma_0(p))=\#\Gamma_0(p)$ and:
\begin{equation} \label{sol}
\# \Gamma_0(p) =\#\left \{\lambda\, \bigg|\, z=\sum_{j\notin  I_0}\al_jg(\la \al_j)\|x_j\|^2\right\}.
\end{equation}
We recall that $I_0=\{j\, |\,x_j=0\}$ and distinguish two cases. 

1. If $I_0=\{1, \ldots, k\}$ (i.e. $x=0$), then $\Gamma_0(p)$ is empty: in fact from~\eqref{sol} we obtain that also $z=0$, contradicting the assumption $p \neq p_0$.

2. If $I_0\subsetneq \{1, \ldots,k\}$ (at least one $x_j\neq 0$), then property 4 of Prop. \ref{g} implies:
\begin{equation}
|z|=\left|\sum_{j\notin I_0}\al_jg(\la \al_j)\|x_j\|^2\right|>\frac{|\la|}{8} \sum_{j\notin I_0}\al_j^2\|x_j\|^2-\frac{\pi}{8}\sum_{j\notin I_0}\al_j\|x_j\|^2,
\end{equation}
or, equivalently:
\begin{equation}\label{modla} 
|\la|< \frac{8 |z|}{\sum_{j\notin I_0}\al_j^2\|x_j\|^2}+\frac{\pi\sum_{j\notin I_0}\al_j\|x_j\|^2}{\sum_{j\notin I_0}\al_j^2\|x_j\|^2} \leq  \frac{8 |z|}{\al_1^2\|x\|^2}+\frac{\pi \al_k}{\al_1^{2}}=:\rho,
\end{equation}
where in the last inequality we have used the fact that $\|x\|^2=\sum_{j\notin I_0}\|x_j\|^2$. The number of solutions of \eqref{sol} is the number of intersections of the horizontal line $w=z$ with the graph of:
\begin{equation}\label{eq:G0}
G_0(\la)=\sum_{j\notin I_0}\al_jg(\la \al_j)\|x_j\|^2,
\end{equation}
in the $(\lambda, w)$-plane, with the restriction $|\la|<\rho$ we found in \eqref{modla}. The function $G_0$ is itself strictly convex, and the number of points of intersections of $w=z$ with its graph is:
\begin{equation}
b\left(\Gamma_0(p)\right)\leq 2\#\{\textrm{poles of $G_0$ on the interval $(0, \rho)$}\}+1.
\end{equation}
Since the function $G_0$ has poles exactly on the sets $\Lambda_j=\{\la\neq 0\,|\, \lambda \al_j\in 2\pi \mathbb{Z},\, j\notin I_0\}$, we obtain:
\begin{equation}\label{eq:bound0}
\begin{aligned}
 b(\Gamma_0(p)) \leq 2\sum_{j\notin I_0}\left\lfloor\frac{\rho \al_j}{2\pi}\right\rfloor+1 &
\leq  2\sum_{j\notin I_0}\left\lfloor\frac{4\al_j|z|}{\pi\al_1^2\|x\|^2}+\frac{\al_k \al_j}{2 \al_1^2}\right\rfloor+1 \leq \\
& \leq (k -\# I_0)\frac{8}{\pi}\frac{\alpha_k}{\alpha_1^2} \frac{|z|\phantom{^2}}{\|x\|^2} + r_0,
\end{aligned}
\end{equation}
where $r_0$ is a bounded remainder (homogeneous of degree $0$ in the singular values) given by:
\begin{equation}
 r_0 = (k-\# I_0)\frac{\alpha_k^2}{\alpha_1^2} +1.
\end{equation}
Let us consider now $b(\Gamma_\infty(p))$. By Thm.~\ref{thm:disjoint}, $\Gamma_\infty(p)$ is a disjoint union of spheres, one sphere for each point $\lambda \in \Lambda_p$, where:
\begin{equation}
\Lambda_p=\{\la\in \Lambda  \mid (z - G_0(\lambda))\lambda > 0 \}.
\end{equation}
Since the total Betti number of sphere is $2$ (independently on the dimension), we have:
\begin{equation}\label{eq:factor2}
b(\Gamma_\infty(p)) = b\left(\bigcup_{\lambda \in \Lambda_p} S^{2N(\lambda)-1} \right) = 2 \# \Lambda_p .
\end{equation}
We assume $z \geq 0$ for simplicity. This implies $\lambda >0$. Moreover, if $\lambda \in \Lambda_p \subseteq \Lambda$, then $\lambda$ must belong to the complement of the set of poles of the function $G_0$, namely
\begin{equation}
\lambda \in \Lambda_0:=\bigcup_{j \in I_0} \Lambda_j  = \bigcup_{j \in I_0} \frac{2\pi}{\alpha_j} \Z \setminus\{0\} \subseteq \Lambda. 
\end{equation}
Thus we finally rewrite:
\begin{equation}\label{eq:lap2}
\Lambda_p =\{\la \in \Lambda_0 \mid \lambda >0, \quad z > G_0(\lambda)\}.
\end{equation}
It only remains to estimate the cardinality of $\Lambda_p$. We distinguish again two cases.

1. $I_0=\{1, \ldots, k\}$ (i.e. $x=0$). By our assumption $p \neq p_0$ it follows that $z>0$. Moreover, in this case $G_0(\lambda) \equiv 0$ and $\Lambda_0 = \Lambda$. Therefore $\Lambda_p = \Lambda$ is infinite and $\Gamma_\infty(p)$ consists of infinitely many spheres, thus $b(\Gamma_\infty(p))=\infty$.

2. $I_0 \subsetneq\{1, \ldots, k\}$. In this case we have to count the $\bar{\la}>0$, such that:
\begin{equation}\label{abc}
z>\sum_{j\notin I_0}\al_jg(\bar{\la} \al_j)\|x_j\|^2, \qquad \text{with } \bar\la \in \Lambda_0.
\end{equation}
Arguing exactly as in~\eqref{modla} we obtain that:
\begin{equation}
|\bar\lambda| <  \frac{8 |z|}{\al_1^2\|x\|^2}+  \frac{\pi \al_k}{\al_1^2} := \rho.
\end{equation}
Thus the number of $\bar\la$ satisfying \eqref{abc} is bounded by the (finite) number of elements $\bar\la \in \Lambda_0$ in the interval $(0,\rho)$ (arguing as in~\eqref{eq:bound0}):
\begin{equation}\label{empty} 
\sum_{j\in I_0}\left\lfloor\frac{\rho \al_j}{2\pi}\right\rfloor\leq  \#I_0 \frac{4}{\pi} \frac{\al_k}{\al_1^2} \frac{|z|\phantom{^2}}{\|x\|^2}+\frac{\#I_0}{2}\frac{\al_k^2}{\al_1^2} ,
\end{equation}
Combining this with \eqref{eq:factor2} we get:
\begin{equation}
b\left(\Gamma_\infty(p)\right)\leq  \#I_0 \frac{8}{\pi} \frac{\al_k}{\al_1^2}\frac{|z|\phantom{^2}}{\|x\|^2}+r_\infty,
\end{equation}
where $r_\infty$ is a bounded remainder (homogeneous of degree $0$ in the singular values) given by:
\begin{equation}
r_\infty = \#I_0 \frac{\alpha_k^2}{\alpha_1^2}.
\end{equation}

Finally, since the union $\Gamma_0(p)\cup \Gamma_\infty(p)$ is disjoint and closed, we obtain:
\begin{equation}\label{eq:sumupper}
\begin{aligned}
b\left(\Gamma(p)\right)& =b\left(\Gamma_0(p)\right)+ b\left(\Gamma_\infty(p)\right) \leq \\ 
&\leq (k-\#I_0) \frac{8}{\pi} \frac{\al_k}{\al_1^2} \frac{|z|\phantom{^2}}{\|x\|^2}+ \#I_0 \frac{8}{\pi} \frac{\al_k}{\al_1^2}\frac{|z|\phantom{^2}}{\|x\|^2}+r_0 + r_\infty = \\
&= \left(k \frac{8}{\pi} \frac{\al_k}{\al_1^2}\right)\frac{|z|\phantom{^2}}{\|x\|^2} + R_2,
\end{aligned}
\end{equation}
where $R_2$ is a bounded remainder (homogeneous of degree $0$ in the singular values) given by:
\begin{equation}
 R_2 = r_0+r_\infty = k \frac{\alpha_k^2}{\alpha_1^2}.
\end{equation}
Notice that if all $x_j\neq0$, then $I_0 = \emptyset$ and $\Gamma(p) = \Gamma_0(p)$, which is finite.
\end{proof}
\end{theorem}
\begin{remark}
Eq. \eqref{eq:sumupper} splits clearly the contribution to the topology into two pieces:
\begin{equation}
b\left(\Gamma_0(p)\right)\leq(k-\#I_0) \frac{8}{\pi} \frac{\al_k}{\al_1^2}\frac{|z|\phantom{^2}}{\|x\|^2}+r_0\quad \textrm{and}\quad b\left(\Gamma_\infty(p)\right)\leq\#I_0 \frac{8}{\pi} \frac{\al_k}{\al_1^2}\frac{|z|\phantom{^2}}{\|x\|^2}+ r_\infty,
\end{equation}
where we interpret the r.h.s. with the convention of Remark~\ref{r:convention}.
\end{remark}

\begin{example}[Heisenberg, conclusion]
The Jacobian of the exponential map in $\mathbb{H}_3$ can be computed explicitly using \eqref{exp} (for the general contact case, see \cite[Lemma 38]{ABBHausdorff}):
\begin{equation}
\det \left(d_{(u,\la)}\hat{E}\right)=-\frac{\|u\|^2(\la \sin \la +2 \cos \la -2)}{\la^4}.
\end{equation}
Setting to zero the previous equation we find critical points of $\hat{E}$:
\begin{equation}
\text{crit}(\hat{E})=\underbrace{\left\{\|u\|^2=0\right\}}_{A}\cup \underbrace{\left\{\la=2k\pi, k\neq 0\right\}}_{R}\cup\underbrace{\left\{\la \mid  \tfrac{\la}{2}=\tan \tfrac{\la}{2}, \, \lambda \neq 0\right\}}_{B}.
\end{equation}

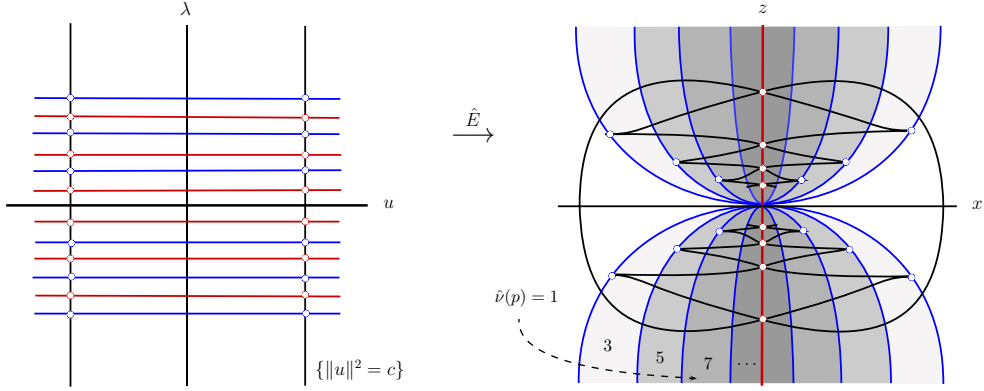
\begin{figure}[t]
\centering
\scalebox{0.6} % Change this value to rescale the drawing.
{
\begin{pspicture}(0,-4.3122654)(21.91586,4.332266)
\definecolor{color22253}{rgb}{0.8,0.0,0.0}
\definecolor{color5388b}{rgb}{0.9647058823529412,0.9568627450980393,0.9568627450980393}
\definecolor{color5075b}{rgb}{0.8,0.8,0.8}
\definecolor{color4546b}{rgb}{0.7098039215686275,0.7058823529411765,0.7058823529411765}
\definecolor{color4225}{rgb}{0.2,0.2,1.0}
\definecolor{color4225b}{rgb}{0.6,0.6,0.6}
\definecolor{color59453b}{rgb}{1.0,0.0,0.2}
\psbezier[linewidth=0.04,linecolor=blue,fillstyle=solid,fillcolor=color5388b](12.62,3.6872)(12.56,-1.4493213)(20.72,-1.6722656)(20.7,3.7077343)
\psbezier[linewidth=0.04,linecolor=blue,fillstyle=solid,fillcolor=color5075b](13.92,3.6877344)(13.82,-1.472966)(19.54,-1.6322656)(19.54,3.6877344)
\psbezier[linewidth=0.04,linecolor=blue,fillstyle=solid,fillcolor=color4546b](14.9,3.6877344)(14.94,-1.5011253)(18.4,-1.5722656)(18.44,3.6877344)
\psbezier[linewidth=0.04,linecolor=color4225,fillstyle=solid,fillcolor=color4225b](16.04,3.6877344)(16.14,-1.5829973)(17.36,-1.4922656)(17.46,3.6877344)
\psbezier[linewidth=0.04,linecolor=blue,fillstyle=solid,fillcolor=color5388b](12.68,-4.2117314)(12.62,0.92478997)(20.78,1.1477344)(20.76,-4.2322655)
\psbezier[linewidth=0.04,linecolor=blue,fillstyle=solid,fillcolor=color5075b](13.98,-4.2122655)(13.88,0.94843477)(19.6,1.1077343)(19.6,-4.2122655)
\psbezier[linewidth=0.04,linecolor=blue,fillstyle=solid,fillcolor=color4546b](14.96,-4.2122655)(15.0,0.97659415)(18.46,1.0477344)(18.5,-4.2122655)
\psbezier[linewidth=0.04,linecolor=blue,fillstyle=solid,fillcolor=color4225b](16.04,-4.2122655)(16.1,1.0584661)(17.36,0.94773436)(17.44,-4.2122655)
\psline[linewidth=0.055999998cm,linecolor=color22253](16.76,3.7277343)(16.74,-4.2722654)
\usefont{T1}{ptm}{m}{n}
\rput(21.515274,-0.26726562){\Large $x$}
\usefont{T1}{ptm}{m}{n}
\rput(16.765274,4.0927343){\Large $z$}
\rput(4.0,-0.2722656){\psaxes[linewidth=0.04,labels=none,ticks=none,ticksize=0.10583333cm,showorigin=false](0,0)(-4,-4)(4,4)}
\psline[linewidth=0.04cm,linecolor=color22253](0.6,0.047734376)(7.42,0.067734376)
\psline[linewidth=0.04cm,linecolor=blue](0.6,0.48773438)(7.42,0.50773436)
\psline[linewidth=0.04cm,linecolor=color22253](0.62,0.8477344)(7.42,0.8477344)
\psline[linewidth=0.04cm,linecolor=blue](0.58,1.3277344)(7.38,1.3077344)
\psline[linewidth=0.04cm,linecolor=color22253](0.62,1.7077343)(7.38,1.6677344)
\psline[linewidth=0.04cm,linecolor=blue](0.62,2.1077344)(7.4,2.0877345)
\psline[linewidth=0.04cm,linecolor=color22253](0.62,-0.6322656)(7.38,-0.6322656)
\psline[linewidth=0.04cm,linecolor=blue](0.64,-1.0922656)(7.4,-1.0922656)
\psline[linewidth=0.04cm,linecolor=color22253](0.62,-1.4522656)(7.38,-1.4522656)
\psline[linewidth=0.04cm,linecolor=blue](0.58,-1.8722656)(7.38,-1.8722656)
\psline[linewidth=0.04cm,linecolor=color22253](0.64,-2.2722657)(7.4,-2.2922657)
\psline[linewidth=0.04cm,linecolor=blue](0.62,-2.6722655)(7.4,-2.6722655)
\usefont{T1}{ptm}{m}{n}
\rput(3.9652734,4.0927343){\Large $\lambda$}
\usefont{T1}{ptm}{m}{n}
\rput(8.475273,-0.24726562){\Large $u$}
\usefont{T1}{ptm}{m}{n}
\rput(10.338941,1.5377344){\huge $\stackrel{\hat{E}}{\longrightarrow}$}
\psbezier[linewidth=0.04](20.76,-0.29226562)(20.72,2.0277343)(19.46,2.9677343)(16.74,2.2477343)(14.02,1.5277344)(14.14,1.3277344)(13.42,1.3077344)(12.7,1.2877344)(15.84,1.3877344)(16.74,1.0677344)(17.64,0.74773437)(19.014715,0.68741465)(18.42,0.6677344)(17.825285,0.64805406)(16.98,0.62773436)(16.72,0.5477344)(16.46,0.46773437)(16.371223,0.26152986)(15.86,0.2877344)(15.348777,0.3139389)(17.02,0.18773438)(17.06,0.12773438)
\psline[linewidth=0.04cm](6.62,-4.2722654)(6.62,3.7077343)
\psline[linewidth=0.04cm](1.42,-4.2922654)(1.42,3.7677343)
\psdots[dotsize=0.16,linecolor=blue,fillstyle=solid,dotstyle=o](6.62,2.1077344)
\psdots[dotsize=0.16,linecolor=blue,fillstyle=solid,dotstyle=o](1.42,2.1077344)
\psdots[dotsize=0.16,linecolor=blue,fillstyle=solid,dotstyle=o](1.42,1.3477343)
\psdots[dotsize=0.16,linecolor=blue,fillstyle=solid,dotstyle=o](6.62,1.3077344)
\psdots[dotsize=0.16,linecolor=blue,fillstyle=solid,dotstyle=o](6.62,0.50773436)
\psdots[dotsize=0.16,linecolor=blue,fillstyle=solid,dotstyle=o](6.62,-1.0922656)
\psdots[dotsize=0.16,linecolor=blue,fillstyle=solid,dotstyle=o](6.62,-1.8922657)
\psdots[dotsize=0.16,linecolor=blue,fillstyle=solid,dotstyle=o](1.44,-1.1122656)
\psdots[dotsize=0.16,linecolor=blue,fillstyle=solid,dotstyle=o](6.64,-2.6722655)
\psdots[dotsize=0.16,linecolor=blue,fillstyle=solid,dotstyle=o](1.42,0.48773438)
\psdots[dotsize=0.16,linecolor=blue,fillstyle=solid,dotstyle=o](1.44,-1.8522656)
\psdots[dotsize=0.16,linecolor=blue,fillstyle=solid,dotstyle=o](1.42,-2.6922655)
\psdots[dotsize=0.16,linecolor=color22253,fillstyle=solid,dotstyle=o](6.62,1.6677344)
\psdots[dotsize=0.16,linecolor=color22253,fillstyle=solid,dotstyle=o](6.62,0.8477344)
\psdots[dotsize=0.16,linecolor=color22253,fillstyle=solid,dotstyle=o](6.62,0.067734376)
\psdots[dotsize=0.16,linecolor=color22253,fillstyle=solid,dotstyle=o](6.62,-0.6322656)
\psdots[dotsize=0.16,linecolor=color22253,fillstyle=solid,dotstyle=o](6.62,-1.4322656)
\psdots[dotsize=0.16,linecolor=color22253,fillstyle=solid,dotstyle=o](6.62,-2.2522657)
\psdots[dotsize=0.16,linecolor=color22253,fillstyle=solid,dotstyle=o](1.42,-2.2722657)
\psdots[dotsize=0.16,linecolor=color22253,fillstyle=solid,dotstyle=o](1.42,-1.4522656)
\psdots[dotsize=0.16,linecolor=color22253,fillstyle=solid,dotstyle=o](1.42,-0.6522656)
\psdots[dotsize=0.16,linecolor=color22253,fillstyle=solid,dotstyle=o](1.42,0.067734376)
\psdots[dotsize=0.16,linecolor=color22253,fillstyle=solid,dotstyle=o](1.42,0.8677344)
\psdots[dotsize=0.16,linecolor=color22253,fillstyle=solid,dotstyle=o](1.42,1.6877344)
\psdots[dotsize=0.16,linecolor=blue,fillstyle=solid,dotstyle=o](13.38,1.3077344)
\psdots[dotsize=0.16,linecolor=blue,fillstyle=solid,dotstyle=o](18.6,0.68773437)
\psdots[dotsize=0.16,linecolor=blue,fillstyle=solid,dotstyle=o](15.78,0.30773437)
\psdots[dotsize=0.16,linecolor=blue,fillstyle=solid,dotstyle=o](18.68,-1.2522656)
\psline[linewidth=0.04cm,fillcolor=color59453b](12.22,-0.29226562)(21.06,-0.29226562)
\usefont{T1}{ptm}{m}{n}
\rput(7.835273,-3.9272656){\Large $\{\|u\|^2=c\}$}
\psbezier[linewidth=0.04](20.76,-0.29226562)(20.72,-2.7722657)(19.48,-3.5322657)(16.76,-2.8122656)(14.04,-2.0922656)(14.18,-1.8722656)(13.48,-1.8122656)(12.78,-1.7522656)(15.86,-1.9522656)(16.76,-1.6322656)(17.66,-1.3122656)(19.034716,-1.2519459)(18.44,-1.2322656)(17.845285,-1.2125853)(17.0,-1.1922656)(16.74,-1.1122656)(16.48,-1.0322657)(16.391222,-0.8260611)(15.88,-0.8522656)(15.368777,-0.8784701)(17.04,-0.75226563)(17.08,-0.6922656)
\psbezier[linewidth=0.04](12.7002,-0.2722656)(12.76,-2.6922655)(14.003459,-3.5322657)(16.730276,-2.8122656)(19.457092,-2.0922656)(19.236391,-1.8922657)(19.98,-1.8522656)(20.72361,-1.8122656)(17.632532,-1.9522656)(16.730276,-1.6322656)(15.82802,-1.3122656)(14.44986,-1.2519459)(15.046065,-1.2322656)(15.642271,-1.2125853)(16.489674,-1.1922656)(16.750326,-1.1122656)(17.010977,-1.0322657)(17.099977,-0.8260611)(17.612482,-0.8522656)(18.124985,-0.8784701)(16.449574,-0.75226563)(16.409473,-0.6922656)
\psbezier[linewidth=0.04](12.7,-0.29226562)(12.7,2.1477344)(14.0,2.9677343)(16.72,2.2477343)(19.44,1.5277344)(19.22,1.4077344)(19.96,1.3877344)(20.06,1.3677344)(17.62,1.3877344)(16.72,1.0677344)(15.82,0.74773437)(14.445285,0.68741465)(15.04,0.6677344)(15.634715,0.64805406)(16.48,0.62773436)(16.74,0.5477344)(17.0,0.46773437)(17.088778,0.26152986)(17.6,0.2877344)(18.111223,0.3139389)(16.44,0.18773438)(16.4,0.12773438)
\psdots[dotsize=0.16,linecolor=color22253,fillstyle=solid,dotstyle=o](16.76,1.0677344)
\psdots[dotsize=0.16,linecolor=color22253,fillstyle=solid,dotstyle=o](16.76,0.5477344)
\psdots[dotsize=0.16,linecolor=color22253,fillstyle=solid,dotstyle=o](16.76,0.16773437)
\psdots[dotsize=0.16,linecolor=blue,fillstyle=solid,dotstyle=o](20.06,1.3877344)
\psdots[dotsize=0.16,linecolor=blue,fillstyle=solid,dotstyle=o](17.62,0.2877344)
\psdots[dotsize=0.16,linecolor=blue,fillstyle=solid,dotstyle=o](14.84,0.68773437)
\psdots[dotsize=0.16,linecolor=color22253,fillstyle=solid,dotstyle=o](16.76,2.2477343)
\psdots[dotsize=0.16,linecolor=blue,fillstyle=solid,dotstyle=o](13.42,-1.8322656)
\psdots[dotsize=0.16,linecolor=blue,fillstyle=solid,dotstyle=o](14.86,-1.2322656)
\psdots[dotsize=0.16,linecolor=blue,fillstyle=solid,dotstyle=o](15.8,-0.8522656)
\psdots[dotsize=0.16,linecolor=color22253,fillstyle=solid,dotstyle=o](16.76,-1.6322656)
\psdots[dotsize=0.16,linecolor=color22253,fillstyle=solid,dotstyle=o](16.76,-1.1122656)
\psdots[dotsize=0.16,linecolor=color22253,fillstyle=solid,dotstyle=o](16.76,-0.75226563)
\psdots[dotsize=0.16,linecolor=blue,fillstyle=solid,dotstyle=o](17.66,-0.8322656)
\psdots[dotsize=0.16,linecolor=blue,fillstyle=solid,dotstyle=o](20.06,-1.8722656)
\psdots[dotsize=0.16,linecolor=color22253,fillstyle=solid,dotstyle=o](16.76,-2.7922657)
\usefont{T1}{ptm}{m}{n}
\rput(13.327334,-3.3822656){\large $3$}
\usefont{T1}{ptm}{m}{n}
\rput(14.487334,-3.6422656){\large $5$}
\usefont{T1}{ptm}{m}{n}
\rput(15.547334,-3.7822657){\large $7$}
\usefont{T1}{ptm}{m}{n}
\rput(16.407333,-3.7822657){\large $\ldots$}
\usefont{T1}{ptm}{m}{n}
\rput(11.527334,-2.3222656){\large $\hat{\nu}(p)=1$}
\psbezier[linewidth=0.03,linestyle=dashed,dash=0.16cm 0.16cm,fillcolor=color59453b,arrowsize=0.05291667cm 2.0,arrowlength=1.4,arrowinset=0.4]{->}(11.36,-2.7922657)(11.26,-3.8922656)(14.14,-4.0338283)(15.3,-4.1122656)
\end{pspicture} 
}
\caption{Qualitative picture of the exponential map for $\mathbb{H}_3$. The critical points are the $\la$-axes $A$, the set $R$ (in red) and the set $B$ (in blue). The broken curve is the section in the $(x,z)$-plane of the image of the cylinder $\{\|u\|^2=c\}$. The number of geodesics to $p$ is constant on each shaded region (the white one is where $\hat{\nu}(p)=1$).
When $c$ varies the blue dots on the right figure (the images of $B\cap \{\|u\|^2=c\}$) ``span'' all the paraboloids $|z|=g(\la_k)\|x\|^2$.} \label{Heisfinal}
\end{figure}

The critical values are the images of these sets. For convenience of notations we label $\la_k$, with $k \in \mathbb{Z} \setminus \{0\}$, the non-zero solutions of $ \tfrac{\la}{2}=\tan \tfrac{\la}{2}$: these numbers, in the case of the Heisenberg group, coincide with the minima of the function $g$ and are of the form $\la_k=(2k+1)\pi+\epsilon_k$. The critical values of $\hat{E}$ decompose into the union of the three sets:
\begin{equation}
\hat{E}(A)=\text{origin},\quad \hat{E}(R)=\text{$z$-axis},\quad \hat{E}(B)=\left \{z=\|x\|^2g(\la_k) \mid k\in \mathbb{Z} \setminus \{0\}\right\}.
\end{equation}
In particular, $\hat{E}(B)$ is a union of paraboloids, and has the following characterization: for $x \neq 0$, we have:
\begin{equation}
\hat\nu(p) = \#\{\lambda \mid z = g(\lambda) \|x\|^2\}.
\end{equation} 
By the properties of $g$, and assuming $z\geq 0$ (resp. $z\leq 0$), two new contributions to $\hat\nu(p)$ appear (or disappear) every time the ratio $|z|/\|x\|^2$ crosses the values $g(\lambda_k)$, for $k \geq 0$ (resp. $k \leq 0$). Thus the function $\hat{\nu}(p)$ ``jumps'' by two every time $p$ crosses $\hat{E}(B)$ transversely (see Fig. \ref{Heisfinal}).

\end{example}
				% upper bound
%!TEX root = enumerative-v5.tex

\section{Lower bounds}\label{s:lower}

%In this section we prove the lower bounds of Thm.s~\ref{thm:orderintro} and \ref{thm:ordertopintro} for the number of geodesics $\hat\nu(p) = \#\Gamma(p)$ and the topology $\hat\beta(p) = b(\Gamma(p))$ in a contact Carnot group $G$, with constants $(k,\vec{n},\vec{\alpha})$. 
According to the decomposition of Sec.~\ref{s:main}, for $p\neq p_0$ we have the following splitting:
\begin{equation}
\Gamma(p) = \Gamma_0(p) \cup \Gamma_\infty(p),
\end{equation}
where $\Gamma_0(p)$ is a finite set and $\Gamma_\infty(p)$ is homeomorphic to a disjoint union of spheres. According to Remark~\ref{r:emptyvanish}, if $p=(x,z)$ is a point with all components $x_j \neq 0$, then $\Gamma_\infty(p) = \emptyset$ (in particular this is the case for a generic point $p$). In this setting we prove the next theorem.

\begin{theorem}[The ``infinitesimal'' lower bound]\label{thm:lower}
Given a contact Carnot group $G$, there exist constants $C_1,R_1$ such that if $p=(x,z)\in G$ has all components $x_j$ different from zero, then:
\begin{equation}
C_1\frac{|z|\phantom{^2}}{\|x\|^2}+R_1\leq \hat{\nu}(p).
\end{equation}
In particular, denoting by $\al_1$ and $\al_k$ the smallest and the largest singular values of $A$:
\begin{equation}\label{numerical}
C_1=\frac{8}{\pi} \frac{\alpha_1}{\alpha_k^2}\sin\left(\frac{\delta \pi}{2}\right)^2\quad \text{with} \quad \delta = \left(\sum_{j=1}^k \frac{\alpha_1}{\alpha_j} \left\lfloor \frac{\alpha_j}{\alpha_1}\right\rfloor\right)^{-1}.
\end{equation}
Moreover, $R_1$ (resp. $C_1$) is homogeneous of degree $0$ (resp. $-1$) in the singular values $\alpha_1,\ldots,\alpha_k$.
\end{theorem}
\begin{proof}
When all the $x_j\neq0$, then $\Gamma(p) = \Gamma_0(p)$. According to Prop.~\ref{t:fiber}, and recalling that $I_0 = \emptyset$, the number of geodesics ending at $p=(x,z)$ is computed by:
\begin{equation}
\hat{\nu}(p) =  \#\{\lambda\mid z = G(\lambda)\}, \qquad G(\lambda):= \sum_{j=1}^k \alpha_j g(\alpha_j \lambda)\|x_j\|^2 .
\end{equation}
The idea of the proof is to build a sequence of values $\hat{\lambda}_n$, growing linearly with $n$, such that $G(\hat{\lambda}_n) \leq c n + d$ for some constants $c$ and $d$. By the strict convexity of $G(\lambda)$, we have at least one contribution to $\hat{\nu}(p)$ for any point $\hat{\lambda}_n$ of the sequence such that $G(\hat{\lambda}_n) < z$. %Then it is sufficient to count the number of points $\hat{\lambda}_n$ of the sequence such that $G(\hat{\lambda}_n) < z$.

Without loss of generality, we assume $z \geq 0$ and then $\lambda \geq 0$. For fixed $0<\delta\leq 1$ and every $j=1, \ldots, k$ define the intervals: 
\begin{equation}
I_{n,j}:=\left[\frac{2n\pi}{\al_j}, \frac{2(n+1)\pi}{\al_j}\right]\qquad \text{and}\qquad \hat{I}_{n,j}:=\left[\frac{2n\pi}{\al_j}+\frac{\delta\pi}{\alpha_j}, \frac{2(n+1)\pi}{\al_j} - \frac{\delta\pi}{\alpha_j}\right].
\end{equation}
Each interval $\hat{I}_{n,j}$ is contained in $I_{n,j}$ and the lengths of these two intervals are (see Fig. \ref{f:figureintervals}):
\begin{equation}
 |I_{n,j}|=\frac{2\pi}{\al_j}=:a_j\qquad \text{and}\qquad |\hat{I}_{n,j}|=a_j(1-\delta).
\end{equation}
The singular values $0<\alpha_1 < \cdots < \alpha_k$ are ordered, then the intervals $I_{n,1}$, for $n \in \N$ are the largest. We also define $y_n:=(2n-\delta)\pi$. Notice that $\frac{y_{n+1}}{\alpha_j}$ is the maximum of the interval $\hat{I}_{n,j}$, and will play an important role in the proof.

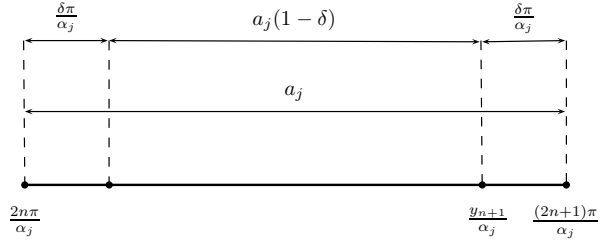
\begin{figure}% Generated with LaTeXDraw 2.0.8
\centering
\scalebox{0.8} % Change this value to rescale the drawing.
{
\begin{pspicture}(0,-1.8880469)(12.42291,1.8880469)
\psline[linewidth=0.04cm](1.5210156,-1.0903516)(10.521015,-1.0903516)
\psdots[dotsize=0.12](1.5210156,-1.0903516)
\psdots[dotsize=0.12](10.521015,-1.0903516)
\psdots[dotsize=0.12](2.9210157,-1.0903516)
\psdots[dotsize=0.12](9.121016,-1.0903516)
\usefont{T1}{ptm}{m}{n}
\rput(1.5224707,-1.7){$\frac{2n\pi}{\al_j}$}
\usefont{T1}{ptm}{m}{n}
\rput(10.51247,-1.7){$\frac{(2n+1)\pi}{\al_j}$}
\usefont{T1}{ptm}{m}{n}
\rput(9.212471,-1.7){$\frac{y_{n+1}}{\al_j}$}
\usefont{T1}{ptm}{m}{n}
\rput(2.2124708,1.6946484){$\frac{\delta\pi}{\al_j}$}
\psline[linewidth=0.02cm,linestyle=dashed,dash=0.16cm 0.16cm](1.5210156,-1.1103516)(1.5010157,1.3096484)
\psline[linewidth=0.02cm,linestyle=dashed,dash=0.16cm 0.16cm](10.521015,-1.0503516)(10.521015,1.3096484)
\psline[linewidth=0.02cm,arrowsize=0.05291667cm 2.0,arrowlength=1.4,arrowinset=0.4]{<->}(1.5210156,0.12964843)(10.501016,0.12964843)
\psline[linewidth=0.02cm,linestyle=dashed,dash=0.16cm 0.16cm](2.9210157,-1.1103516)(2.9210157,1.3096484)
\psline[linewidth=0.02cm,linestyle=dashed,dash=0.16cm 0.16cm](9.121016,-1.0703516)(9.101016,1.3096484)
\psline[linewidth=0.02cm,arrowsize=0.05291667cm 2.0,arrowlength=1.4,arrowinset=0.4]{<->}(2.9210157,1.3096484)(9.081016,1.2896484)
\psline[linewidth=0.02cm,arrowsize=0.05291667cm 2.0,arrowlength=1.4,arrowinset=0.4]{<->}(1.5210156,1.3096484)(2.9010155,1.3096484)
\psline[linewidth=0.02cm,arrowsize=0.05291667cm 2.0,arrowlength=1.4,arrowinset=0.4]{<->}(9.121016,1.2896484)(10.501016,1.3096484)
\usefont{T1}{ptm}{m}{n}
\rput(6.0,0.43464842){$a_j$}
\usefont{T1}{ptm}{m}{n}
\rput(6.0,1.6546484){$a_j(1-\delta)$}
\usefont{T1}{ptm}{m}{n}
\rput(9.81247,1.6746484){$\frac{\delta\pi}{\al_j}$}
\end{pspicture} 
}
\caption{The intervals $\hat{I}_{n,j}\subset I_{n,j}$.}\label{f:figureintervals}
\end{figure}
Each function $\lambda \mapsto g(\alpha_j\lambda)$ is unbounded in the intervals $I_{n,j}$ (it has poles at the extrema), but it is controlled on all the smaller intervals $\hat{I}_{n,j}$, as stated by the next lemma.
\begin{lemma}\label{l:estimates}
There exist constants $c_1(\delta)$, $d_1(\delta)$ such that, for $j=1,\ldots,k$:
\begin{equation}
g(\alpha_j\lambda) \leq c_1(\delta)n + d_1(\delta), \qquad \forall \lambda \in \hat{I}_{n,j}.
\end{equation}
\end{lemma}
\begin{proof}
By Prop.~\ref{g}, for all $j=1,\ldots,k$ the functions $\lambda \mapsto g(\alpha_j \lambda)$ are strictly convex on the intervals $\hat{I}_{n,j} \subset I_{n,j}$. Each function is clearly unbounded on $I_{n,j}$ but, when restricted on $\hat{I}_{n,j}$, it achieves its maximum value at the point $\tfrac{y_{n+1}}{\alpha_j}$ (i.e. the maximum of the interval $\hat{I}_{n,j}$). Therefore, by explicit evaluation, for all $\lambda \in \hat{I}_{n,j}$ we have:
\begin{equation}
g(\alpha_j\lambda) \leq g(y_{n+1}) = \frac{2\pi}{8\sin(\delta\pi/2)^2} n + \frac{2\pi - \delta \pi + \sin(\delta\pi)}{8 \sin(\delta\pi/2)^2} = c_1(\delta)n+ d_1(\delta). \qedhere
\end{equation}
\end{proof}
The next lemma implies that, for each $n \geq 0$, the large interval $I_{n,1}$ contains at least one point that belongs to all the smaller intervals $\hat{I}_{m_1,1},\ldots,\hat{I}_{m_k,k}$, for some $m_1,\ldots,m_k$.
\begin{lemma}\label{l:claim}
Let $\hat{I}_j=\bigcup_{m\geq 0}\hat{I}_{m, j}$ for all $j=1,\ldots,k$. If $0 < \delta \leq 1$ is small enough then:
\begin{equation} 
\forall n \geq 0\qquad  I_{n, 1}\cap\bigcap_{j=1}^k\hat{I}_j\neq \emptyset.
\end{equation}
\end{lemma}
\begin{proof}
We argue by contradiction. Assume there exists $n\geq 0$ such that for all $\lambda \in I_{n,1}$ we can find $j\in\{1, \ldots, k\}$ with $\lambda \notin \hat{I}_{j}$. This implies:
\begin{equation}\label{intest} 
s_n:=\int_{I_{n,1}}\#\left\{j\mid\lambda\in \hat{I}_ j\right\}dz\leq (k-1)a_1.
\end{equation}
On the other hand the above integral equals:
\begin{equation}
s_n=\sum_{j=1}^k|\hat{I}_{j}\cap I_{n,1}|\geq k a_1- \sum_{j=1}^k\frac{2\delta\pi}{\alpha_j}\left\lfloor \frac{\al_j}{\al_1}\right\rfloor\geq (k-1)a_1+\left( a_1 - \delta \sum_{j=1}^k\frac{2\pi}{\alpha_j}\left\lfloor \frac{\al_j}{\al_1}\right\rfloor\right).
\end{equation}
Recalling that $a_1 = \frac{2\pi}{\alpha_1}$, if we choose
\begin{equation}\label{eq:bestdelta}
0< \delta < \left(\sum_{j=1}^k \frac{\alpha_1}{\alpha_j} \left\lfloor \frac{\alpha_j}{\alpha_1}\right\rfloor\right)^{-1},
\end{equation}
we obtain $s_n> (k-1)a_1$, contradicting \eqref{intest}.
\end{proof}
The next lemma builds a sequence $\hat{\lambda}_n$ where the behaviour of $G(\lambda)$ is controlled.
\begin{lemma}\label{l:estimates2}
There exists an unbounded, increasing sequence $\{\hat\lambda_n\in I_n\}_{n\in \N}$ and constants $c_k(\delta)$, $d_k(\delta)$ such that:
\begin{equation}
\sum_{j=1}^kg(\hat{\lambda}_n \al_j)\leq c_k(\delta) n + d_k(\delta).
\end{equation}
\end{lemma}
\begin{proof}
By Lemma~\eqref{l:claim}, for all $n \geq 0$ there is a point $\hat{\lambda}_n\in I_{n,1} \cap \hat{I}_{m_1,1}\cap \hat{I}_{m_2, 2}\cap\cdots \cap \hat{I}_{m_k,k}$, for some $m_1,\ldots,m_k$. This sequence is unbounded and increasing. By construction, $m_1 = n$ and
\begin{equation}
m_j \leq \left\lfloor\frac{(n+1)\alpha_j}{\alpha_1}\right\rfloor \leq \frac{\alpha_k}{\alpha_1} n + 2\frac{\alpha_k}{\alpha_1}, \qquad j=2,\ldots,k.
\end{equation}
By the estimates of Lemma~\ref{l:estimates}, we have
\begin{equation}
\sum_{j=1}^kg(\hat{\lambda}_n \al_j)\leq \sum_{j=1}^k \left( c_1(\delta)m_j+ d_1(\delta)\right) \leq \underbrace{\left[c_1(\delta) k \frac{\alpha_k}{\alpha_1}\right]}_{c_k(\delta)} n  + \underbrace{\left[2c_1(\delta)(k-1) \frac{\alpha_k}{\alpha_1} + k d_1(\delta) \right]}_{d_k(\delta)}.\qedhere
\end{equation}
\end{proof}
We are now ready for the computation of the lower bound for $\hat{\nu}(p)$. Indeed
\begin{equation}\label{solutions}
\hat\nu(p)=\#\{\la \mid z=G(\la)\}, \qquad G(\la):= \sum_{j=1}^k \al_j g(\la \al_j)\|x_j\|^2.
\end{equation}
By Prop.~\ref{g}, each function $\lambda \mapsto g(\alpha_j \lambda)$ is strictly convex in the intervals $I_{n,j}$, for $n \in \N$, and has poles at the extrema of $I_{n,j}$ (excluding $\lambda = 0$), i.e. the discrete set $\Lambda_j$. Then also $G(\la)$ is a strictly convex function in each interval in which it is defined, with poles at $\Lambda = \cup_{j=1}^k \Lambda_j$. 

Consider the sequence $\hat{\lambda}_n$ of Lemma~\ref{l:estimates2}. %If $G(\hat\la_n)<z$ the number of solutions of \eqref{solutions} is at least the number of poles of $G$ that are smaller than $\hat{\la}_n$, and since $\hat{\la}_n\in I_n$ this number is at least $2n$.
There are at least $2$ solutions contributing to Eq.~\eqref{solutions} for any value $\hat{\lambda}_n$ such that $G(\hat{\lambda}_n) < z$. This follows by strict convexity of $G$ in the interval between two successive poles containing $\hat\lambda_n$. The only exception to this rule is when $\hat{\lambda}_n$ belongs to $I_{0,k}$: in this case there is only $1$ solution (there is no pole at $\lambda =0$).
We have:
\begin{equation}
G(\hat{\lambda}_n) \leq \alpha_k \|x\|^2 \sum_{j=1}^k g(\alpha_j \hat{\lambda}_n) \leq \alpha_k \|x\|^2 \left[c_k(\delta) n +d_k(\delta)\right].
\end{equation}
Thus, $\hat{\lambda}_0$ gives a contribution of $1$ to $\hat\nu(p)$, while each point $\hat{\lambda}_n$, with $n \geq 1$ of the sequence, such that $G(\hat{\lambda}_n) < z$, give a contribution of $2$ to $\hat\nu(p)$. Taking in account all the contributions:
\begin{equation}
\hat{\nu}(p) \geq 2\left\lfloor\frac{1}{c_k(\delta)} \frac{|z|\phantom{^2}}{\alpha_k\|x\|^2} -\frac{d_k(\delta)}{c_k(\delta)} \right\rfloor + 1  \geq \frac{2}{\alpha_k c_k(\delta)} \frac{|z|\phantom{^2}}{\|x\|^2} -2\frac{d_k(\delta)}{c_k(\delta)}-1.
\end{equation}
Plugging in the constants obtained above, we obtain:
\begin{equation}
\hat{\nu}(p) \geq C(\delta)\frac{|z|\phantom{^2}}{\|x\|^2} + R(\delta),
\end{equation}
with:
\begin{equation}
C(\delta):= \frac{8}{\pi} \frac{\alpha_1}{\alpha_k^2} \sin\left(\frac{\delta \pi}{2}\right)^2, \qquad R(\delta):=4\frac{1-k}{k} + \frac{\alpha_1}{\alpha_k}\frac{\delta \pi - \sin(\delta\pi) - 2\pi}{\pi}-1.
\end{equation}
Both $C(\delta)$ and $R(\delta)$ are non-decreasing functions of $\delta$, for $0<\delta \leq 1$, thus the best estimate is given by the values at the largest $\delta$. According to~\eqref{eq:bestdelta} this value is:
\begin{equation}
\delta_M:= \left(\sum_{j=1}^k \frac{\alpha_1}{\alpha_j}\left\lfloor\frac{\alpha_j}{\alpha_1}\right\rfloor\right)^{-1}.
\end{equation}
Notice that $C_1:=C(\delta_M)$ is homogeneous of degree $-1$ w.r.t. the singular values $\alpha_1,\ldots,\alpha_k$, while $R_1:=R(\delta_M)$ is homogeneous of degree $0$.
\end{proof}

The previous theorem holds if all the $x_j$ are different from zero. When some of the $x_j = 0$, continuous families might appear, but the topology of these families is controlled.

\begin{theorem}[The ``infinitesimal'' lower bound for the topology]\label{thm:lowertop} Let $G$ be a contact Carnot group.
There exist constants $R_1',C_1'$ such that for every $p=(x,z)\in G$ with $p\neq p_0$:
\begin{equation}
C_1'\frac{|z|\phantom{^2}}{\|x\|^2}+R_1'\leq \hat{\beta}(p).
\end{equation}
In particular, denoting by $\al_1$ and $\al_k$ the smallest and the largest singular values of $A$:
\begin{equation}
C_1'=\frac{8}{\pi} \frac{\alpha_1}{\alpha_k^2}\sin\left(\frac{\delta' \pi}{2}\right)^2\quad\textrm{with}\quad \delta' = \left(\sum_{j \notin I_0} \frac{\alpha_1}{\alpha_j} \left\lfloor \frac{\alpha_j}{\alpha_1}\right\rfloor\right)^{-1}.
\end{equation}
Moreover, $R_1'$ (resp. $C_1'$) is homogeneous of degree $0$ (resp. $-1$) in the singular values $\alpha_1,\ldots,\alpha_k$.
\end{theorem}
\begin{proof}
Recall that $I_0 = \{j\mid x_j =0\}$. If $I_0 = \emptyset$, then the statement reduces to Thm.~\ref{thm:lowertop} since $\Gamma(p) = \Gamma_0(p)$ is finite and $\hat{\nu}(p) = \# \Gamma(p) = b(\Gamma(p)) = \hat\beta(p)$. Then assume $I_0 \neq \emptyset$. By Thm.~\ref{t:fiber}, $\Gamma(p) =\Gamma_0(p) \cup \Gamma_\infty(p)$ and:
\begin{equation}
\hat{\beta}(p) = b(\Gamma(p)) \geq b(\Gamma_0(p)) = \# \Gamma_0(p).
\end{equation}
In particular $\Gamma_0(p)$ is in one-to-one correspondence with its projection on the $\lambda$ component, since all the $u_j$ are uniquely determined by the point $p =(x,z)$ once $\lambda$ is known. Therefore
\begin{equation}
\# \Gamma_0(p) = \#\{\lambda \mid z = G_0(\lambda)\}, \qquad G_0(\lambda): = \sum_{j\notin I_0} \alpha_j g(\alpha_j \lambda) \|x_j\|^2.
\end{equation}
Now we only have to bound from below the number of solutions of $z = G_0(\lambda)$. The proof is analogous to the one of Thm.~\ref{thm:lower}, where only the indices $j \notin I_0$ appear.
\end{proof}
				% lower bound
%!TEX root = enumerative-v5.tex

\section{Isometries and families of geodesics}

\subsection{Isometries of the Heisenberg group}\label{s:isoheis}

Isometries are distance-preserving transformations and, in Carnot groups, are smooth (see \cite{LeDonnesmooth}). The set of all sub-Riemannian isometries $\ISO(G)$ of a Carnot group is a Lie group, and any isometry is the composition of a group automorphism and a group translation (see \cite{Hamenstadt,LeDonneiso}). Here we consider the subgroup $\ISO_0(G)$ of isometries that fix the identity and we denote this subgroup simply $\ISO(G)$.

%In this section we focus on the case $G =\mathbb{H}_{2n+1}$ (we refer to the notation of Sec.~\ref{s:prel}, and Example~\ref{ex:heisenberg}). Then we use this result to study the isometry group of any contact Carnot group.
\begin{lemma}\label{l:isoH}
The isometry group of $\mathbb{H}_{2n+1}$ is:
\begin{equation}
\ISO(\mathbb{H}_{2n+1}) = \{(M,\theta) \mid \theta = \pm 1,\; MM^* = \I_{2n}, \; MJM^* = \theta J \},
\end{equation}
with the action of $\ISO(\mathbb{H}_{2n+1})$ on $\mathbb{H}_{2n+1}$ given by:
\begin{equation}
(M,\theta) \cdot (x,z) = (Mx,\theta z).
\end{equation}
Moreover:
\begin{equation}
\ISO(\mathbb{H}_{2n+1}) \simeq \mathrm{O}(2n) \cap \Sp(2n) \rtimes \Z_2 \simeq \U(n) \rtimes \Z_2.
\end{equation}
\end{lemma}
\begin{proof}
A diffeomorphism is an isometry of Carnot groups fixing the identity if and only if is a Lie group isomorphism. In particular, it is induced by Lie algebra isomorphisms $\phi : \mathfrak{h}_{2n+1} \to \mathfrak{h}_{2n+1}$ that are orthogonal transformations on the first layer. Since $\phi$ is a Lie algebra isomorphism, it preserves the stratification. Then we can write $\phi = (M,\theta) \in \mathrm{O}(2n)\times \R$, such that
\begin{equation}
\phi(f_i) = \sum_{j=1}^{2n} M_{ji} f_j, \qquad \phi (f_0) = \theta f_0.
\end{equation}
The isomorphism condition $[\phi(f_i),\phi(f_j)] = J_{ij} \phi(f_0)$ implies:
\begin{equation}
MJM^* = \theta J.
\end{equation}
It follows that $\theta^{2} = 1$. Then:
\begin{equation}
\ISO(\mathbb{H}_{2n+1}) = \{(M,\theta) \mid \theta = \pm 1,\; MM^* = \I_{2n}, \; MJM^* = \theta J \}.
\end{equation}
This Lie algebra isomorphism generates a Lie group isomorphism that, in exponential coordinates, reads $(M,\theta) \cdot (x,z) = (Mx,\theta z)$. Let $\ISO(\mathbb{H}_{2n+1})_+ \lhd \ISO(\mathbb{H}_{2n+1})$ be the normal subgroup:
\begin{equation}
\ISO(\mathbb{H}_{2n+1})_{+} := \{(M,1) \mid MM^* = \I, \; MJM^* = J\} \simeq \mathrm{O}(2n)\cap \Sp(2n).
\end{equation}
Moreover, let $K$ be any matrix such that $KJK^* = -J$. Then, let :
\begin{equation}
H:=\{(\I,1),(K,-1)\} \simeq \Z_2
\end{equation}
be another subgroup of $\ISO(\mathbb{H}_{2n+1})$. Any element of $\ISO(\mathbb{H}_{2n+1})$ can be written uniquely as the product $mh$ of an element of $m \in \ISO(\mathbb{H}_{2n+1})_+$ and an element of $h \in H$. Thus the map $mh \mapsto (m,h)$ is a group isomorphism:
\begin{equation}
\ISO(\mathbb{H}_{2n+1}) = \ISO(\mathbb{H}_{2n+1})_{+} \rtimes H,
\end{equation}
where $H$ acts on $\ISO(\mathbb{H}_{2n+1})_+$ with the adjoint action. As we observed $\ISO(\mathbb{H}_{2n+1})_{+} \simeq \mathrm{O}(2n)\cap \Sp(2n)$ and $H \simeq \Z_2$, thus
\begin{equation}
\ISO(\mathbb{H}_{2n+1}) \simeq \mathrm{O}(2n)\cap \Sp(2n) \rtimes \Z_2.
\end{equation}
\begin{remark}
With this identification, the action of $\varphi : \Z_2 \to \mathrm{Aut}(\mathrm{O}(2n)\cap \Sp(2n))$ is:
\begin{equation}
\varphi(1)M = M, \qquad \varphi(-1) M = KMK^*,
\end{equation}
the product on $\mathrm{O}(2n)\cap \Sp(2n) \rtimes \Z_2$ reads:
\begin{equation}
(M,\theta) (M',\theta') = (M \varphi(\theta)M',\theta\theta'),
\end{equation}
and the action of $\mathrm{O}(2n)\cap \Sp(2n) \rtimes \Z_2$ on $\mathbb{H}_{2n+1}$ is:
\begin{equation}
(M,\theta) \cdot (x,z) = \begin{cases} (Mx,z) & \theta = 1, \\ (MKx,-z) & \theta = -1. \end{cases}
\end{equation}
\end{remark}
Finally, to see that $\mathrm{O}(2n) \cap \Sp(2n) \simeq \U(n)$, write $M \in \mathrm{GL}(2n,\R)$ as 
%\begin{equation}
$M =\left(\begin{smallmatrix}
A & B \\ C & D 
\end{smallmatrix}\right)$.
%\end{equation}
Then $M \in \mathrm{O}(2n) \cap \Sp(2n) $ if and only if:
\begin{equation}
M = \begin{pmatrix}
A & B \\ -B & A
\end{pmatrix}, \qquad AA^* +BB^* = \I_n, \quad AB^* - BA^* = 0.
\end{equation}
Thus the map $M \mapsto  A + iB$ is the group isomorphism $\mathrm{O}(2n) \cap \Sp(2n) \simeq \U(n)$.
\end{proof}

\subsubsection{Stabilizers of points}
Let $p \in \mathbb{H}_{2n+1}$. We restrict our attention to the connected component $\ISO(\mathbb{H}_{2n+1})_+$ that contains the identity. As in the proof of Lemma~\ref{l:isoH}, we identify:
\begin{equation}
\ISO(\mathbb{H}_{2n+1})_+ = \U(n).
\end{equation}
With this identification, the action $\rho:\U(n) \times \mathbb{H}_{2n+1} \to \mathbb{H}_{2n+1}$ is
\begin{equation}
\rho(A+iB, (x,z)) = (Mx,z), \qquad M= \begin{pmatrix}
A & B \\ -B & A
\end{pmatrix}.
\end{equation}
What is the stabilizer subgroup $\ISO_p(\mathbb{H}_{2n+1}) \subseteq \ISO(\mathbb{H}_{2n+1})_+$ that fixes $p \in \mathbb{H}_{2n+1}$?

\begin{lemma}\label{l:isopointH}
Let $p=(x,z) \in \mathbb{H}_{2n+1}$. Then
\begin{equation}
\ISO_p(\mathbb{H}_{2n+1}) \simeq \begin{cases}
\U(n) & x = 0, \\
\U(n-1) & x \neq 0.
\end{cases}
\end{equation}
\end{lemma} 
\begin{proof}
Let $A+iB \in \U(n)$. Let $p=(x,z) \in \mathbb{H}_{2n+1}$, with $x \neq 0$ and write $x=(v,w)$ with $v,w \in \R^n$. Then
\begin{equation}
\rho(A+iB,p) = p \Longleftrightarrow Mx =x \Longleftrightarrow (A+iB)(v-iw) = v-iw.
\end{equation}
This means that $A+iB$ must be a unitary matrix with a prescribed eigenvector $v-iw$ with eigenvalue $1$. This identifies a copy of $\U(n-1) \subset \U(n)$ that fixes $p$. On the other hand, if $x =0$, the point $p=(0,z)$ is fixed for any element of $\ISO(\mathbb{H}_{2n+1})_+$.
\end{proof}

\subsubsection{Stabilizers of geodesics}
Let $\gamma(t)$ be the geodesic with initial covector $(u,\lambda) \in T_{0}^*\mathbb{H}_{2n+1}$. What is the subgroup $\ISO_\gamma(\mathbb{H}_{2n+1}) \subset \ISO(\mathbb{H}_{2n+1})_+$ that fixes the whole geodesic? Recall that
\begin{equation}
\gamma(t) = \left( \int_0^t e^{-\lambda J \tau} ud\tau, -\frac{1}{2}  \int_0^t \langle e^{-\lambda J \tau} u, J \int_0^\tau e^{-\lambda J s} u ds \rangle d\tau  \right).
\end{equation}
\begin{lemma}\label{l:isogeoH}
	Let $(u,\lambda) \in T_0^*\mathbb{H}_{2n+1}$ be the initial covector of the geodesic $\gamma$. Then
\begin{equation}
\ISO_\gamma(\mathbb{H}_{2n+1}) \simeq 
 \begin{cases} \U(n) & u =0, \\ \U(n-1) & u \neq 0. \end{cases}
\end{equation}
\end{lemma} 
\begin{proof}
To stabilize $\gamma$ is equivalent to stabilize its ``horizontal'' component. Indeed, let $A+iB \in \U(n)$ be an isometry and $\gamma(t) = (x(t),z(t))$. Then $\rho(A + iB,(x(t),z(t))) = (x(t),z(t))$ if and only if $Mx(t) = x(t)$ for all $t$.
For $u \neq 0$, take one derivative w.r.t. $t$ at $t=0$; we obtain $Mu = u$, as in the proof of Lemma~\ref{l:isopointH}. This identifies a subgroup $\U(n-1) \subset \U(n)$. This condition also implies also $Mx(t) = x(t)$. In fact:
\begin{equation}
Mx(t) = M\int_0^t e^{-\tau \lambda J} u = \int_0^t e^{-\tau \lambda J}M u = x(t), 
\end{equation}
where we used the fact that, being an isometry, $MJ = JM$. Thus, in this case, $\ISO_\gamma(\mathbb{H}_{2n+1}) = \U(n-1)$. When $u=0$ the geodesic is the trivial one, and is stabilized by the whole $\U(n)$.
\end{proof}
\begin{remark}\label{r:samesub}
Notice that in this case $u=0$ if and only the geodesic is trivial $\gamma(t)\equiv 0$. When $u \neq 0$ two possibilities can occur: 1) $x\neq 0$, in which case  $\ISO_\gamma(\mathbb{H}_{2n+1})= \ISO_p(\mathbb{H}_{2n+1})\simeq \U(n-1)$; 2) $x=0$ and  the subgroup $\ISO_\gamma(\mathbb{H}_{2n+1})\simeq \U(n-1)$ is properly contained in $\ISO_p(\mathbb{H}_{2n+1})\simeq \U(n)$.
\end{remark}

\subsubsection{Isometrically equivalent geodesics}

\begin{definition}\label{d:equiv} Let $\gamma_1,\gamma_2$ be geodesics with the same endpoints. We say that $\gamma_1$ is \emph{isometrically equivalent} to $\gamma_2$ if there exists $g\in \textrm{ISO}(G)$ such that $\gamma_1=g \gamma_2$.
\end{definition}

Let $p \in \mathbb{H}_{2n+1}$, and $\gamma$ be a normal geodesic such that $\gamma(0) = 0$ and $\gamma(1) = p$. By acting with $\ISO_p(\mathbb{H}_{2n+1})$ we obtain families of isometrically equivalent by construction. Still, since $\ISO_\gamma(\mathbb{H}_{2n+1}) \subseteq \ISO_p(\mathbb{H}_{2n+1})$, we may obtain in this way non-distinct geodesics. To avoid duplicates, we have to take the quotient w.r.t. the subgroup $\ISO_\gamma(\mathbb{H}_{2n+1})$. 

Let $X_\gamma$ be the set of geodesics isometrically equivalent to a given one $\gamma$. This is a homogeneous space w.r.t. the action of $\ISO_p(\mathbb{H}_{2n+1})$. From Lemma~\ref{l:isopointH} and~\ref{l:isogeoH} we obtain the structure of $X_\gamma$.
\begin{proposition}\label{p:isoequivH}
Let $\gamma$ be a geodesic such that $\gamma(0) = 0$ and $\gamma(1) = p$, with initial covector $(u,\lambda) \in T_0^*\mathbb{H}_{2n+1}$. Then:
\begin{equation}
X_\gamma = \ISO_p(\mathbb{H}_{2n+1}) / \ISO_\gamma(\mathbb{H}_{2n+1}) \simeq \begin{cases} S^{2n-1} & u \neq 0, \lambda \in 2\pi \Z\setminus\{0\}, \\ {1} & \text{otherwise}.
\end{cases}
\end{equation}
\end{proposition}
\begin{proof}
If $u=0$, then $\gamma(t) = 0$ is the trivial geodesic. In this case $X_0$ is just a point (the trivial geodesic). Then we may assume $u \neq 0$. Let $p=(0,z)$. An explicit computation leads to
\begin{equation}
0 = \int_0^1 e^{-\tau \lambda J} u \Longleftrightarrow \lambda = 2m\pi, \quad m \in \Z\setminus\{0\}.
\end{equation}
Then, when $\lambda = 2m\pi$ (and $u \neq 0$), according to Lemma~\ref{l:isopointH} and~\ref{l:isogeoH} we have:
\begin{equation}
\ISO_p(\mathbb{H}_{2n+1}) / \ISO_\gamma(\mathbb{H}_{2n+1}) = \U(n)/\U(n-1) \simeq S^{2n-1}.
\end{equation}
If $\lambda \neq 2m\pi$, then $p=(x,z)$ with $x \neq 0$. According to Lemma~\ref{l:isopointH} and~\ref{l:isogeoH} (see also Remark~\ref{r:samesub}) we have $\ISO_p(\mathbb{H}_{2n+1}) = \ISO_\gamma(\mathbb{H}_{2n+1}) = \U(n-1)$. Thus their quotient is the trivial group.
\end{proof}
\begin{remark}
In fact, in terms of the endpoint, the only possibility for having a family of  isometrically equivalent geodesics ending at $p$ is that $x =0$ zero. In fact, $\lambda = 2m\pi$ and $u\neq 0$ if and only if $p = (0,z)$ with $z \neq 0$. This means that for \emph{non-vertical} points $p$, all the geodesics connecting $p$ with the origin are not isometrically equivalent, while if $p=(0,z)$ is \emph{vertical}, for any geodesic $\gamma$ connecting $p$ with the origin, we have a family of distinct geodesics (all with the same energy) diffeomorphic to $S^{2n-1}$.
\end{remark}

\subsection{Isometries of contact Carnot groups}\label{s:isocarnot}

%We now pass to the study of the isometry group of a contact Carnot group $G$. %By ``isometry'' we still mean distance preserving transformations that fix the origin.
\begin{lemma}\label{l:isoG}
The isometry group of the contact Carnot group $G$ with parameters $(k,\vec{n},\vec{\alpha})$ is:
\begin{equation}
\ISO(G) = \{(M_1,\ldots,M_k,\theta) \mid \theta =\pm 1,\; M_i M_i^* = \I_{2n_i}, \; M_i J_{n_i}M_i^* = \theta J_{n_i}\},
\end{equation}
with the action of $\ISO(G)$ on $G$ given by:
\begin{equation}
(M_1,\ldots,M_k,\theta) \cdot (x_1,\ldots,x_k,z) = (M_1 x_1,\ldots,M_k x_k,\theta z).
\end{equation}
Moreover this group is isomorphic to:
\begin{align}
\ISO(G) & \simeq \mathrm{O}(2n_1) \cap \Sp(2n_1) \times \cdots \times \mathrm{O}(2n_k) \cap \Sp(2n_k) \rtimes \Z_2 \\
& \simeq \U(n_1) \times\cdots\times \U(n_k) \rtimes \Z_2.
\end{align}
\end{lemma}
\begin{proof}
The proof is analogous to the one of Lemma~\ref{l:isoH}, after splitting the equations in the real eigenspaces associated with the eigenvalues of $A$.
\end{proof}
\begin{remark}\label{r:restriction}
As above, we restrict to the connected component $\ISO(G)_+$. We identify:
\begin{equation}
\ISO(G)_+ = \U(n_1)\times \cdots\times \U(n_k).
\end{equation}
With this identification, the action $\rho: \ISO(G)_+ \times G \to G$ is
\begin{equation}
\rho(A_1+iB_1, \ldots, A_k+iB_k,(x_1,\ldots,x_k,z)) = (M_1 x_1,\ldots,M_k x_k,z),
\end{equation}
where $A_j+iB_j \in \U(n_j)$ for all $j=1,\ldots,k$ and $
M_j := \left(\begin{smallmatrix}
A_j & B_j \\ -B_j & A_j \end{smallmatrix}\right)$.
\end{remark}

\subsubsection{Stabilizers of points}

For $p \in G$, let $\ISO_p(G) \subseteq \ISO(G)_+$ its stabilizer.
%Which is the subgroup $\ISO_p(G) \subseteq \ISO(G)_+$ that fixes $p \in G$?
\begin{lemma}\label{l:isopointG}
Let $p = (x_1,\ldots,x_k,z) \in G$. Then:
\begin{equation}
\ISO_{p}(G) = \begin{cases} \U(n_1) & x_1 = 0 \\ \U(n_1-1) & x_1 \neq 0 \end{cases} \times \cdots \times \begin{cases} \U(n_k) & x_k = 0 \\ \U(n_k-1) & x_k \neq 0 \end{cases}.
\end{equation}
\end{lemma}
\begin{proof}
By Remark~\ref{r:restriction}, the isometry $(A_1+iB_1,\ldots,A_k+iB_k) \in \ISO(G)_+$ fixes $p =(x_1,\ldots,x_k,z)$ if and only if $(A_j+iB_j)x_j = x_j$ for all $j=1,\ldots,k$. This means that $A_j+iB_j\in \ISO(\mathbb{H}_{2n_j+1})_+$ fixes the point $p_j := (x_j,z) \in \mathbb{H}_{2n_j+1}$, for all $j=1,\ldots,k$. Then :
\begin{equation}
\ISO_p(G) = \ISO_{p_1}(\mathbb{H}_{2n_1+1}) \times \cdots \times \ISO_{p_k}(\mathbb{H}_{2n_k+1}),
\end{equation}
and the result follows from Lemma~\ref{l:isopointH}.
\end{proof}

\subsubsection{Stabilizers of geodesics}

Let $(u,\lambda) \in T_{0}^*G$. Let $\gamma$ be the associated geodesic, such that $\gamma(0) =0$ and $p=\gamma(1)$. What is the stabilizer subgroup of the geodesic $\ISO_\gamma(G) \subseteq \ISO_p(G)$? As usual, we write $u=(u_1,\ldots,u_k)$, with $u_i \in \R^{2n_i}$.  Accordingly $\gamma(t) = (x_1(t),\ldots,x_k(t),z(t))$, with $x_i(t) \in \R^{2n_i}$. In particular:
\begin{align}
x_i(t) & = \int_0^t e^{-\tau\lambda\alpha_i J } u_i d\tau, \\
z(t) & = -\frac{1}{2}\sum_{i=1}^k \int_0^t \langle e^{-\tau\lambda\alpha_i J}u_i, \alpha_i J\int_{0}^\tau e^{-s \lambda\alpha_i J} u_i ds \rangle d\tau,
\end{align}
where we suppressed the explicit mention of the dimension of the matrices $J_{n_i}$.
Notice that $u =0$ if and only if the geodesic is the trivial one $\gamma(t) \equiv 0$.

\begin{lemma}\label{l:isogeoG}
Let $(u_1,\ldots,u_k,\lambda) \in T_0^*G$ the initial covector of the geodesic $\gamma$. Then:
\begin{equation}
\ISO_\gamma(G) = \begin{cases} \U(n) & u_1 =0 \\ \U(n-1) & u_1 \neq 0 \end{cases} \times \cdots\times \begin{cases} \U(n) & u_k =0 \\ \U(n-1) & u_k \neq 0 \end{cases}.
\end{equation}
\end{lemma} 
\begin{proof}
Let $(A_1+iB_1,\ldots,A_k +iB_k) \in \ISO(G)_+$. According to Remark~\ref{r:restriction}, this isometry fixes the geodesic $(x_1(t),\ldots,x_k(t),z(t))$ if and only if
\begin{equation}
M_j x_j(t) = x_j(t) , \qquad M_j = \begin{pmatrix}
A_j & B_j \\ -B_j & A_j
\end{pmatrix}, \qquad \forall j=1,\ldots,k.
\end{equation}
This implies that $A_j+iB_j \in \ISO(\mathbb{H}_{2n_j+1})_+$ fixes the geodesic $\gamma_j$ of $\mathbb{H}_{2n_j+1}$ associated with the initial covector $(u_j,\alpha_j\lambda)$. Then:
\begin{equation}
\ISO_\gamma(G) = \ISO_{\gamma_1}(\mathbb{H}_{2n_1+1}) \times \cdots \times \ISO_{\gamma_k}(\mathbb{H}_{2n_k+1}),
\end{equation}
and the result follows from Lemma~\ref{l:isogeoH}.
\end{proof}
%\begin{remark}
%In the proof of Lemma~\ref{l:isogeoG}, starting from a geodesic $\gamma$ in $G$ with initial covector $(u_1,\ldots,u_k,\lambda)$, we built ``auxiliary'' geodesics $\gamma_i$ in $\mathbb{H}_{2n_i+1}$ associated with initial covector $(u_i,\alpha_i \lambda)$. Notice that the points $p_i:=(x_i,z)$ are \emph{not} the final points of the geodesics $\gamma_i$. In fact, if we write $\gamma(t) = (x_1(t),\ldots,x_k(t),z(t))$, we have that $\gamma_i(t) = (x_i(t),z_i(t))$, with
%\begin{equation}
%z(t) = -\frac{1}{2}\sum_{i=1}^k \int_0^t \langle e^{-\tau\lambda\alpha_i J}u_i, \alpha_i J \int_{0}^\tau e^{-s \lambda\alpha_i J} u_i ds \rangle d\tau = \sum_{i=1}^k \alpha_i z_i(t).
%\end{equation}
%\end{remark}
\subsubsection{Isometrically equivalent geodesics}

Let $\gamma$ be a geodesic connecting the origin with a point $p \in G$. Let $(u_1,\ldots,u_k,\lambda)$ be the initial covector of the geodesic, and let $p =(x_1,\ldots,x_k,z)$ its endpoint. Let $X_\gamma$ be the set of geodesic isometrically equivalent to the given one. This is an homogeneous space w.r.t. the action of $\ISO_p(G)$.
\begin{proposition}\label{p:isoequivG}
Let $G$ a contact Carnot group with parameters $(k,\vec{n},\vec{\alpha})$. Let $\gamma$ be a geodesic in $G$ with initial covector $(u_1,\ldots,u_k,\lambda)$, such that $\gamma(0) = 0$ and $\gamma(1) = p$. Then:
\begin{equation}\label{eq:factor}
X_\gamma = \ISO_p(G) / \ISO_\gamma(G) \simeq  X_{\gamma_1} \times \cdots \times X_{\gamma_k},
\end{equation}
where:
\begin{equation}
X_{\gamma_i} :=  \begin{cases}
S^{2n_i-1} & u_i \neq 0,\quad \alpha_i\lambda = 2m_i \pi, \\
1 & \text{otherwise},
\end{cases} \qquad m_i \in \Z \setminus \{0\}.
\end{equation}
\end{proposition}
\begin{proof}
By the proofs of Lemma~\ref{l:isogeoG} and~\ref{l:isopointG} we have
\begin{equation}
\ISO_p(G) = \ISO_{p_1}(\mathbb{H}_{2n_1+1}) \times \cdots \times \ISO_{p_k}(\mathbb{H}_{2n_k+1}),
\end{equation}
\begin{equation}
\ISO_\gamma(G) = \ISO_{\gamma_1}(\mathbb{H}_{2n_1+1}) \times \cdots \times \ISO_{\gamma_k}(\mathbb{H}_{2n_k+1}),
\end{equation}
where $p_i = (x_i,z) \in \mathbb{H}_{2n_i+1}$ and $\gamma_i$ is the normal geodesic in $\mathbb{H}_{2n_i+1}$ with initial covector $(u_i,\alpha_i\lambda) \in T_0^*\mathbb{H}_{2n_i+1}$, for all $i=1,\ldots,k$. 
Since each factor $\ISO_{\gamma_i}(\mathbb{H}_{2n_i+1})$ is a subgroup of the corresponding $\ISO_{p_i}(\mathbb{H}_{2n_i+1})$, the quotient of the direct product of Lie groups factors in the direct product of the quotients:
\begin{equation}
\cart_{i=1}^k \ISO_{p_i}(\mathbb{H}_{2n_i+1}) / \cart_{i=1}^k \ISO_{\gamma_i}(\mathbb{H}_{2n_i+1}) = \cart_{i=1}^k \ISO_{p_i}(\mathbb{H}_{2n_i+1}) / \ISO_{\gamma_i}(\mathbb{H}_{2n_i+1}).
\end{equation}
Then:
\begin{equation}
X_\gamma = \cart_{i=1}^k X_{\gamma_i}, \qquad X_{\gamma_i} = \ISO_{p_i}(\mathbb{H}_{2n_i+1}) / \ISO_{\gamma_i}(\mathbb{H}_{2n_i+1}).
\end{equation}
Recall that the geodesic $\gamma_i$ of $\mathbb{H}_{2n_i+1}$ is associated with initial covector $(u_i,\alpha_i \lambda)$ by construction. Thus for each factor $X_{\gamma_i}$ we proceed as in the proof of Prop.~\ref{p:isoequivH} and we obtain the result.
\end{proof}

\begin{example}
Prop.~\ref{p:isoequivG} implies that the for generic geodesic (i.e. with generic initial covector), the manifold $X_\gamma$ of distinct isometrically equivalent geodesics is trivial.
\end{example}

%\begin{example}
%When $k=1$, $G$ is isometric to $\mathbb{H}_{2n+1}$. Then Prop.~\ref{p:isoequivG} recovers Prop.~\ref{p:isoequivH}.
%\end{example}

\begin{example}
Consider the generic Carnot group $G$, associated with the generic choice of $A \in \so(2n)$. In this case $n =k$, $n_1=\cdots=n_k = 1$ and all the $\alpha_i$ are not commensurable. The only geodesics $\gamma$ admitting a non-trivial manifold $X_\gamma$ of distinct isometrically equivalent geodesics are those with initial covector $(u,\lambda)$, such that $\lambda = 2m\pi/\alpha_{i}$ for a unique $i \in  \{1,\ldots,n\}$ and $m \in \Z \setminus\{0\}$. In this case: $X_\gamma  \simeq S^1$. In fact $\alpha_j \lambda \neq 2m_j\pi$ for all $j \neq i$ otherwise some $\alpha_j$ would be commensurable with $\alpha_i$. Then there is only one factor in Eq.~\eqref{eq:factor}. Notice that these geodesics have endpoint $(x,z)$, with $z \neq 0$, $x_i = 0$.
\end{example}

\subsection{Families of isometrically non-equivalent geodesics}\label{s:families}

We ended the previous section discussing families $X_\gamma$ of \emph{isometrically equivalent} geodesics connecting two points. These families arose as homogeneous space w.r.t. the stabilizer $\ISO_p(G)$ of the final point $p = \gamma(1)$ of a fixed geodesic $\gamma$. In this section we adopt a different point of view, and we investigate how many \emph{isometrically non-equivalent} geodesics join two points in $G$.

It may well be that some of the families of Thm.~\ref{thm:disjoint} contain geodesics that are isometrically equivalent, as in Def.~\ref{d:equiv}. This is the case in the Heisenberg groups $\mathbb{H}_{2n+1}$, where all the families are $S^1$ of equivalent geodesics. Is this the correct picture for any contact Carnot group? In other words, are the spheres appearing in $\Gamma_\infty(p)$ families of isometrically equivalent geodesics? In general the answer is no, and the picture is more complicated as shown in the next theorem.

%\begin{remark}
%The initial covector of isometrically equivalent geodesics have all the same $\lambda$ component. For this reason there are no families of isometrically equivalent geodesics in $\hat{E}^{-1}(x,z)_{\text{inv}}$. In fact, each initial covector in this set has a different $\lambda$ component.
%\end{remark}

\begin{theorem}\label{t:noninvnoneqfamilies}
Let $G$ be a contact Carnot group. The set $\bar \Gamma_{\infty}(p) $ of  equivalence classes of isometrically equivalent geodesics ending at $p \neq p_0$ is homeomorphic to the disjoint union:
\begin{equation}
\bar \Gamma_{\infty}(p)  \simeq\bigcup_{\lambda \in {\Lambda}_p} S_{\geq 0}^{\ell(\lambda)-1}\qquad \ell(\lambda):=\# L(\lambda),
\end{equation}
where $S^n_{\geq 0} = S^{n} \cap \R^{n+1}_{\geq 0}$ is the intersection of the $n$-sphere with the positive quadrant in $\R^{n+1}$ and $\Lambda_p$ is defined in \eqref{lap}.
\end{theorem}
\begin{remark}
%Thus a family $X\subset \Gamma_\infty(p)$ is made of equivalent geodesics if under the above equivalence relation it corresponds to just a point in $\bar\Gamma_\infty(p).$

When all the $\alpha_1,\ldots,\alpha_k$ are pair-wise non-commensurable, then $\# L(\lambda) = 1$ for all $\lambda \in \Lambda_p\subseteq \Lambda$ and $N(\lambda) = 1$.  Thus all the ``continuous'' families in $\Gamma_\infty(p)$ are topologically $S^1$ of isometrically equivalent geodesics. Nevertheless, for \emph{resonant structures} (i.e. when some of the $\alpha_i$ are commensurable) there exist continuous families of non-isometrically equivalent geodesics.
\end{remark}
\begin{proof}
 Fix $\bar\lambda \in \Lambda_p$. Without loss of generality, we can assume that $L(\bar\lambda) = \{1,\ldots,\ell\}$ for $\ell = \#L(\bar\lambda)$. This implies $x_1=\cdots=x_\ell = 0$ by Prop.~\ref{propo2}. From Lemma~\ref{l:isopointG}:
\begin{equation}
\ISO_p(G) = \U(n_1)\times \cdots \times \U(n_\ell) \times \U(n_{\ell+1}-1) \times \cdots \times \U(n_k-1),
\end{equation}
and the action $\rho: \ISO_p(G) \times G \to G$ is:
\begin{equation}
\rho(A_1+iB_1,\ldots,A_k+iB_k,(x_1,\ldots,x_k,z))= (M_1	x_1,\ldots, M_k x_k,z), 
\end{equation}
with $M_i = \left(\begin{smallmatrix}
A_j & B_j \\ -B_j & A_j
\end{smallmatrix}\right)$.

In particular $\ISO_p(G)$ is the subgroup that fixes all the components $x_{\ell+1},\ldots,x_k$ (with no other restriction on the other components). It is easy to check that the action on the initial covector $(u_1,\ldots,u_k,\lambda)$ is exactly the same. In particular, $\ISO_p(G)$ is the subgroup that fixes all the components $u_{\ell+1},\ldots,u_k$ with no other restriction on the other components.

Consider one connected component of $\Gamma_\infty(p)$, given by $\Gamma_\infty(p) \cap \{ \lambda = \bar\lambda\}$. As in the proof of Thm. 22, specifically equation \eqref{eq:below}, and assuming without loss of generality that $L(\bar\lambda) = \{1,\dots,\ell\}$, we have that
\begin{equation}
\Gamma_\infty(p) \cap \{ \lambda = \bar\lambda\} = \left\{ (u_1,\dots,u_\ell) \in \R^{2\ell} \,\Bigg\vert\, \sum_{j \in L(\bar\lambda)} \|u_j\|^2 = c(\bar\lambda) \right\} \simeq S^{2N(\bar\lambda)-1},
\end{equation}
where $c(\bar{\lambda}) > 0$, $\ell = \ell(\bar\lambda) = \# L(\bar \lambda)$, and $N(\bar\lambda) = \sum_{j \in L(\bar\lambda)} n_j$.

The action of $\ISO_p(G)$ on $S^{2N(\bar\lambda)-1}$ is the action of $\U(n_1)\times \cdots \times \U(n_\ell)$, namely each copy of $\U(n_j)$ acts on each component $u_j$ with $j \in L(\bar \lambda)$. Thus consider the map:
\begin{equation}
\xi :  S^{2N(\bar\lambda)-1} \to S^{\ell-1}_{\geq 0} \qquad \xi(u_1,\ldots,u_\ell) := (\|u_1\|,\ldots,\|u_\ell\|).
\end{equation}
%\begin{figure} 
%\psscalebox{0.75	 0.75} % Change this value to rescale the drawing.
%{
%\begin{pspicture}(0,-2.5534792)(12.945714,2.5534792)
%\psline[linecolor=black, linewidth=0.02, arrowsize=0.05291666666666667cm 2.0,arrowlength=1.4,arrowinset=0.0]{->}(0.0,-0.25347915)(5.3,-0.25347915)
%\psline[linecolor=black, linewidth=0.02, arrowsize=0.05291666666666667cm 2.0,arrowlength=1.4,arrowinset=0.0]{->}(2.3,-2.5534792)(2.3,2.3465209)
%\pscircle[linecolor=black, linewidth=0.04, dimen=outer](2.3,-0.25347915){2.0}
%\rput[bl](2.4142857,2.0){$u_2 \in \mathbb{R}^{2n_2}$}
%\rput[bl](4.457143,-0.76776487){$u_1 \in \mathbb{R}^{2n_1}$}
%\psline[linecolor=black, linewidth=0.02, arrowsize=0.05291666666666667cm 2.0,arrowlength=1.4,arrowinset=0.0]{->}(8.9,-1.2534791)(12.3,-1.2534791)
%\psline[linecolor=black, linewidth=0.02, arrowsize=0.05291666666666667cm 2.0,arrowlength=1.4,arrowinset=0.0]{->}(9.3,-1.6534791)(9.3,1.6465209)
%\rput[bl](9.485714,1.389378){$\|u_2\| \in \mathbb{R}$}
%\rput[bl](11.385715,-1.810622){$\|u_1\| \in \mathbb{R}$}
%\psarc[linecolor=black, linewidth=0.04, dimen=outer](9.3,-1.2534791){2.0}{0.0}{90.0}
%\rput[bl](4.142105,0.82020503){$\simeq S^{2n_1+2n_2 -1}$}
%\rput[bl](10.994737,0.009678738){$\simeq S_{\geq 0}^{1}$}
%\rput[bl](7.336842,-0.148216){$\xi$}
%\psline[linecolor=black, linewidth=0.04, arrowsize=0.05291666666666667cm 2.0,arrowlength=1.4,arrowinset=0.0]{->}(6.436842,-0.25347915)(8.489473,-0.25347915)
%\end{pspicture}
%}
%\caption{An example of the quotient by $\ISO_p(G)$, for $L(\bar\lambda) = \{1,2\}$.}\label{fig:quotient}
%\end{figure}
This map indeed descends to a continuous map on the quotient. % (see Fig.~\ref{fig:quotient}):
\begin{equation}
\tilde\xi : S^{2N(\bar\lambda)-1}/\U(n_1)\times \cdots \times \U(n_\ell) \to S_{\geq 0}^{\ell-1}.
\end{equation}
It is bijective (recall that $u_j \in \R^{2n_j}$ and the action of $\U(n_j)$ on $\R^{2n_j}$ is the classical action of $\U(n_j)$ on $\mathbb{C}^{n_j}$, which is transitive on spheres with the same radius). 
%\red{The quotient space $S^{2N(\bar\lambda)-1}/U(n_1)\times \ldots U(n_\ell)$ is Hausdorff. In fact, consider equivalence classes $([u_1],\ldots,[u_\ell]) \neq ([u'_1],\ldots,[u'_\ell])$. Then there exists $\bar{j}$ such that $\|u_{\bar{j}}\| \neq \|u'_{\bar{j}}\|$. In particular, the whole orbits of $(u_1,\ldots,u_k)$ and $(u_1',\ldots,u_k')$ are separated by two open saturated sets say $U$ and $U'$ (namely $U$ and $U'$ are sets of whole orbits). Their projection is again open and disjoint by construction. Then the map $\tilde\xi$ is a continuous one-to-one map from a compact, Hausdorff space to a compact space, hence it is a homeomorphism.}
Being a continuous map from a compact space to a Hausdorff space, $\bar{\xi}$ is closed, then is open, thus it is an homemorphism.
\end{proof}
			% symmetries of contact Carnot groups
%!TEX root = enumerative-v5.tex

\section{Contact sub-Riemannian manifolds}\label{s:contact}

%In this section we define the nilpotent approximation of a contact sub-Riemannian structure at a point $p_0$, and we review some of its basic properties. Then we relate the local geodesic count on the original manifold with the geodesic count on the nilpotent structure.

\subsection{The nilpotent approximation}

Let $M$ be a contact sub-Riemannian manifold and let $p_0 \in M$. All our considerations being local, up to restriction to a coordinate neighbourhood $U$ of $p_0$, we assume that $M = \R^{2n+1}$ and that the sub-Riemannian structure ($\distr,\metr$) on $M$ is defined by a set $f_1,\ldots,f_{2n}$ of global orthonormal vector fields. Namely
\begin{equation}
\distr = \spn\{f_1,\ldots,f_{2n}\}, \qquad \text{and} \qquad \langle f_i |f_j \rangle = \delta_{ij}.
\end{equation}
The vector fields $f$ are assumed to be bounded with all derivatives as well. This will certainly be true if they are the coordinate representation of local orthonormal fields on a neighbourhood $U$ of $p_0$ of a larger sub-Riemannian manifold. %Due to this assumption, it is easy to verify that \emph{any} convergence below is uniform in all derivatives.
\begin{definition}\label{d:adapted}
Coordinates $(x,z) \in \R^{2n}\times \R$ are \emph{adapted} at $p_0$ if they are centred at $p_0$ and
\begin{equation}
\distr_{p_0} = \spn\left\lbrace \frac{\partial}{\partial x_1}, \ldots, \frac{\partial}{\partial x_{2n}}\right\rbrace.
\end{equation}
\end{definition}
\begin{example}\label{ex:darboux}
\emph{Darboux's coordinates} on a contact manifolds are local coordinates $(x,z) \in \R^{2n}\times \R$ such that the contact form has the following form:
\begin{equation}
\alpha = -dz + \frac{1}{2} \sum_{i,j=1}^{2n} J_{ij}x_i d x_j, \qquad\text{where}\qquad J = \begin{pmatrix}
0 & \mathbbm{1}_n \\ -\mathbbm{1}_n & 0
\end{pmatrix}.
\end{equation}
In particular, in these coordinates $d\alpha = \sum_{i < j} J_{ij} dx_i\wedge dx_j$. The classical Darboux's theorem states that Darboux's coordinates always exist in a neighbourhood of any point $p_0$. Since $\distr_{p_0} = \ker \alpha|_{p_0} = \spn\{\partial_{x_1},\ldots,\partial_{x_{2n}}\}$, Darboux's coordinates are indeed adapted at $p_0$.
\end{example}

%\begin{remark}\label{r:privileged}
%In the language of sub-Riemannian geometry, and strictly speaking of contact structures, adapted coordinates are called \emph{privileged}. In general more properties are required (see \cite{nostrolibro,Bellaiche}).
%\end{remark}

In these coordinates we define ``non-homogeneous dilations'' $\delta_\epsilon : M \to M$ by:
\begin{equation}
\delta_\epsilon(x,z) = (\epsilon x, \epsilon^2 z), \qquad \epsilon > 0,
\end{equation}
and the following family of vector fields:
\begin{equation}
f_i^\epsilon:= \epsilon \delta_{\frac{1}{\epsilon}*} f_i = \hat{f}_i + \epsilon W_i^\epsilon, \qquad \epsilon>0.
\end{equation}
The fields $f_i^\epsilon$ represent the ``blowup'' of the original structure in a neighbourhood of $p_0$ through the dilations $\delta_\epsilon$. The nilpotent approximation is the ``principal part'' of the original structure w.r.t. this non-homogeneous blowup.

\begin{definition} For all $\epsilon >0$, the \emph{$\epsilon$-blowup} is the sub-Riemannian structure $(M,f^\epsilon)$ on $M$ defined by declaring $f_1^\epsilon,\ldots,f_{2n}^\epsilon$ a set of global orthonormal fields.
Likewise, the \emph{nilpotent approximation} (at $p_0$) is the sub-Riemannian structure $(M,\hat{f})$ on $M$ defined by declaring $\hat{f}_1,\ldots,\hat{f}_{2n}$ a set of global orthonormal fields.
\end{definition}
We call $\distr^{\epsilon}$ (resp. $\hat{\distr}$) the distribution of the $\epsilon$-blowup (resp. of the nilpotent structure). %The next proposition describes the structure of the nilpotent approximation of a contact sub-Riemannian manifold.

\begin{proposition}\label{p:tangiscontcarn}
The nilpotent approximation $(M,\hat{f})$ at $p_0$ of a contact manifold is a contact Carnot group, with contact form given by
\begin{equation}
\hat{\alpha} = \lim_{\epsilon\to 0}\frac{1}{\epsilon^2}\delta_\epsilon^* \alpha.
\end{equation}
Let $f_0$ be a vector field transversal to $\distr$ (in the original structure), and let
\begin{equation}
\hat{f}_0:= \lim_{\epsilon \to 0} \epsilon^2 \delta_{\frac{1}{\epsilon}*} f_0.
\end{equation}
Then the Lie algebra $\mathfrak{g} = \mathfrak{g}_1\oplus \mathfrak{g}_2$ of the contact Carnot group $G=(M,\hat{f})$ is
\begin{equation}
\mathfrak{g}_1= \spn\{\hat{f}_1,\ldots,\hat{f}_{2n}\}, \qquad \mathfrak{g}_2 = \spn \{\hat{f}_0\},
\end{equation}
with structural constants given by $A \in \so(2n)$ such that:
\begin{equation}
[\hat{f}_i,\hat{f}_j] = A_{ij} \hat{f}_0, \qquad A_{ij} = -\left.\frac{d\alpha(f_i,f_j)}{\alpha(f_0)}\right|_{p_0}.
\end{equation}
\end{proposition}
\begin{proof}
We first prove that the nilpotent structure is contact. For $\epsilon>0$ let $\alpha^\epsilon:= \tfrac{1}{\epsilon^2}\delta_{\epsilon}^* \alpha$. Indeed $\distr^\epsilon = \ker \alpha^\epsilon$. Let $(x,z)$ be the set of adapted coordinates that define the dilation $\delta_\epsilon$. Then
\begin{equation}
\alpha = \sum_{i=1}^{2n} \xi_i dx_i + w dz,
\end{equation}
for some smooth functions $\xi_i, w : \R^{2n+1} \to \R$, bounded with all their derivatives. Since $\distr_{p_0} = \ker\alpha|_{p_0} = \spn\{\partial_{x_1},\ldots,\partial_{x_{2n}}\}$ in adapted coordinates we have the following Taylor expansions
\begin{equation}\label{eq:Taylor}
\xi_i(x,z) = \sum_{j=1}^{2n} a_{ij} x_j + b z + R_i(x,z), \qquad w(x,z) = w_0 + R_0(x,z).
\end{equation}
where the remainder terms $R_i(x,z)$ (resp. $R_0(x,z)$) are actually bounded by polynomials of degree $\geq 2$ (resp. $\geq 1$) in $(x,z)$. Moreover $a_{ij}$ is non-degenerate since $d\alpha|_{\distr}$ is non-degenerate and $w_0 \neq 0$. A straightforward calculation using the definition of $\delta_\epsilon^*$ gives
\begin{equation}
\alpha^\epsilon = \sum_{i=1}^{2n} \frac{1}{\epsilon}\xi_i(\epsilon x, \epsilon^2 z) dx_i + w(\epsilon x,\epsilon^2 z) dz.
\end{equation}
In particular, using Eq.~\eqref{eq:Taylor}, we notice that $\alpha^\epsilon$ converges uniformly to $\hat{\alpha}$:
\begin{equation}
\hat{\alpha} = \lim_{\epsilon \to 0} \alpha^\epsilon = \sum_{i,j=1}^{2n} a_{ij} x_j dx_i + w_0 dz.
\end{equation}
Indeed $\alpha \wedge (d \alpha)^n = w_0\det(a) \neq 0$, which implies non-degeneracy of the contact form. Moreover, $\ker \hat{\alpha} = \spn\{\hat{f}_1,\ldots,\hat{f}_{2n}\}$. In fact, for all $i=1,\ldots,2n$, we have
\begin{equation}
\hat{\alpha}(\hat{f}_i) = \lim_{\epsilon\to 0} \frac{1}{\epsilon^2} \delta_\epsilon^* \alpha (\epsilon \delta_{\frac{1}{\epsilon}*} f_i) = \lim_{\epsilon\to 0} \frac{1}{\epsilon}\alpha(f_i) = 0.
\end{equation}

Now we show that the nilpotent approximation $(M,\hat{f})$ is a Carnot group. Consider the fields $f_1,\ldots,f_{2n}$ defining the original structure, and any field $f_0$ transversal to $\distr$. Then
\begin{equation}\label{eq:structuraloriginal}
[f_i,f_j] = \sum_{k=1}^{2n} c_{ij}^k f_k + c_{ij}^0 f_0 , \qquad \forall i,j=1,\ldots,2n,
\end{equation}
for some family of smooth functions $c_{ij}^0$ and $c_{ij}^k$. Now consider the blowup of Eq.~\eqref{eq:structuraloriginal}, namely we act on both sides with $\epsilon^2 \delta_{1/\epsilon*}$, and we take the limit for $\epsilon \to 0$. The first term on the r.h.s. vanishes in the limit (due to the factor $\epsilon^2$), and we obtain
\begin{equation}\label{eq:structuralnilpotent}
[\hat{f}_i,\hat{f}_j] = A_{ij} \hat{f}_0.
\end{equation}
where $A_{ij} := c_{ij}^0(p_0)$ is a constant skew-symmetric matrix. Analogously, one can check that 
\begin{equation}
[\hat{f}_i,\hat{f}_0] = [\hat{f}_0,\hat{f}_0] = 0, \qquad \forall i=1,\ldots,2n.
\end{equation}
Thus the fields $\hat{f}_1,\ldots,\hat{f}_{2n}$ and $\hat{f}_0$ define a graded, nilpotent Lie algebra $\mathfrak{g} = \mathfrak{g}_1 \oplus \mathfrak{g}_2$ with
\begin{equation}
\mathfrak{g}_1:= \spn\{\hat{f}_1,\ldots,\hat{f}_{2n}\}, \qquad \mathfrak{g}_2 = \spn \{\hat{f}_0\}.
\end{equation}
Since $M=\R^{2n+1}$ is simply connected and the Lie algebra of vector fields $\mathfrak{g}$ is nilpotent, there exists a unique group structure on $M$ such that $\mathfrak{g}$ is its Lie algebra of left-invariant vector fields. The definition of the product law can be written explicitly in exponential coordinates on $G$ induced by the fields $\hat{f}_1,\ldots,\hat{f}_{2n}$, $\hat{f}_0$ through the Backer-Campbell-Hausdorff formula and is left to the reader. Thus $G:=(M,\hat{f})$ has the structure of a contact Carnot group. Finally, % see \eqref{eq:cartan}:
\begin{equation}
d\alpha(f_i,f_j) = f_i(\alpha(f_j))  - f_j(\alpha(f_i)) - \alpha([f_i,f_j]) = - c_{ij}^0 \alpha(f_0).
\end{equation}
Using the relation $A_{ij} = c_{ij}^0(p_0)$, it is sufficient to evaluate the above formula at $p_0$ to obtain
\begin{equation}
A_{ij} = -\left.\frac{d\alpha(f_i,f_j)}{\alpha(f_0)}\right|_{p_0}.
\end{equation}
Indeed $A$ is not degenerate, as a consequence of the non-degeneracy assumption on $d\alpha|_{\distr}$.
\end{proof}
	
\subsection{Adapted vs exponential coordinates} \label{s:adaptedvsexp}

Recall that, at the beginning of this section we put adapted coordinates $(x,z) \in \R^{2n}\times \R$  on $M$. This choice defined the family of non-homogeneous dilations $\delta_\epsilon$ that, in turn defined the nilpotent approximation $(M,\hat{f})$ as the ``limit'' of the $\epsilon$-blowup structures. Any choice of a global orthonormal frame $f_i$ and $f_0$ transverse to $\distr$ for the original structure induces a global orthonormal frame $\hat{f}_i$ and $\hat{f}_0$ (transverse to $\hat{\distr}$) for the nilpotent approximation, where
\begin{equation}
\hat{f}_i = \lim_{\epsilon\to 0} \epsilon \delta_{\frac{1}{\epsilon}*} f_i, \qquad \hat{f}_0 = \lim_{\epsilon\to 0} \epsilon^2 \delta_{\frac{1}{\epsilon}*} f_0.
\end{equation}
Since $G=(M,\hat{f})$ is a contact Carnot group, the fields $\hat{f}_1,\ldots,\hat{f}_{2n}$ and $\hat{f}_0$ induce exponential coordinates $(\theta,\rho) \in \R^{2n}\times\R$. Namely a point has coordinates $(\theta,\rho)$ if and only if
\begin{equation}
(x,z) = \mathrm{exp}_G\left(\sum_{i=1}^{2n} \theta_i\hat{f}_i + \rho \hat{f}_0\right).
\end{equation}
%See Fig.~\ref{fig:implications} for a diagram of logical implications. 
The next lemma clarifies the relation between adapted coordinates $(x,z)$ and exponential coordinates $(\theta,\rho)$ on the same base space $M=\R^{2n+1}$.

\begin{figure}
\centering
\psscalebox{0.9} % Change this value to rescale the drawing.
{
\begin{pspicture}(0,-1.4)(16,0.8)
\footnotesize
\rput[l](0,0){\rnode{A}{\psframebox{\begin{minipage}[c]{9em}\center  adapted coordinates\\  $(x,z) \in \R^{2n}\times \R$ \end{minipage}}}}
\rput[l](4,0){\rnode{B}{\psframebox{\begin{minipage}[c]{9em}\center  dilations $\delta_\epsilon$ \\  $\delta_\epsilon(x,z) =(\epsilon x,\epsilon^2z)$ \end{minipage}}}}
\rput[l](8,0){\rnode{C}{\psframebox{\begin{minipage}[c]{11em}\center  nilpotent approximation \\  $\hat{f} = \displaystyle\lim_{\epsilon\to 0} \epsilon \delta_{\frac{1}{\epsilon}*} f$ \end{minipage}}}}
\rput[r](16,0){\rnode{D}{\psframebox{\begin{minipage}[c]{11em}\center  exponential coordinates\\  $(\theta,\rho) \in \R^{2n}\times \R$ \end{minipage}}}}
\ncline[nodesep=2pt]{->}{A}{B}
\ncline[nodesep=2pt]{->}{B}{C}
\ncline[nodesep=2pt]{->}{C}{D}
\ncbar[nodesep=2pt,angle=-90]{->}{D}{A}
\naput{change of coordinates $(x,z) = (B\theta, \theta^*S\theta + c\rho)$}
\end{pspicture}
}
\caption{Adapted coordinates on $M$ define the dilation map $\delta_\epsilon$ that, in turn, defines the nilpotent approximation $(M,\hat{f})$.}\label{fig:implications}
\end{figure}
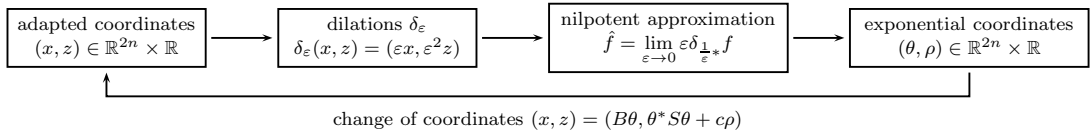

\begin{lemma}\label{l:linearchange}
Let $(x,z) \in \R^{2n}\times \R$ be adapted coordinates for the contact structure $(\R^{2n+1},f)$, and let $(\theta,\rho) \in \R^{2n} \times \R$ be exponential coordinates for the Carnot structure $(\R^{2n+1}, \hat{f})$, induced by some choice of $f_i,f_0$ (and consequently $\hat{f}_i,\hat{f}_0$). Then the two sets of coordinates are related by
\begin{equation}
x = B\theta, \qquad z = \theta^* S\theta + c \rho,
\end{equation}
where $B \in \mathrm{GL}(2n)$, $c \in \R\setminus\{0\}$ and $S \in \mathrm{Mat}(2n)$.
\end{lemma}
\begin{proof}
For $i=1,\ldots,2n$ we have, in adapted coordinates:
\begin{equation}
f_i = \sum_{j=1}^{2n} B_{ji}(x,z)\frac{\partial}{\partial x_j} + b_i(x,z) \frac{\partial}{\partial z}, \qquad
f_0 = \sum_{j=1}^{2n} C_{j}(x,z)\frac{\partial}{\partial x_j} + c(x,z) \frac{\partial}{\partial z},
\end{equation}
for some smooth functions $B_{ij},b_i,C_{j},c:\R^{2n+1} \to \R$ that satisfy:
\begin{equation}\label{eq:properties}
b_i(0,0) = 0, \qquad \det B_{ij}(0,0) \neq 0, \qquad c(0,0) \neq 0.
\end{equation}
By explicit computation we obtain
\begin{equation}
\hat{f}_i = \sum_{j=1}^{2n} B_{ji}(0,0) \frac{\partial}{\partial x_j} + \sum_{j=1}^{2n}\frac{\partial b_i}{\partial x_j}(0,0) x_j\frac{\partial}{\partial z}, \qquad \hat{f}_0 = c(0,0) \frac{\partial}{\partial z}.
\end{equation}
By definition of exponential coordinates (see the proof of Lemma~\ref{l:expcoords}) we obtain that 
\begin{equation}
x = B \theta, \qquad \text{and} \qquad z = \theta^* S\theta + c \rho,
\end{equation}
where $B$ is the matrix with components $B_{ij}(0,0)$, $c = c(0,0)$ and the matrix $S$ has components
\begin{equation}
S_{ij} = \frac{1}{2}\sum_{\ell=1}^{2n}\frac{\partial b_i}{\partial x_\ell}(0,0) B_{\ell j}(0,0), \qquad B \in \mathrm{GL}(2n)\qquad \text{by \eqref{eq:properties}}. \qedhere
\end{equation}
%Indeed $c \neq 0$ and $B \in \mathrm{GL}(2n)$ by \eqref{eq:properties}.
\end{proof}

The following proposition compares the geometry of the original structure with the $\epsilon$-blowup and is left to the reader.

\begin{proposition}\label{p:blowup}
The composition $\gamma \mapsto \gamma_\epsilon=\delta_{\frac{1}{\epsilon}}\gamma$ gives a homeomorphism between the set of admissible curves for $(M, f)$ and admissible curves for $(M, f^\epsilon).$ If $\gamma(0)=0, \gamma(1)=p$ and $\gamma$ is a geodesic for $(M, f)$, then $\gamma_\epsilon$ is a geodesic for $(M, f^\epsilon)$ with $\gamma_\epsilon(0)=0,\gamma_\epsilon(1)=\delta_\frac{1}{\epsilon}(p)$; the energies of these curves are related by $J_\epsilon(\gamma_\epsilon)=\epsilon^{-2}J(\gamma)$.
\end{proposition}
%\begin{proof}
%Since $\gamma:[0,1]\to M$ is admissible, there exists $u_1, \ldots, u_{2n} \in L^1(I)$ such that $\dot{\gamma}(t)=\sum u_i(t)f_i(\gamma(t))$ a.e. and:
%\begin{equation}
%\dot{\gamma}_{\epsilon}(t)=\sum_{i=1}^{2n} u_i(t) \delta_{\frac{1}{\epsilon}*}f_i (\gamma_\epsilon(t)) 
%=\sum_{i=1}^{2n} \frac{u_i(t)}{\epsilon} f^\epsilon_i (\gamma_\epsilon(t))=\sum_{i=1}^{2n} u^{\epsilon}_i(t)f^{\epsilon}_i(\gamma_\epsilon(t)),
%\end{equation}
%where we have set $u_i^{\epsilon}=\epsilon^{-1}u_i.$ In particular $\gamma_\epsilon$ is admissible, and its energy is obtained dividing by $\epsilon^2$ the energy of $\gamma$. Thus, a critical point for $J$, i.e. a geodesics for $(M,f)$, is mapped to a critical point for $J_\epsilon=\epsilon^{-2}J$, and all critical points arise in this way, as $\gamma\mapsto \gamma_\epsilon$ is invertible.
%\end{proof}

\subsection{Semicontinuity of the counting function}
Let $E_\epsilon, \hat{E}:T^*_{p_0} M\to M$ be, respectively, the sub-Riemannian exponential maps for $(M, f^{\epsilon})$ and $(M, \hat{f})$. We define now the \emph{counting functions} $\nu_{\epsilon},\hat{\nu}:M\to (0, \infty]$ as:
\begin{equation}
\nu_\epsilon(p)=\# E_\epsilon^{-1}(p)\qquad \textrm{and}\qquad  \hat{\nu}(p)=\#\hat{E}^{-1}(p).
\end{equation}
In other words, $\nu_\epsilon(p)$ counts the number of geodesics between $0$ and $p$ for the $\epsilon$-blowup and $\hat{\nu}(p)$ for the limit Carnot group. Setting $\nu=\nu_1$ (the counting function for the original structure $(M, f)$), we notice that Prop. \ref{p:blowup} implies indeed:
\begin{equation}
\nu_\epsilon(p)=\nu(\delta_\epsilon(p)).
\end{equation}
In fact given a geodesic $\gamma:I\to M$  for $(M, f)$ between $0$ and $\delta_\epsilon(p)$, then $\delta_{\frac{1}{\epsilon}}\gamma$ is a geodesic for $(M, f^\epsilon)$ with final point $\delta_{\frac{1}{\epsilon}}(\gamma(1))=\delta_{\frac{1}{\epsilon}}(\delta_\epsilon(p))=p$ (and vice-versa). The next theorem compares the asymptotics of $\nu(\delta_{\epsilon}(p))$ with the one of $\hat{\nu}(p)$.
\begin{theorem}[Counting in the limit]\label{t:limit}
Let $M$ be a contact sub-Rieman\-nian manifold. For the generic $p \in M$ sufficiently close to $p_0$:
\begin{equation}
\hat{\nu}(p)\leq \liminf_{\epsilon \to 0}\nu(\delta_\epsilon(p)).
\end{equation}
where $\delta_\epsilon$ is the non-homogeneous dilation defined in some set of adapted coordinates in a neighbourhood of $p_0$.
\end{theorem}
\begin{proof}
If $p$ is a regular value of $\hat{E}$, then the fiber if $\hat{E}^{-1}(p)$ is discrete, hence $\hat{\nu}(p)$ is finite by Thm. \ref{thm:disjoint}.
Consider an open bounded set $U\subset T_0^*M$, where bounded means that it is contained in a compact set $K$, such that:
\begin{equation}
\hat{E}^{-1}(p)\subset U\subset K.
\end{equation}
We claim that there exists $\epsilon_K>0$ such that $p$ is a regular value of $E_\epsilon|_{U}$ for every $\epsilon<\epsilon_K.$ If this was not true, then we can find a sequence $\{\epsilon_n\}_{n\in \mathbb{N}}$ converging to zero and a sequence $\{\lambda_n\}_{n\in \mathbb{N}}\subset K$ such that $E_{\epsilon_n}(\lambda_n)=p$ and $\rank (d_{\lambda_n}E_{\epsilon_n})<\dim (M)$. Then, by compactness of $K$, up to subsequences we can assume $\lambda_n\to \hat{\lambda}$ with
$\hat{E}(\hat{\lambda})=p$, by uniform convergence of $E_{\epsilon_n}|_K$ to $\hat{E}|_K$ with all derivatives (see \cite[Prop. 5.15]{ABR}). Moreover, by the same argument, $d_{\lambda_n}E_{\epsilon_n}\to d_{\hat{\lambda}}\hat{E} $ and since the set of points where the rank of $d\hat{E}$ is not maximal is closed, we also have $\rank (d_{\hat{\lambda}}\hat{E})<\dim (M)$, which contradicts the fact that $p$ was a regular value of $\hat{E}.$

Consider now the function $\bar{E}:\overline{U}\to M $ (where $\overline{U}=U\times[0, \epsilon_K]$) given by $(u, \epsilon)\mapsto E_\epsilon(u)$ (where we have set $E_0=\hat{E}$); the uniform convergence of $E_\epsilon$ with all derivatives on compact sets implies $\bar{E}$ is smooth (in fact $C^1$ is enough for us). 
By the above observation $\bar{X}=\bar{E}^{-1}(p)$ is a smooth submanifold of $\overline{U}$ and its dimension is one.  In fact: 
\begin{equation}
(d_{(u, \epsilon)}\bar{E})(\dot{u}, \dot{\epsilon})=(d_{u}E_{\epsilon})\dot{u}+\frac{\partial \bar{E}}{\partial \epsilon}(u, \epsilon)\dot{\epsilon}, \qquad (\dot u, \dot\epsilon) \in T_{(u,\epsilon)} \overline{U}.
\end{equation}
Since $p$ is a regular value of $E_\epsilon$ for all $\epsilon \in [0, \epsilon_K]$, the image of $d_{u}E_{\epsilon}$ is enough to generate $T_pM$.

On the other hand, we claim that \emph{zero} is a regular value for the the projection $\pi:\bar{X}\to [0, \epsilon_K]$ on the second factor. To prove this, observe that tangent space to $\bar{X}$ at $(u, 0)$ is:
\begin{equation}
T_{(u, 0)}\bar{X}=\left\{(\dot{u}, \dot{\epsilon})\,\big|\, (d_{u}\hat{E})\dot{u}+\frac{\partial \bar{E}}{\partial \epsilon}(u, 0)\dot{\epsilon}=0\right\},
\end{equation}
and since $\hat{E}$ a submersion at $p$:
\begin{equation}
T_{(u, 0)}\bar{X}\cap \ker{d\pi}\simeq T_u\hat{E}^{-1}(p)=\{0\}.
\end{equation}
Thus $T_{(u, 0)}\bar{X}$ must contain some vector $(\dot{u}, \dot{\epsilon})$ with $\dot\epsilon\neq 0$, i.e. zero is not critical for $\pi$, proving the claim. Then $\epsilon'>0$ small enough also is noncritical for $\pi$; in particular, by Ehresmann's theorem, $\pi|_{\pi^{-1}[0, \epsilon']}$ is a fibration ($U$ is contained in a compact set) and:
\begin{equation}
\forall \epsilon<\epsilon':\quad E_\epsilon|_U^{-1}(p)\simeq \hat{E}|_U^{-1}(p).
\end{equation}
Since $\nu_\epsilon(p)\geq \#  E_\epsilon|_U^{-1}(p)$ the conclusion follows (see Fig. \ref{fig:geo}).
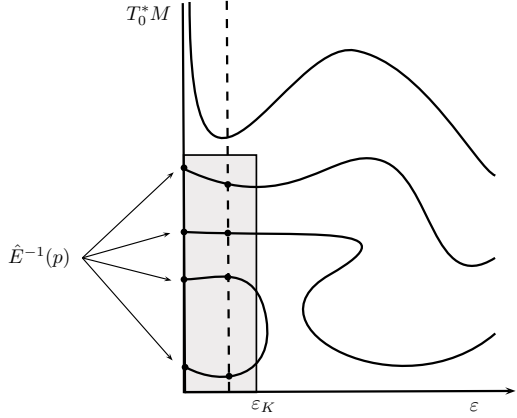
\begin{figure}% Generated with LaTeXDraw 2.0.8
\scalebox{0.75} % Change this value to rescale the drawing.
{
\begin{pspicture}(0,-3.7288477)(9.82291,3.7088478)
\definecolor{color689b}{rgb}{0.9333333333333333,0.9254901960784314,0.9254901960784314}
\psframe[linewidth=0.03,dimen=outer,fillstyle=solid,fillcolor=color689b](5.1810155,0.96884763)(3.8610156,-3.2711523)
\psline[linewidth=0.04cm](3.9010155,-3.2711523)(3.8610156,3.6488476)
\psline[linewidth=0.04cm,arrowsize=0.05291667cm 2.0,arrowlength=1.4,arrowinset=0.4]{->}(3.9010155,-3.2511523)(9.761016,-3.2311523)
\psbezier[linewidth=0.04](3.9010155,-1.2511524)(4.4210157,-1.2511524)(5.1793227,-0.92762464)(5.341016,-1.9511523)(5.5027084,-2.97468)(4.590543,-3.179227)(3.9155352,-2.8111525)
\psbezier[linewidth=0.04](3.8810155,-0.41115233)(6.0410156,-0.5111523)(8.061016,-0.25115234)(6.4610157,-1.1711524)(4.861016,-2.0911524)(7.7410154,-3.6311524)(9.401015,-2.2111523)
\psbezier[linewidth=0.04](3.8810155,0.70884764)(5.818831,-0.21115234)(6.519106,1.1660937)(7.402622,0.8488476)(8.286139,0.5316016)(8.281015,-1.5711523)(9.396336,-0.87115234)
\psbezier[linewidth=0.04](3.9810157,3.6688476)(3.9810157,-1.3311523)(5.7088614,3.0217755)(6.9101825,2.8088477)(8.111504,2.5959198)(9.078667,0.70884764)(9.401015,0.58884764)
\psdots[dotsize=0.12](3.8810155,0.7288477)
\psdots[dotsize=0.12](3.8810155,-0.41115233)
\psdots[dotsize=0.12](3.8810155,-1.2511524)
\psdots[dotsize=0.12](3.9010155,-2.8111525)
\psline[linewidth=0.04cm,linestyle=dashed,dash=0.16cm 0.16cm](4.6810155,-3.2311523)(4.6410155,3.6888475)
\psdots[dotsize=0.12](4.6610155,0.42884767)
\psdots[dotsize=0.12](4.6610155,-0.43115234)
\psdots[dotsize=0.12](4.6810155,-2.9711523)
\usefont{T1}{ptm}{m}{n}
\rput(9.032471,-3.5061524){$\epsilon$}
\usefont{T1}{ptm}{m}{n}
\rput(5.292471,-3.5061524){$\epsilon_K$}
\usefont{T1}{ptm}{m}{n}
\rput(3.2624707,3.4338477){$T_0^*M$}
\psline[linewidth=0.02cm,arrowsize=0.05291667cm 2.0,arrowlength=1.4,arrowinset=0.4]{<-}(3.7010157,0.64884764)(2.0810156,-0.87115234)
\psline[linewidth=0.02cm,arrowsize=0.05291667cm 2.0,arrowlength=1.4,arrowinset=0.4]{<-}(3.6410155,-0.43115234)(2.0810156,-0.87115234)
\psline[linewidth=0.02cm,arrowsize=0.05291667cm 2.0,arrowlength=1.4,arrowinset=0.4]{<-}(3.6610155,-1.2711524)(2.0810156,-0.87115234)
\psline[linewidth=0.02cm,arrowsize=0.05291667cm 2.0,arrowlength=1.4,arrowinset=0.4]{<-}(3.6810157,-2.6911523)(2.0810156,-0.87115234)
\usefont{T1}{ptm}{m}{n}
\rput(1.3224707,-0.84615237){$\hat{E}^{-1}(p)$}
\psdots[dotsize=0.12](4.6610155,-1.2111523)
\end{pspicture} 
}
\caption{Picture of $\bigcup_{\epsilon\in I} E_\epsilon^{-1}(p)\subset I\times T_0^*M$. Even if $\epsilon_K$ is small, some geodesics can still ``escape'' out of $K$. The shaded region denotes $[0,\epsilon_K]\times K$.}
\label{fig:geo}
\end{figure}
\end{proof}

%Consider now the case $\hat{\nu}(p)=\infty.$ Being $p$ a regular value of $\hat{E}$, then $\hat{E}^{-1}(p)$ is discrete and for every $m\in \mathbb{N}$ we can find an open set $U_m$ with compact closure $K_m\subset T_0^*\R^n$ containing exactly $m$ geodesics. Repeating the above proof for $U=U_m$, we find the existence of a number $\epsilon'_m$ such that:
%$$\forall \epsilon<\epsilon'_m:\quad E_\epsilon|_{U_m}^{-1}(p)\simeq \hat{E}|_{U_m}^{-1}(p).$$
%In particular:
%$$m=\#  \hat{E}|_{U_m}^{-1}(p)=\#  E_\epsilon|_{U_m}^{-1}(p)\leq \nu_\epsilon(p),$$
%which tells $\nu_\epsilon(p)\geq m$ for every $m>0$ and $\epsilon>0$ sufficiently small, i.e. $\lim_\epsilon \nu_\epsilon(p)=\infty.$ \end{proof}

%\begin{remark}
%The definition of $\epsilon$-blowup and nilpotent approximation remains unchanged for a general sub-Riemannian manifold provided that i) the point $p_0$ is \emph{equiregular} (this is an assumption on the dimension of higher order distributions, true for the generic choice of the point $p_0$) and ii) we choose \emph{privileged coordinates}. Since we are mainly interested in contact structures, we do not dwell into further details. Prop.s~\ref{p:blowup}-\ref{p:convergence} and~\ref{t:limit} remain true with no change whatsoever in their proof in this more general setting.
%\end{remark}
%
%Everything is now ready for the next theorem.

\begin{theorem}\label{t:darboux}
Let $M$ be a contact manifold and $(x,z)$ be Darboux's coordinates on a neighbourhood $U$ of $q \in M$. There exist constants $C(q), R(q)$ such that, for the generic $p=(x,z)\in U$:
\begin{equation}
\liminf_{\epsilon \to 0}\nu(\delta_\epsilon(p)) \geq C(q)\frac{|z|\phantom{^2}}{\|x\|^2} + R(q).
\end{equation}
\end{theorem}
\begin{proof}
We consider on $U$ the original structure $(U,f)$ and the nilpotent structure $(U,\hat{f})$ defined in adapted (e.g. Darboux's) coordinates (see Fig.~\ref{fig:implications}). The classical Sard theorem implies that the generic $p \in U$ is a regular value for $\hat{E}:T_{q}^*U \to U$. Then, by Thm.~\ref{t:limit},
\begin{equation}
\liminf_{\epsilon \to 0}\nu(\delta_\epsilon(p)) \geq \hat{\nu}(p).
\end{equation}
Now choose some orthogonal local frame $f_1,\ldots,f_{2n}$ and $f_0$ transversal to $\distr$ for the original structure. This induces exponential coordinates $(\theta,\rho)$ on $U$ (see Sec.~\ref{s:adaptedvsexp}). By Prop.~\ref{p:tangiscontcarn}, the nilpotent structure $(U,\hat{f})$ is a contact Carnot group such that
\begin{equation}
[\hat{f}_i,\hat{f}_j] = A_{ij} \hat{f}_0, \qquad A_{ij} = \left.\frac{d\alpha(f_j,f_i)}{\alpha(f_0)}\right|_{q}.
\end{equation}
The generic point $p$ has exponential coordinates $(\theta,\rho)$ with all $\theta_j\neq 0$. Then, by Thm.~\ref{thm:lower} we have
\begin{equation}
\hat{\nu}(p) \geq C_1\frac{|\rho|\phantom{^2}}{\|\theta\|^2}+R_1,
\end{equation}
where $C_1=C_1(q)$ and $R_1=R_1(q)$ are computed in the proof of Thm.~\ref{thm:lower} in terms of the singular values of $A$. Indeed they depend on the point $q$ at which we consider the nilpotentization. Darboux's (adapted) coordinates $(x,z)$ and exponential coordinates $(\theta,\rho)$ are related by the transformation of Lemma~\ref{l:linearchange} and we obtain the result.
\end{proof}

%We almost immediately obtain the following corollary.

\begin{theorem}\label{t:local}
Let $M$ be a contact sub-Riemannian manifold and $q\in M$. Then there exists a sequence $\{q_m\}_{m\in \N}$ in $M$ such that:
\begin{equation}
\lim_{m\to \infty}q_m= q\qquad \text{and}\qquad \lim_{m\to \infty}\nu(q_m)=\infty.
\end{equation}
\end{theorem}
\begin{proof}
In Darboux's coordinates in a neighbourhood $U$ of $q$, for every $m\in \mathbb{N}$ pick a point $p_m=(x_m,z_m)$ such that: 1) $p_m$ is a regular value of $\hat{E}$ and 2) $\frac{|z_m|\phantom{^2}}{\|x_m\|^2}\geq m$. The existence of such $p_m$ is guaranteed by Sard's Lemma. Consider now $\delta_\epsilon(p_m)$. 

If $p_m$ is regular value for $\hat{E}$, then $\hat\nu(p_m)$ is finite. Hence one can choose a fixed $U_m$ in the proof of Thm. 45 containing all geodesics arriving at $p_m$, and thus there exists $\varepsilon_m$ such that
\begin{equation}
\# \hat{E}^{-1}(p_m) = \# E_\varepsilon|_{U_m}^{-1} \leq \nu(\delta_\varepsilon(p_m)), \qquad \forall \varepsilon \leq \varepsilon_m.
\end{equation}
Notice that we can assume $\lim_{m \to +\infty} \epsilon_m=0$. Setting $q_m = \delta_{\varepsilon_m}(p_m)$ yields the statement.
\end{proof}

				% general facts in contact structures

\bibliographystyle{abbrv}
\bibliography{biblio-enumerative}

\end{document}